\documentclass{amsart}
\usepackage{amsxtra}
\usepackage{mathrsfs}
\usepackage{verbatim}
\usepackage{calc}
\usepackage{hyperref}
\usepackage[shortalphabetic,initials]{amsrefs}
\usepackage{bdsstandard}

\theoremstyle{plain}
\newtheorem{thm}{Theorem}[subsection]
\newtheorem*{thm*}{Theorem}
\newtheorem{thms}{Theorem}[section]
\newtheorem{prop}[thm]{Proposition}
\newtheorem{cor}[thm]{Corollary}
\newtheorem{lem}[thm]{Lemma}
\newtheorem*{lem*}{Lemma}
\newtheorem{prop*}{Proposition}

\theoremstyle{definition}
\newtheorem{defn}[thm]{Definition}
\newtheorem*{defn*}{Definition}
\newtheorem{eg}[thm]{Example}
\newtheorem{egs}[thms]{Example}

\newtheorem*{ex*}{Exercise}
\newtheorem{rk}[thm]{Remark}
\newtheorem{rks}[thms]{Remark}
\newtheorem{ntn}[thm]{Notation}

\newcommand{\FLV}{\ensuremath{\F\L\V}\xspace}
\newcommand{\FLG}{\ensuremath{\F\L\G}\xspace}
\newcommand{\flg}{\ensuremath{(\mathrm{FLG})}\xspace}
\newcommand{\ngerms}[1][n]{\ensuremath{(#1\text{-}\mathrm{germs})}\xspace}
\newcommand{\nbuds}[1][n]{\ensuremath{(#1\text{-}\mathrm{buds})}\xspace}
\newcommand{\Gn}[1][n]{\ensuremath{\G_{#1}}\xspace}
\newcommand{\Bn}[1][n]{\ensuremath{\B_{#1}}\xspace}
\newcommand{\nTh}[1][n]{\ensuremath{(#1\text{-}\mathrm{Th})}\xspace}
\newcommand{\nInf}[1][n]{\ensuremath{(#1\text{-}\mathrm{Inf})}\xspace}
\newcommand{\Isom}{\ensuremath{\operatorname{Isom}}\xspace}

\newcommand{\sEnd}{\ensuremath{\operatorname{\E\!\mathit{nd}}}\xspace}
\newcommand{\sAut}{\ensuremath{\operatorname{\A\!\mathit{ut}}}\xspace}

\newcommand{\Autstr}{\ensuremath{\operatorname{\A\!\mathit{ut}_{\textrm{str}}}}
\xspace}

\newcommand{\laz}{\ensuremath{L}\xspace}

\newcommand{\Vect}[1][]{\ensuremath{\mathrm{(Vect}{}_{#1}\mathrm{)}}\xspace}
\newcommand{\can}{\mathrm{can}}
\newcommand{\tr}{\mathrm{tr}}
\newcommand{\fet}{\textnormal{f\'et}}

\newcommand{\fpqc}{\textnormal{fpqc}}
\newcommand{\nlim}[1][1]{\ensuremath{#1\text{-}\varprojlim}}

\DeclareMathOperator{\Ab}{Ab}

\DeclareMathOperator{\Fib}{Fib}
\DeclareMathOperator{\CFG}{CFG}
\DeclareMathOperator{\Gp}{Gp}
\DeclareMathOperator{\Inf}{Inf}
\DeclareMathOperator{\St}{St}

\renewcommand{\L}{\ensuremath{\mathscr{L}}\xspace}

\begin{document}

\renewcommand{\O}{\ensuremath{\mathscr{O}}\xspace}

\title{On the moduli stack of commutative, $1$-parameter formal Lie groups}
\author{Brian D. Smithling}
\address{Max-Planck-Institut f\"ur Mathematik, Vivatsgasse 7, 53111 Bonn,
Germany}
\email{bsmithli@mpim-bonn.mpg.de}

\begin{abstract}
We commence a general algebro-geometric study of the moduli stack of
commutative, $1$-parameter formal Lie groups.  We emphasize the pro-algebraic
structure of this stack: it is the inverse limit, over varying $n$, of moduli
stacks of $n$-buds, and these latter stacks are algebraic.  Our main results
pertain to various aspects of the height stratification relative to fixed prime 
$p$ on the stacks of buds and formal Lie groups.
\end{abstract}
\maketitle

\section*{Introduction}\label{s:intro}

The aim of this paper is to explicate some of the basic algebraic geometry of
the moduli stack of commutative, $1$-parameter formal Lie groups.  Formal group
laws have received extensive study in the mathematical literature for well over
50 years.  Their applications cut across a wide swath of mathematics;
Hazewinkel's massive treatise \cite{haz78} of almost 30 years ago already
counted over 500 citations.
%
We shall make no attempt to survey the vast literature
produced since then.  But we do wish to signal one area of mathematics in which
the moduli stack of formal groups has appeared in an explicit and important way:
namely, \emph{stable homotopy theory}.

Formal group laws enter into topology in the study of generalized cohomology
theories: to any even, periodic cohomology theory on the category of topological
spaces, one may associate a formal group law in a way canonical up to choice of
coordinate; see, for example, Lurie's survey article \cite{lur06}.  In 
\cite{qui69}, Quillen observed that the formal group law associated to complex
cobordism $MU$ is \emph{universal}.  In fact, the affine schemes $\Spec MU_*$ 
and $\Spec MU_*MU$ admit a natural internal groupoid structure
\[\tag{$*$}\label{disp:MU.gpd}
   \Spec MU_*MU \rra \Spec MU_*
\] 
which represents the functor
\[
   \vcenter{
   \xymatrix@R=0ex{
      (\text{Affine schemes})^\opp \ar[r]
         & \text{Groupoids}\\
      \Spec R \ar@{|->}[r]
         & \biggl\{\parbox{28ex}{\centering formal group laws and strict isomorphisms over $R$}\biggr\}.
   }
   }
\]
Thus the stack of formal groups arises \emph{naturally} in topology.  

In the 1970's, Morava began to emphasize the role of group laws of
given \emph{height}, and in particular the \emph{Morava stabilizer groups}, in
stable homotopy theory, and he advocated for the importation of algebraic 
geometry into the subject as a means to gain conceptual insight. The notion of 
height plays a fundamental role in homotopy theory's so-called \emph{chromatic 
picture}; see Ravenel's important reference \cite{rav86} for details.  More
recently, owing much to the influence and insight of Hopkins, topologists have 
come to understand that a great deal of the chromatic picture's impressive 
computational architecture is, in some sense, governed by the structure of the 
moduli stack of formal Lie groups; see the introduction to \cite{ghmr05} for
some discussion.


For example, in \cite{hop99}*{21.4}, Hopkins interprets the Landweber exact 
functor theorem as a statement closely related to flatness of quasi-coherent 
sheaves on the stack of formal groups. In \cite{prib04}*{4.8}, Pribble obtains a 
purely stack-theoretic analog of the chromatic convergence theorem of 
Hopkins-Ravenel \cite{rav92}*{\s8.6}. In \cite{nau07}*{26}, Naumann gives a 
satisfying algebro-geometric explanation of Hovey's and Strickland's result 
\cite{hovstr05}*{4.2} that for any two Landweber-exact $BP_*$-algebras $R$ and 
$S$ of the same height%
\footnote{This notion of height is closely related to, but in some sense is not
completely compatible with, the notion of height we use when speaking of
the \emph{height stratification}; see \eqref{rk:p-typ}.}%
, the comodule categories of the Hopf algebroids 
\[
   (R, R\tensor_{BP_*}BP_*BP \tensor_{BP_*} R) 
   \quad\text{and}\quad 
   (S, S\tensor_{BP_*}BP_*BP \tensor_{BP_*} S)
\]
are equivalent: namely, he shows that the underlying \emph{stacks} are 
equivalent.  Many important change of rings theorems in stable homotopy theory 
can be seen as arising in this way.

Despite the apparent importance and utility of the moduli stack of formal Lie 
groups in homotopy theory (to say nothing of whatever applications it may have 
in  other branches of mathematics), surprisingly little foundational material on 
this stack has yet appeared in the mathematical literature.  Our intent in this paper 
is  to take some of the first steps towards filling this gap.  That said, let us 
hasten to add that many things we shall discuss have already been treated 
elsewhere in one form or another. Notably, \cite{nau07}*{\s6} contains a quick 
account of some of the basic algebro-geometric moduli theory of formal groups: 
in particular, Naumann gives an algebro-geometric definition of formal groups, 
shows that they form an fpqc-stack, and gives an intrinsic description of the 
stack associated to \eqref{disp:MU.gpd} as the stack of formal groups with 
trivialized canonical bundle.  The \emph{height stratification} is defined and 
plays a prominent role in \cite{prib04}*{\s4.4}.  Hopkins has covered a
considerable amount of moduli theory in \cite{hop99} and in other courses at
MIT.  And there have been many papers --- let us signal especially \cite{laz55}
--- devoted to classification of formal group \emph{laws}; it is this work, of 
course, on which ours shall ultimately rest.

The essential feature of this paper is, perhaps then, its \emph{scope}.  We have
tried to develop the moduli theory largely from the ground up, beginning from 
the foundations of the classical formal group law literature.  Although the 
topologists have emerged as the primary consumers of the theory, let us now 
extend our apologies to them:  we have aproached the subject as a pure piece of 
algebraic geometry, with essentially no regard for considerations arising in 
topology aside from a few remarks in \eqref{rk:p-typ}.  For a much more 
comprehensive account of the moduli stack of formal Lie groups and its relation 
to stable homotopy theory, we refer to the forthcoming \cite{goerss?}.


Let us now discuss the main contents of this paper.  

Section \ref{s:defs} serves chiefly to collect terminology surrounding the various objects
at play.  Following, for example, \cites{gro74, mes72}, we shall 
refer to the algebro-geometric version of a formal group law as a ``formal Lie 
group'', as opposed to just ``formal group''.  Though the ``Lie'' is now less 
commonly used than in earlier treatments of the subject, we feel that ``formal 
group'' should be reserved for any group object arising in a formal-geometric 
context.  Our choice of terminology (re)emphasizes that the objects we study are
the algebro-geometric analog of an infinitesimal neighborhood of the identity 
of a ($1$-dimensional) Lie group; see \eqref{def:fv} and \eqref{def:fg}.

In \s\ref{s:moduli_theory} we address the first properties of the stack of formal Lie groups \FLG
and of the related stacks we consider.  Unfortunately, \FLG is not
\emph{algebraic} in the sense traditional in algebraic geometry
\cite{lamb00}*{4.1}:  for example, its diagonal is not of finite type
\cite{lamb00}*{4.2}.  One remedy for this defect, due to Hopkins and followed 
in \citelist{\cite{goerss04}*{1.8}\cite{prib04}*{3.15}\cite{nau07}*{6}}, is the
following.  Let $\Lambda$ denote the Lazard ring.  Then the universal formal 
group law specifies a map 
\[\tag{$**$}\label{disp:Laz_map}
   \Spec \Lambda \to \FLG
\]
which is at least surjective, flat, and affine (it is not an honest presentation because 
it is not of finite type).   So one simply \emph{redefines} the notion of 
algebraic stack to mean an ``affine-ized'' version of the usual one, using flat 
covers, so that \eqref{disp:Laz_map} is a presentation and \FLG is algebraic. 
See \cite{hol?} for an axiomatization of the idea.  

This modified definition carries several advantages.  One is that a considerable 
amount of algebraic geometry may still be done ``as usual''on such stacks. 
Amongst examples of import to the topologists, we note that there are 
satisfactory notions of closed substack, quasi-coherent sheaf, etc.
Another advantage is that the $2$-category of Hopf algebroids becomes 
antiequivalent to the $2$-category of algebraic stacks equipped with a 
presentation, and the category of comodules for a given Hopf algebroid becomes 
equivalent to the category of quasi-coherent sheaves on the associated stack.  
Hence a useful link is forged to homotopy theory.

On the other hand, the modified definition is ultimately awkward from the point 
of view of geometry: many examples of objects that ought to be algebraic are 
not%
\footnote{A phenomenon already present in examples of interest to homotopy
theorists, as noted in \cite{goerss04}*{footnote 5}.}%
, including even all non-quasi-compact \emph{schemes}.  We shall view the
modified definition as unsuitable.  But we are still left with the problem of
finding a good setting in which to regard \FLG as a geometric object. 
One possibility is to again attempt to weaken the notion of
algebraic stack, starting, say, by requiring our presentations only to be fpqc coverings by
schemes.  Or, at least in the case of \FLG, we could get by with
presentations that are a bit more specialized; see \cite{goerss04}*{1.2, 1.13} 
for some musings on this point.  A satisfactory approach along these lines 
would certainly be highly desirable.  But many phenomena in  algebraic geometry 
are only well understood with certain finiteness hypotheses in place, and we 
believe that such an approach is likely to present some untoward foundational 
issues, at least weighed against the immediate aims of this paper.  So we shall take 
another tack.

We shall continue to understand algebraic stack in the traditional sense, as in
\cite{lamb00}.  Then, indeed, \FLG is not algebraic.  But in some sense it 
is not far from algebraic.  Namely, we may naturally describe it as a 
\emph{pro-algebraic} stack.  We do so by
considering the algebro-geometric classification of what are
classically known as \emph{$n$-bud laws} \eqref{def:nbud.law}.  These were introduced in Lazard's
seminal paper \cite{laz55}; informally, they are just truncated
formal group laws.  The moduli stack \Bn of $n$-buds
is a perfectly good algebraic stack \eqref{st:nbuds=AT_nbsBLT_n}, and, quite as a formal group law is a ``limit'' of
$n$-bud laws, we obtain the stack of formal Lie groups as the
``limit'' of the stacks \Bn \eqref{st:FLG=limBn}.


In \s\s\ref{s:ht_strat_buds} and \ref{s:ht_strat_flg} we turn to the essential 
feature of the geometry of the stacks \Bn and \FLG, respectively, namely the 
\emph{height stratification} relative to a fixed prime  $p$.  These sections 
form the core of the paper.  The height stratification on \FLG consists of an
infinite descending chain of closed substacks
\[
   \FLG = \FLG^{\geq 0} \ctnsneq \FLG^{\geq 1} \ctnsneq \dotsb,
\]
and, for each $n$, the height stratification on \Bn consists of a finite 
descending chain of closed substacks
\[
   \Bn = \Bn^{\geq 0} \ctnsneq \Bn^{\geq 1} \ctnsneq \dotsb.
\]
As $n$ varies, the stratifications on \Bn are compatible in a suitable sense, 
and their ``limit'' recovers the stratification on \FLG \eqref{st:FLG>=h=lim}.

One of our main results in the following.

\begin{thm*}[\ref{st:nbuds_sm}]
\Bn is smooth over $\Spec \ZZ$ of relative dimension $-1$ at every point, and, 
when it is defined, $\Bn^{\geq h}$ is smooth over 
$\Spec\FF_p$ of relative dimension $-h$ at every point.
\end{thm*}

We emphasize that the theorem asserts not only that the locally closed strata
are smooth, but that the closed subschemes \emph{defining} the stratification
are themselves smooth.  Hence each stratum has smooth closure.  

Much of our subsequent effort is devoted to studying the strata $\FLG^h \subset
\FLG$ and $\Bn^h \subset \Bn$ of height $h$ formal Lie groups and $n$-buds, 
respecitvely. It is a classical result of Lazard \cite{laz55}*{Th\'eor\`eme IV} 
that over a separably closed field of characteristic $p$, formal group laws are 
classified up to isomorphism by their height.  We obtain a generalization as 
follows.  Let $H = H_h$ be a ``Honda'' formal group law of height $h$ defined 
over $\FF_p$; see \eqref{st:[p]_H=T^p^h.over.Fp}.  Then for any $\FF_p$-scheme 
$S$, we may view $H$ as a group law over $\Gamma(S,\O_S)$, and we define the 
functor $\sAut(H)\colon S \mapsto \Aut_{\Gamma(S,\O_S)}(H)$.

\begin{thm*}[\ref{st:FLG^h=B(AutH)}]
$\FLG^h$ is equivalent to the classifying stack $B\bigl(\sAut(H)\bigr)$ for the
fpqc topology.
\end{thm*}

There is a version of the theorem when working with the stack $\FLG_{\tr}^h$ of
formal Lie groups of height $h$ with trivialized canonical bundle. 
One has an exact sequence
\[
   1 \to \Autstr(H) \to \sAut(H) \to \GG_m,
\]
where
$\Autstr(H)$ is the sub-group functor of \emph{strict} automorphisms of $H$.
Then one obtains $\FLG_{\tr}^h \approx \sAut(H) \bs \GG_m$.  Of course, here
$\GG_m$ acts naturally on the right-hand side; the action appears on the
left-hand side as $\GG_m$'s natural action on trivializations, and this action
realizes the forgetful functor $\FLG_{\tr}^h \to \FLG^h$ as a $\GG_m$-torsor.
We note that when $K$ is a field containing $\FF_{p^h}$, $\sAut(H)(K)$ is precisely
the $h$th Morava stabilizer group studied in homotopy theory.

There is also a version of \eqref{st:FLG^h=B(AutH)} for buds. Let $H^{(n)}$ 
denote the $n$-bud law obtained from $H$ by discarding terms of degree $\geq 
n+1$, and let $\sAut(H^{(n)})$ denote the functor sending each $\FF_p$-scheme
$S \mapsto \Aut_{\Gamma(S,\O_S)}(H^{(n)})$.

\begin{thm*}[\ref{st:Bn^h=B(AutH^(n))}]
When it is defined, $\Bn^h$ is equivalent to the
classifying stack $B\bigl(\sAut(H^{(n)})\bigr)$ for the finite \'etale topology
\cite{sga3-1}*{IV 6.3}.
\end{thm*}


The results \eqref{st:Bn^h=B(AutH^(n))} and \eqref{st:FLG^h=B(AutH)} accord the
groups $\sAut(H^{(n)})$ and $\sAut(H)$ important places in the theory.  We
investigate their structure in the following way.  For each $n$, there is a
natural filtration of normal subgroups
\[
   \sAut(H^{(n)}) =: \A_0^{H^{(n)}} 
	   \supset \A_1^{H^{(n)}}
		\supset \dotsb
		\supset \A_{n-1}^{H^{(n)}}
		\supset \A_n^{H^{(n)}} := 1
\]
and an infinite filtration of normal subgroups
\[
	\sAut(H) := \A_0^H \supset \A_1^H \supset \A_2^H \supset \dotsb;
\]
see \eqref{def:A^H^(n)} and \eqref{def:A^H}, respectively.  In the case of 
$\sAut(H)$, the $\A_\bullet^H$-topology recovers the usual topology on the 
Morava stabilizer group.

We calculate the successive quotients of the $\A_\bullet^{H^{(n)}}$- and
$\A_\bullet^H$-filtrations as follows.  For the
$\A_\bullet^{H^{(n)}}$-quotients, let $l$ be the nonnegative integer such that
$p^l \leq n < p^{l+1}$, and assume $h \leq l$.

\begin{thm*}[\ref{st:AutH^(n)_succ_quot}, \ref{st:AutH_succ_quots}]
There are natural identifications of presheaves on $\Sch{}_{/\FF_p}$
\[
	\A^{H^{(n)}}_i / \A^{H^{(n)}}_{i+1} \ciso
	\begin{cases}
		\mu_{p^h-1}, & i = 0;\\
		\GG_a^{\Fr_{p\smash{^h}}}, & i = p-1,\ p^2-1,\dotsc,\ p^{l-h}-1;\\
		\GG_a, & i = p^{l-h+1}-1,\ p^{l-h+2}-1, \dotsc,\ p^l-1;\\
		0, & \text{otherwise},
	\end{cases}
\]
and
\[
   \A^H_i / \A^H_{i+1} \ciso
   \begin{cases}
	   \mu_{p^h-1}, & i = 0;\\
	   \GG_a^{\Fr_{p\smash{^h}}}, & i = p-1,\ p^2-1,\ p^3-1,\dotsc;\\
	   0, & \text{otherwise.}
   \end{cases}
\]
\end{thm*}

Here $\mu_{p^h-1} \subset \GG_m$ is the sub-group scheme of $(p^h-1)$th roots of
unity, and $\GG_a^{\Fr_{p\smash{^h}}} \subset \GG_a$ is the sub-group scheme of
fixed points for the $p^h$th-power Frobenius operator.  It follows that
$\sAut(H^{(n)})$ is a smooth group scheme over $\FF_p$ of dimension $h$
\eqref{st:AutH^(n)=sm_dim_h}.

In addition to \eqref{st:FLG^h=B(AutH)}, we obtain another description of
$\FLG^h$ via a classical theorem of Dieudonn\'e \cite{dieu57}*{Th\'eor\`eme 3} and Lubin
\cite{lub64}*{5.1.3}, the full details of which we shall not recall here.  Very roughly,
their theorem characterizes $\Aut_{\FF_{p^h}}(H)$ as the profinite group $G$ of
units in a certain $p$-adic division algebra; see \eqref{rk:O_D=End(H)} for a
precise formulation.

\begin{thm*}[\ref{rk:O_D=End(H).analog}]
There is an equivalence of stacks over $\FF_{p^h}$,   
\[
   \FLG^h \fib{\Spec \FF_{p^h}} \Spec \FF_{p^h} \approx \ilim B(G/N),
\]   
where the limit is taken over all open normal subgroups $N$ of $G$.
\end{thm*}

The theorem is really a corollary of Dieudonn\'e's and Lubin's theorem and of
\eqref{st:FLG^h=limU_n}, where we show that $\FLG^h$ is a limit of certain
classifying stacks of \emph{finite \'etale} (but nonconstant) groups over
$\FF_p$.  These groups all become constant after base change to $\FF_{p^h}$.

In \s\ref{s:val_crit} we describe some aspects of the
stacks \FLG and \Bn related to separatedness and properness.

\begin{thm*}[\ref{thm:Bn.univ.closed}]
$\Bn$ is universally closed over $\Spec \ZZ$, and, when it is defined, 
$\Bn^{\geq h}$ is universally closed over $\Spec \FF_p$.
\end{thm*}

The stacks \Bn and $\Bn^{\geq h}$ fail to be
proper because they are not separated; see \eqref{rk:sep}.  The failure of
separatedness prevents us from concluding in a formal way that \FLG and
$\FLG^{\geq h}$ also satisfy the valuative criterion of universal closedness.
Nevertheless, these stacks do satisfy the valuative criterion in many cases; see
\eqref{st:fl_u_clsd}.  By contrast, we deduce from our theory that the
\emph{stratum} $\FLG^h$ does satisfy the valuative criterion of separatedness.

\begin{thm*}[\ref{st:fl_sep}]
Let \O be a valuation ring and $K$ its field of fractions. Then $\FLG^h(\O) \to
\FLG^h(K)$ is fully faithful for all $h\geq 1$.
\end{thm*}

When \O is a discrete valuation ring, the theorem is a (very) special case of de
Jong's general theorem \cite{dj98}*{1.2} that restriction of
$p$-divisible groups from $\Spec \O$ to the generic point $\Spec K$ is fully
faithful.

We have included an appendix at the end describing some of the basic theory of
limits in $2$-categorical contexts; or more precisely, of limits in bicategories.

It is our hope that our study of the moduli stack of formal Lie groups will
serve as a rough model for treating the related stack of $p$-divisible groups
for fixed prime $p$. Indeed, this last stack also admits a pro-algebraic
description: it is a limit, over varying $n$, of algebraic stacks of certain
finite locally free group schemes of order $p^n$.

This paper is a condensed and slightly reorganized version of the author's 
Ph.D.\ thesis \cite{sm07}.


\subsection*{Acknowledgments}
It is a pleasure to thank Mark Behrens, Eugenia Cheng, Tom Fiore,
Mark Kisin, and Mike Schulman for helpful conversation and suggestions.  I am
especially grateful to Peter May for reading a preliminary draft of this paper and for
offering a number of excellent comments; and to Paul Goerss for a great deal of
help towards understanding the place of the stack of formal Lie groups in
homotopy theory, including sharing with me a draft of \cite{goerss?}.  
Above all, I deeply thank my Ph.D. advisor,
Bob Kottwitz, for his patience and constant encouragement in overseeing this
work.

\subsection*{Notation and conventions}\label{s:not&convs}
Except where noted otherwise, we adopt the following notation and conventions.

We relate objects in a category by writing
\begin{center}
\begin{tabular}{l l l}
   $=$ & for & equal;\\
   $\ciso$ & for & canonically isomorphic;\\
   $\iso$ & for & isomorphic; and\\
   $\approx$ & for & \parbox[t]{64.5ex}{equivalent or $2$-isomorphic (e.g.\
      for categories, fibered categories, stacks, etc.).}
\end{tabular}
\end{center}
For simplicty, we often label arrows obtained in a tautological or canonical way
by $\xra\can$. Context will always make the precise meaning clear.

We write \Sets for the category of sets, $(\Gp)$ for the category of groups, and \Sch for the category of schemes.  
We write $\Fib(\C)$ for the $2$-category of fibered categories over a category
\C. We abbreviate the term ``category fibered in groupoids'' by CFG, and we
write $\CFG(\C)$ for the $2$-category of CFG's over a category \C. By a
\emph{stack}, we mean any fibered category in which morphisms and objects 
satisfy descent.  In particular, we do not require stacks to be CFG's, contrary 
to many authors' convention; compare, for example, \cite{lamb00}*{3.3.1}.  


Given an object $S$ in a category \C, we write $\C_{/S}$ for the overcategory of
$S$-objects, that is, of objects $T$ equipped with a morphism $T \to S$ in \C.
When $\C = \Sch$ and $S$ is an affine scheme $\Spec A$, we often write
$\Sch{}_{/A}$ instead of $\Sch{}_{/S}$.

Given an object $S$ in a category \C, we typically write just $S$ again for
both the presheaf of sets on, and fibered category over, \C determined by $S$,
$T \mapsto \Hom_\C(T,S)$.  When clarity demands that the notation distinguish
between $S$ and its associated presheaf or fibered category, we denote the
latter two by $\ul S$.

By default, ``sheaf'' or ``presheaf'' means ``sheaf of sets'' or ``pre\-sheaf of
sets'', respectively.


When working in a $2$-categorical context (for example, the case of fibered
categories and stacks), we say that a diagram of objects and $1$-morphisms
\[
   \xymatrix{
      A \ar[r] \ar[d] & B \ar[d]\\
      C \ar[r] & D
      }
\]
is \emph{Cartesian} if the two composites $A \to D$ are isomorphic, and the choice of
an isomorphism induces an equivalence $A \eqarrow C \fib D B$.  

All rings are commutative with $1$.  We write $k(s)$ for the residue
field at the point $s$ of a scheme.  Given a ring $A$ and an $A$-module $M$, we write $\wt M$ for the quasi-coherent
$\O_{\Spec A}$-module on $\Spec A$ obtained from $M$.

Given a scheme $S$, we write $\Gamma(S)$ for the global sections
$\Gamma(S,\O_S)$ of the structure sheaf.  In addition, for each integer $n \geq
0$ and indeterminates $T_1,\dotsc$, $T_m$, we define the ring
\begin{align*}
   \Gamma_n(S;T_1,\dotsc,T_m) 
      := {}& \Gamma\bigl(S, \O_S[T_1,\dotsc,T_m]/(T_1,\dotsc,T_m)^{n+1}\bigr)\\
      \ciso {}& \Gamma(S)[T_1,\dotsc,T_m](T_1,\dotsc,T_m)^{n+1},
\end{align*}
where $\O_S[T_1,\dotsc,T_m]/(T_1,\dotsc,T_m)^{n+1}$ is the sheaf on $S$ defined
on each open subset $U$ by
\[
   U \mapsto \Gamma(U,\O_S)[T_1,\dotsc,T_m]/(T_1,\dotsc,T_m)^{n+1}.
\]
In particular, we have $\Gamma_0(S;T_1,\dotsc,T_m) \ciso \Gamma(S)$.

\section{Definitions}\label{s:defs}

We begin by introducing some of the basic language and notation related to the
objects we study in this paper.

\subsection{Review of formal group laws I}\label{ss:reviewI}

Before beginning to study formal group laws from the point of view of algebraic
geometry, it will be convenient to briefly review some of the foundations of the
classical algebraic theory.  We will review aspects of the classical theory
related to the notion of \emph{height}, as well as some refinements to the 
material we discuss here, in \s\ref{ss:reviewII}.  The reader may wish to skip
this section and refer back only as needed.

Let $A$ be a ring.

\begin{defn}\label{def:fl.gp.law}
A \emph{formal Lie group law,} or just \emph{formal group law, over $A$} is a
power series $F(T_1,T_2)\in A[[T_1, T_2]]$ in two variables satisfying
\begin{enumerate}
\renewcommand{\theenumi}{I}
\item\label{it:gp.law.id}
	(identity) $F(T,0) = F(0,T) = T$;
\renewcommand{\theenumi}{A}
\item\label{it:gp.law.assoc}
	(associativity) $F\bigl(F(T_1,T_2),T_3\bigr) = F\bigl(T_1,F(T_2,T_3)\bigr)$; 
	and 
\renewcommand{\theenumi}{C}
\item\label{it:gp.law.comm}
	(commutativity) $F(T_1,T_2) = F(T_2,T_1)$;
\end{enumerate}
here the equalities are of elements in the rings $A[[T]]$, $A[[T_1,T_2,T_3]]$,
and $A[[T_1,T_2]]$, respectively. A \emph{homomorphism} $F \to G$ of formal
group laws over $A$ is a power series $f(T) \in A[[T]]$ with constant term $0$
satisfying
\[\tag{$*$}\label{disp:gp.law.hom.cond}
	f\bigl(F(T_1,T_2)\bigr) = G\bigl(f(T_1),f(T_2)\bigr)		
\]
in $A[[T_1,T_2]]$.
\end{defn}

One obtains the notion of a bud law as a ``truncated'' version of
\eqref{def:fl.gp.law}.  Let $n\geq 0$.

\begin{defn}\label{def:nbud.law}
An \emph{$n$-bud law over $A$} is an element 
\[
	F(T_1,T_2) \in A[T_1,T_2]/(T_1, T_2)^{n+1}
\]
satisfying conditions \eqref{it:gp.law.id}, \eqref{it:gp.law.assoc}, and
\eqref{it:gp.law.comm} of \eqref{def:fl.gp.law} in the rings 
\[
   A[T]/(T)^{n+1},\ A[T_1,T_2,T_3]/(T_1,T_2,T_3)^{n+1},\ 
	   \text{and}\ A[T_1, T_2]/(T_1,T_2)^{n+1},
\]
respectively. A \emph{homomorphism} $F \to G$ of $n$-bud laws over $A$ is an
element $f(T) \in A[T]/(T)^{n+1}$ with constant term $0$ satisfying
\eqref{disp:gp.law.hom.cond} in the ring $A[T_1,T_2]/(T_1,T_2)^{n+1}$.
\end{defn}

\begin{eg}\label{eg:add&mult_laws}
Over any ring,
\begin{itemize}
\item
   the \emph{additive law} or \emph{$n$-bud} is given by $F(T_1,T_2) = T_1 +
   T_2$; and
\item
   the \emph{multiplicative law} or \emph{$n$-bud}, $n \geq 2$, is given by
   \[
      F(T_1,T_2) = (1+T_1)(1+T_2) - 1 = T_1 + T_2 + T_1T_2.
   \]
\end{itemize}
\end{eg}

\begin{rk}\label{rk:inverse}
We recall that any group law or bud law $F$ over $A$ is automatically equipped 
with a unique inverse homomorphism $i(T) = i_F(T)\colon F \to F$ satisfying $F\bigl(i(T),T\bigr) = 
F\bigl(T,i(T)\bigr) = 0$.
\end{rk}


\begin{rk}\label{rk:End_A(F)}
Let $F$ be a group law or bud law over $A$.  Then the axioms endow the monoid
$\End_A(F)$ with a natural (noncommutative) \emph{ring} structure: 
multiplication is composition of power series in the group law case, or of
truncated polynomials in the bud law case, and addition is application of $F$;
that is, the sum of $f(T)$ and $g(T)$ is $F\bigl(f(T),g(T)\bigr)$.  We denote
the sum in $\End_A(F)$ by $f +_F g$ to avoid confusion with usual addition of
power series or of truncated polynomials.  The unit element in $\End_A(F)$ is
just $\id(T):=T$.
\end{rk}

\begin{rk}\label{rk:+.in.End_A(F)}
It is convenient to now make the following technical, though obvious, remark,
which we'll use in \s\ref{ss:aut_buds_ht_h}.  Suppose that
\[
	f(T) = a_n T^n + a_{n+1} T^{n+1} + \dotsb
	\quad\text{and}\quad
	g(T) = b_m T^m + b_{m+1} T^{m+1} + \dotsb
\]
are endomorphisms of the group law or bud law $F$.  Say $n \leq m$.  Then since $F$ is of the form
\[
	F(T_1, T_2) = T_1 + T_2 + \text{(higher order terms)},
\]
we observe that
\begin{multline*}
	(f +_F g)(T) = a_n T^n + a_{n+1} T^{n+1} + \dotsb + a_{m-1} T^{m-1}\\
		+ (a_m + b_m) T^m + \text{(higher order terms)}.
\end{multline*}
\end{rk}

Let $\varphi\colon A \to B$ be any ring homomorphism. If $F$ is a formal group 
law or bud law over $A$, then application of $\varphi$ to the coefficients of 
$F$ yields a group law or bud law, respectively, $\varphi_*F$ over $B$.  While
it is easy to see that \emph{universal} group laws and bud laws exist, Lazard
\cite{laz55} obtained the following remarkable description of the rings over
which these universal laws are defined.  Let $n \geq 1$.


\begin{thm}[Lazard]\label{st:Laz.thm}
There exists a universal $n$-bud law $U_n$ over the polynomial ring 
$\ZZ[t_1,\dotsc,t_{n-1}]$ and a universal formal group law $U$ over the polynomial 
ring $\ZZ[t_1,t_2,\dotsc]$.
\end{thm}

In other words, the rings $\ZZ[t_1,\dotsc,t_{n-1}]$ and $\ZZ[t_1,t_2,\dotsc,]$ 
\emph{corepresent} the functors $A \mapsto \{n\text{-bud laws over}\ A\}$ and 
$A \mapsto \{\text{formal group laws over}\ A\}$, respectively.  We will recall
a refinement of Lazard's theorem later in \eqref{st:Laz.thm.refinement}.

\begin{proof}[Proof of \eqref{st:Laz.thm}]
This was first proved, though only implicitly for bud laws, in 
\cite{laz55}*{Th\'eor\`emes II and III and their proofs}.  One may also consult 
\cite{haz78}*{I 5.3.1, 5.7.3, 5.7.4} (\cite{haz78} refers to buds as 
``chunks'').
\end{proof}

\begin{rk}
The choice of $U_n$ and $U$ in the theorem are by no means canonical.  That is,
$\ZZ[a_1,\dotsc,a_{n-1}]$ and $\ZZ[a_1,a_2,\dotsc,]$ \emph{noncanonically}
corepresent the respective functors $A \mapsto \{n\text{-bud laws over}\ A\}$
and $A \mapsto \{\text{formal group laws over}\ A\}$.  So already one sees an
advantage in ``dividing out'' out by the isomorphisms between the universal laws
to obtain more canonical objects, so that no particular choice of universal law
is preferred.  Hence one is led to the \emph{stacks} of formal Lie groups and of
$n$-buds.
\end{rk}

\subsection{Formal Lie varieties}\label{ss:flv}

There are many (equivalent) ``geometric'' definitions of formal Lie groups in
the literature.  We shall follow \cites{gro74, mes72} in our treatment.  In
passing from formal group laws to formal Lie groups, one may take as point of
departure the following observation: to give a formal group law over the ring
$A$ is to give a group structure on the formal scheme $\Spf A[[T]]$ over $\Spec
A$ with the $0$ section as identity; see \eqref{eg:trivfg}.  In general, a
formal Lie group is something modeled locally by this picture; see
\eqref{def:fg}.  Hence the basic geometry of formal Lie groups lies in the realm
of \emph{formal} geometry.  In this section, we introduce what are, in some
sense, the basic formal-geometric objects underlying the formal Lie groups,
namely the \emph{formal Lie varieties}.

Let $S$ be a scheme.

\begin{defn}\label{def:fv}
A \emph{(pointed, $1$-parameter) formal Lie variety over $S$} is a sheaf $X$ on 
$\Sch{}_{/S}$ for the fppf topology equipped with a section $\sigma\colon S \to 
X$, such that, Zariski locally on $S$, there is an isomorphism of pointed 
sheaves $X \iso \Spf \O_S[[T]]$, where $\Spf\O_S[[T]]$ is pointed by the 
$0$-section.  A \emph{morphism} of formal Lie varieties is a morphism of pointed
sheaves.
\end{defn}

In other words, a (pointed, $1$-parameter) formal Lie variety is a pointed
formal scheme over $S$ locally (on $S$) of the form $\Spf\O_S[[T]]$.

\begin{rk}\label{rk:general_fv}
More generally, a \emph{pointed formal Lie variety} (without condition on the
number of parameters) is a pointed sheaf over $S$ Zariski locally of the form
$\Spf \O_S[[T_1,\dotsc,T_n]]$ for some $n$ locally constant on $S$
\cite{gro74}*{III 6.7}.  A \emph{formal Lie variety} (without specified point)
is a sheaf over $S$ fpqc locally of the form
\[
   \Spf \O_S[[T_1,\dotsc,T_n]]
\]
for varying $n$ \cite{gro74}*{VI 1.3}.  Though these generalizations of
\eqref{def:fv} are certainly not without their place, we will have little
occasion to consider them here.  So, for simplicity, we shall abuse language and
always understand ``formal Lie variety'' to mean ``pointed, $1$-parameter formal
Lie variety'', unless explicitly stated otherwise.
\end{rk}

\begin{eg}\label{eg:whAA}
The most basic and important example of a formal Lie variety over any base $S$
is just $\Spf\O_S[[T]]$ itself, equipped with the $0$ section, that is, with the
map specified on algebras $\O_S[[T]] \to \O_S$ by $T\mapsto 0$. We denote this
example by $\wh\AA_S$ or, when the base is clear from context, by $\wh\AA$. When
$S$ is an affine scheme $\Spec A$, we also denote $\wh\AA_S$ by $\wh\AA_A$.

Our notation is nonstandard. It is typical to write $\wh\AA_S^1$ for the formal
line $ \Spf \O_S[[T]] $ obtained by completing $\AA_S^1$ at the origin. But
since our interest is almost exclusively in pointed, $1$-parameter formal Lie
varieties, we shall suppress the superscript ${}^1$ to reduce clutter, and we
shall always understand $\wh\AA_S$ to be equipped with the zero section.
\end{eg}

\begin{eg}
More generally, if $T$ is any smooth scheme of relative dimension $1$ over $S$
and $S \to T$ is a section, then the completion of $T$ along the section is a
formal Lie variety over $S$.
\end{eg}

\begin{rk}
One can give a more intrinsic version of the definition of formal Lie variety:
see \cite{gro74}*{VI 1.3} or \cite{mes72}*{II 1.1.4}.  But \eqref{def:fv} has
the advantage of being reasonably concrete, and it will certainly suffice for
our purposes.
\end{rk}

\subsection{Formal Lie groups}\label{ss:flg}

Let $S$ be a scheme.

\begin{defn}\label{def:fg}
A \emph{(commutative, $1$-parameter) formal Lie group over $S$} is an fppf sheaf
of commutative groups on $\Sch{}_{/S}$ such that the
underlying pointed sheaf of sets is a formal Lie variety \eqref{def:fv}.
\end{defn}

In other words, a formal Lie group is a formal Lie variety $(X,\sigma)$ made
into a commutative group object in the category of sheaves on $\Sch{}_{/S}$, such
that the given section $\sigma$ is the identity section.

\begin{rk}\label{rk:general_fg}
One can certainly formulate \eqref{def:fg} without the commutativity
condition.  And, analogously to \eqref{rk:general_fv}, one may define a
\emph{formal Lie group} in full generality as an fppf sheaf of
(not-necessarily-commutative) groups over $S$ whose underlying pointed sheaf of
sets is a formal Lie variety in the general sense of \eqref{rk:general_fv}.  But
as before, we shall have no use for the more general notion.  So we shall
abuse language and always use ``formal Lie group'' in the sense stated in
\eqref{def:fg}.
\end{rk}

\begin{eg}\label{eg:trivfg}
Consider the formal Lie variety $\wh\AA$ over $S$ \eqref{eg:whAA}.  To
make $\wh\AA$ into a formal Lie group, we must define a multiplication map
$\wh\AA \fib S \wh\AA \to \wh\AA$. Since
\[
   \wh\AA \fib S \wh\AA \ciso \Spf \O_S[[T_1,T_2]],
\] 
we may equivalently define a continuous map of $\O_S$-algebras
\[
   \O_S[[T_1,T_2]] \longleftarrow \O_S[[T]].
\]
Any such map is determined by the image $F(T_1,T_2)$ of $T$ in the global
sections $\Gamma(S)[[T_1,T_2]]$. Then $\wh\AA$ becomes a formal Lie group with
the $0$ section as identity exactly when $F$ is a formal group law over
$\Gamma(S)$ in the classical sense \eqref{def:fl.gp.law}. Hence, to give a
formal group law is to give a formal Lie group with a choice of coordinate. We
write $\wh\AA^F = \wh\AA_S^F$ for the group structure on $\wh\AA$ obtained from
$F$.
\end{eg}

\begin{eg}\label{eg:trivfg_morph}
Let $F$ and $G$ be formal group laws over $\Gamma(S)$. Then a morphism of formal
Lie varieties $f\colon \wh\AA \to \wh\AA$ is a morphism of formal Lie groups
$f\colon \wh\AA^F \to \wh\AA^G$ exactly when the diagram of maps on global
sections
\[
   \xymatrix@C+4ex{
      \Gamma(S)[[T_1,T_2]]
         & \Gamma(S)[[T_1,T_2]] \ar[l]_-{(f\times f)^\#} \\
      \Gamma(S)[[T]] \ar[u]^-{F(T_1,T_2)}
         & \Gamma(S)[[T]] \ar[l]_-{f^\#} \ar[u]_-{G(T_1,T_2)}
      }
\]
commutes, that is, when
\[
   f^\#\bigl( F(T_1,T_2) \bigr) = G\bigl( f^\#(T_1),f^\#(T_2) \bigr).
\]
Hence we recover the classical notion of a group law homomorphism $F \to G$
\eqref{def:fl.gp.law}.
\end{eg}

\begin{eg}\label{eg:fg_eg's}\hfill
\begin{itemize}
\item
	The \emph{additive formal Lie group $\wh\GG_a = \wh\GG_{a,S}$ over $S$} is
$\wh\AA_S^F$ for
\[
	F(T_1,T_2) = T_1 + T_2
\]
the additive group law \eqref{eg:add&mult_laws}.  $\wh\GG_a$ is the completion
of $\GG_a$ at the identity.
\item
	The \emph{multiplicative formal Lie group $\wh\GG_m = \wh\GG_{m,S}$ over $S$}
is $\wh\AA_S^F$ for
\[
	F(T_1,T_2) = T_1 + T_2 + T_1T_2
\]
the multiplicative group law \eqref{eg:add&mult_laws}.  $\wh\GG_m$ is the
completion of $\GG_m$ at the identity.
\item
   If $E$ is an elliptic curve over $S$, then the completion of $E$ at the
identity is a formal Lie group over $S$.  Note that this furnishes many examples
of formal Lie groups not admitting a global coordinate.
\item
	More generally, completion at the identity of any smooth commutative group
	scheme of relative dimension $1$ yields a formal Lie group.  When $S$ is
	$\Spec$ of an algebraically closed field, then $\GG_a$, $\GG_m$, and elliptic
	curves are the \emph{only} such connected group schemes.
\end{itemize}
\end{eg}

At this juncture, we could perfectly well begin to consider the moduli stack of
formal Lie groups.  But, as noted in the introduction, it turns out that this
stack is not algebraic.  So in the next section, we shall begin laying the
groundwork to study the related moduli problem of classifying \emph{$n$-buds};
this will afford an ``algebraic approximation'' to the moduli stack of formal
Lie groups, in a sense we make precise in \eqref{st:FLG=limBn}.  We shall return
to the moduli stacks of formal Lie varieties and of formal Lie groups in
\s\ref{ss:flv_stack} and \s\ref{ss:flg_stack}, respectively.

\subsection{Infinitesimal neighborhoods}

Roughly speaking, in algebra, a bud law is a truncation of a formal group law.
In the geometric setting, the role of truncation will be played by \emph{taking
infinitesimal neighborhoods}. In other words, roughly
speaking, a bud will be an infinitesimal neighborhood of the identity of a
formal Lie group. Hence the basic geometry of buds lies in the realm of
\emph{infinitesimal} geometry. In this section we introduce some of the basic
language surrounding these ideas.

Let us begin by recalling the notion of infinitesimal neighborhood from
\cite{gro74}*{VI 1.1} or \cite{mes72}*{II 1.01}.  Let $S$ be a base scheme, $Y$ 
and $X$ sheaves on $\Sch{}_{/S}$ for some fixed topology between the fpqc and
Zariski topologies, inclusive, and $Y \inj X$ a monomorphism. Let $n \geq 0$.

\begin{defn}\label{def:inf}
The \emph{$n$th infinitesimal neighborhood of $Y$ in $X$} is the subsheaf of $X$
obtained as the sheafification of the subpresheaf of $X$ defined on points by
\[
   T \mapsto 
      \left\{ T \rightarrow X \left|
	   \begin{minipage}{52ex}
		\centering there exists a commutative diagram
		\[
			\xymatrix{
				T' \ar@{^{(}->}[r] \ar[d]
					& T \ar[d]\\
				Y \ar@{^{(}->}[r]
					& X
			}
		\]
		for some closed subscheme $T'$ of $T$ defined by a
		quasi-coherent ideal $\I \subset \O_T$ satisfying $\I^{n+1} =
		0$.
	   \end{minipage}
	\right.
	\right\}.
\]
We denote the $n$th
infinitesimal neighborhood by $\Inf_Y^n(X)$ or, when $Y$ is clear from
context, by $X^{(n)}$.
\end{defn}

\begin{rk}
When $Y \inj X$ is a closed immersion of schemes, say with associated sheaf of
ideals $\I \subset \O_X$, one typically defines $\Inf_Y^n(X)$ to be the closed
subscheme of $X$ whose underlying topological space is the image of $Y$ and
whose structure sheaf is the restriction of $\O_X/\I^{n+1}$.  It is not hard to
verify that, in this situation, this notion of $\Inf$ and the notion of
\eqref{def:inf} agree; see \cite{mes72}*{II 1.02}.  
\end{rk}

\begin{rk}\label{rk:Inf_fctrl}
One verifies at once that $\Inf$ is \emph{functorial} in $X$ and $Y$: precisely,
any commutative diagram
\[
   \xymatrix{
      Y \ar@{^{(}->}[r]  \ar[d]    & X \ar[d]\\          
      Y' \ar@{^{(}->}[r]            & X',
      }
\]
in which the horizontal arrows are monomorphisms, induces a map $\Inf_Y^n(X) \to
\Inf_{Y'}^n(X')$ for any $n$.
\end{rk}

\begin{rk}\label{rk:Inf_bc}
Formation of $\Inf_Y^n(X)$ is compatible with base change on $S$ by
\cite{mes72}*{II 1.03}.
\end{rk}

We will be especially interested in infinitesimal neighborhoods in the case $Y =
S$, so that $X$ is \emph{pointed}.  In this situation, we have the 
following definition.

\begin{defn}\label{def:nInf}
$X$ is \emph{$n$-infinitesimal} (resp., \emph{ind-infinitesimal}) if the natural
arrow $X^{(n)} \to X$ (resp., $\clim_n X^{(n)} \to X$) is an
isomorphism.  We denote by $\nInf(S)$ (resp., $\nInf[\infty](S)$) the category
of $n$-infinitesimal (resp., ind-infini\-tes\-i\-mal) sheaves over $S$.
\end{defn}

\begin{eg}
The motivation for introducing \eqref{def:nInf} is simply that any formal Lie
variety is ind-infinitesimal, and any $n$th infinitesimal neighborhood of a
pointed sheaf is $n$-infinitesimal.
\end{eg}

\begin{rk}
Of course, one has the obvious implications
\[
   \text{$n$-infinitesimal} \implies \text{$m$-infinitesimal for $m \geq n$}
      \implies \text{ind-infinitesimal}.
\]
\end{rk}

\begin{rk}
One verifies at once that the properties of being $n$-infinitesimal or ind-infinitesimal are stable under base change.  Hence \nInf and $\nInf[\infty]$ define fibered categories over \Sch.
\end{rk}

\begin{rk}\label{rk:truncation}
As a special case of \eqref{rk:Inf_fctrl}, the assignment $X \mapsto X^{(n)}$ 
specifies a functor
\[
   \left\{\parbox{19ex}{\centering
      pointed fppf sheaves on $\Sch{}_{/S}$
   }\right\}
   \to \nInf(S).
\]
Hence we get a natural ``truncation'' functor
\[
   \nInf[\infty](S) \to \nInf(S).
\]
We shall return to the topic of truncation a bit more in earnest in \s\ref{ss:inf_inf_sheaves}.
\end{rk}

\begin{rk}\label{rk:inf_prods}
If $X$ and $Y$ are $n$- and $m$-infinitesimal, respectively, sheaves over $S$, then the product (in the category of pointed sheaves over $S$) $X\fib S Y$ is $(n+m)$-infinitesimal.  Hence a product of ind-infinitesimal sheaves is ind-infinitesimal.  

On the other hand, $\nInf(S)$ itself contains finite products: namely, the
product of sheaves $X$ and $Y$ in $\nInf(S)$ is $(X\fib S Y)^{(n)}$.
\end{rk}


\subsection{Germs}\label{ss:germs}

In this section we introduce the ``truncated'' analogs of the formal Lie 
varieties, namely the \emph{germs}; these will hold the same relation to buds 
that the formal Lie varieties hold to formal Lie groups.  To define $n$-germs,
we could, so to speak, just ``take the $n$th infinitesimal
neighborhood'' of the definition of formal Lie variety \eqref{def:fv}; that is, we
could copy \eqref{def:fv} word-for-word, only replacing $\O_S[[T]]$ everywhere
with $\O_S[T]/(T)^{n+1}$.  Instead, we shall give a somewhat more intrinsic
definition. We shall deduce that our definition amounts to the infinitesimal
version of a formal Lie variety in \eqref{st:ngermloctriv}.

Let $n\geq 0$, and let $S$ be a scheme.

\begin{defn} \label{def:ngerm}
A \emph{(pointed, $1$-parameter) $n$-germ over $S$} is a pointed, representable,
$n$-infinitesimal \eqref{def:nInf} sheaf $X$ over $S$ satisfying
\begin{enumerate}
\renewcommand{\theenumi}{LF1}
\item\label{it:lf1}
   if \I is the sheaf of ideals on $X$ associated to the given section $\sigma
   \colon S \to X$, then each successive quotient $\sigma^*(\I^i/\I^{i+1})$ for $i =
   1,\dotsc$, $n$ is a locally free $\O_S$-module of rank $1$.
\end{enumerate}
We often denote $n$-germs as triples $(X,\pi,\sigma)$, where $\pi$ denotes
the structure map $X \to S$ and $\sigma$ denotes the section.
\end{defn}

It is convenient to allow the $n = 0$ case; then \eqref{it:lf1} is vacuous, and
a $0$-germ over $S$ consists just of an isomorphism $X \isoarrow S$.

\begin{rk}\label{rk:ngerm.linguistics}
One could give a fully general definition of $n$-germs in a way entirely
analogous to \eqref{rk:general_fv}. But, just as in \eqref{rk:general_fv}, we
won't have occasion to consider the more general notion. So we shall abuse
language and always use ``$n$-germ'' in the sense stated in \eqref{def:ngerm}.
\end{rk}

\begin{eg}\label{eg:trivngerm}
The most basic and important example of an $n$-germ over $S$ is just the
infinitesimal version of $\wh\AA_S$ \eqref{eg:whAA},
\[
	\TT := \TT_S := \TT_{n,S} := \Spec \O_S[T]/(T)^{n+1},
\]
equipped with the $0$ section. In our discussion of germs, we will often work
with respect to a fixed nonnegative integer $n$ and/or a fixed base $S$, and,
accordingly, we will often favor the plainer notation \TT or $\TT_S$ over
$\TT_{n,S}$, provided no confusion seems likely.  When $S$ is an affine scheme
$\Spec A$, we also denote $\TT_S$ by $\TT_A$.  

Our choice of the letter \TT is for ``trivial''; see \eqref{def:trivgerm}.
\end{eg}

\begin{rk}\label{rk:trunc_germs}
Note that for $m \geq n$, we have $\wh\AA ^{(n)}_S \ciso \TT_{m,S}^{(n)} \ciso 
\TT_{n,S}$.  More generally, any truncation of a formal Lie variety or of a germ
is a germ.
\end{rk}

\begin{rk}
Taking $i = 1$ in the definition \eqref{def:ngerm}, we see that, in particular,
the conormal sheaf associated to $\sigma$ is a line bundle on $S$.  This line
bundle will play an important role in our later study of the \emph{height
stratification}.  We remark in addition that, since $\sigma$ is a section of
$\pi$, the conormal sheaf is canonically isomorphic to $\sigma^*
(\Omega^1_{X/S})$ \cite{egaIV4}*{17.2.5}.
\end{rk}

\begin{rk}
It is immediate that the morphisms $\pi$ and $\sigma$ in \eqref{def:ngerm} are
inverse homeomorphisms of topological spaces.  Hence each sheaf
$\sigma^*(\I^i/\I^{i+1})$ appearing in condition \eqref{it:lf1} is canonically
isomorphic to $\pi_*(\I^i/\I^{i+1})$.
\end{rk}

\begin{rk}
In the language of \cites{gro74,mes72}, we could have defined an $n$-germ over
$S$ as an object of $\nInf(S)$, representable and smooth to order $n$ along the section
\citelist{\cite{gro74}*{VI 1.2} \cite{mes72}*{II 3.1.2}}, such that the
conormal sheaf is a line bundle on $S$.  The equivalence of this notion with that of our definition \eqref{def:ngerm}, if not immediately apparent, will be made clear in \eqref{st:ngermloctriv}.
\end{rk}

\begin{defn}\label{def:(ngerms).and.Gn}
We define $\ngerms(S)$ to be the full subcategory of pointed sheaves on
$\Sch{}_{/S}$ consisting of the $n$-germs.  We define $\G_n(S)$ to be the 
groupoid of $n$-germs and their isomorphisms.
\end{defn}

It follows from \eqref{st:ngerm.base.change} below that $\ngerms$ and $\G_n$ 
define a fibered category and a CFG, respectively, over \Sch.  For the moment, 
consider the general situation of a scheme $X$ over $S$ and a quasi-coherent 
sheaf of ideals $\I \subset \O_X$.

\begin{lem}[compare \cite{egaIV4}*{16.2.4}]
Let $i\geq 1$, and suppose $\I^j/\I^{j+1}$ is $S$-flat for $0 \leq j\leq i-1$.  Then formation of $\I^i$ is compatible with base change on $S$; that is, for any Cartesian square
\[
   \xymatrix{
      X' \ar[r]^-p \ar[d] & X\ar[d]\\
      S\smash{'} \ar[r] & S\smash{,}
   }
\]
$\I' := p^*\I$ is naturally identified with a sheaf of ideals in $\O_{X'}$, and the natural arrow $p^*(\I^i) \to (\I')^i$ is an isomorphism.
\end{lem}

\begin{proof}
Consider the exact sequence
\[
   \I^{\tensor i} \xra m \O_X \to \O_X/\I^i \to 0,
\]
where $m$ is the multiplication map.  Now, $\O_X/\I^i$ admits a filtration
\[
   \O_X/\I^i
	  \supset \I/\I^i
	  \supset \I^2/\I^i 
	  \supset \dotsb 
     \supset \I^{i-1}/\I^i 
     \supset 0
\]
whose successive quotients are $S$-flat by hypothesis. Hence $\O_X/\I^i$ is
$S$-flat. That is, $m$ has flat cokernel. Hence formation of the image of $m$ is
compatible with base change. The lemma follows.
\end{proof}

\begin{lem}\label{st:ngerm.base.change}
Suppose $X$ is an $n$-germ over $S$.  Then for any base change $S' \to S$, $S' \fib S X$ is an $n$-germ over $S'$.
\end{lem}

\begin{proof}
Immediate from the previous lemma, since base change for arbitrary schemes
preserves quotients of quasi-coherent modules and the property of being finite
locally free of given rank.
\end{proof}

\subsection{Trivial germs}\label{ss:trivngerms}

Let us begin with the definition.

\begin{defn}\label{def:trivgerm}
The $n$-germ $X$ over $S$ is \emph{trivial} if it is isomorphic to $\TT_S$ \eqref{eg:trivngerm}.
\end{defn}

We shall now show that every germ is trivial Zariski locally. As a first step,
let $\pi\colon X \to S$ be a scheme over $S$ admitting a section $\sigma$ which
is a closed immersion and a homeomorphism.

\begin{lem}\label{st:nbudaff}  
$X$ is affine over $S$.
\end{lem}

\begin{proof} 
Apply \cite{sga3-1}*{VI$_{\mathrm{B}}$ 2.9.1}, which says that for any morphism of schemes
$\pi$, if $\pi$ is a homeomorphism between the underlying 
topological spaces, then $\pi$ is an affine morphism.  Alternatively, we can
give a simple direct proof in the present situation by using that the inverse
homeomorphism $\sigma$ is a closed immersion.  Consider an open affine $V$ in
$X$.  Since a closed immersion is an affine morphism, $U := \sigma^{-1}(V)$ is
affine in $S$.  Then $\pi^{-1}(U) = V$ is affine.  But we can certainly cover
$S$ with open affines of the form $U$, so we're done.
\end{proof}

Let us now specialize to the case that $(X,\pi,\sigma)$ is an $n$-germ.

\begin{lem}\label{st:freesuccquots} 
Let $\I \subset \O_X$ be the sheaf of ideals corresponding to $\sigma$. Then the natural map
\[
   (\I/\I^2)^{\tensor i} \to \I^i/\I^{i+1}
\]
is an isomorphism for $1\leq i \leq n$.
\end{lem}

\begin{proof} 
The map is an epimorphism of locally free sheaves of the same rank, hence an isomorphism.
\end{proof}

\begin{prop}\label{st:ngermloctriv}
Zariski locally on $S$, $X$ is trivial.
\end{prop}

\begin{proof}
Localizing on $S$, we may assume by \eqref{st:nbudaff} that $S = \Spec A$, $X =
\Spec B$, the composition
\[
	A \xla{\sigma^\#} B \xla{\pi^\#} A
\]
is $\id_A$, and $I := \ker[B \to A]$ is an ideal such that $I^{n+1} = 0$ and
$I/I^2$ is a free $A$-module of rank $1$.

Let $\overline t$ generate $I/I^2$ as an $A$-module and choose any lift $t\in
I$. We claim that the map $\varphi\colon A[T]/(T)^{n+1} \to B$ sending $T
\mapsto t$ is an isomorphism.  Indeed, $\varphi$ is well-defined because $t\in
I$ and $I^{n+1} = 0$.  Let us consider $A[T]/(T)^{n+1}$ as filtered $T$-adically, and $B$ as filtered $I$-adically. Then $\varphi$ is a map of filtered rings.  But by \eqref{st:freesuccquots}, $\varphi$ induces an isomorphism on associated graded algebras. Hence $\varphi$ is an isomorphism.
\end{proof}

In light of the proposition, it will be important later to understand the 
automorphisms of $\TT_S$.

\begin{defn}\label{def:AutTT_n}
We define $\sEnd(\TT_n)$ to be the presheaf of monoids on \Sch
\[
   \sEnd(\TT_n) \colon S \mapsto \End_{\ngerms(S)}(\TT_{n,S})
\]
and $\sAut(\TT_n)$ to be the presheaf of groups on \Sch
\[
   \sAut(\TT_n)\colon S \mapsto \Aut_{\ngerms(S)}(\TT_{n,S}).
\]
\end{defn}

Fix a base scheme $S$.  To give an endomorphism of $\TT_S$ is to give a map of
augmented $\O_S$-algebras $\O_S[T]/(T)^{n+1} \to \O_S[T]/(T)^{n+1}$;
this, in turn, is given by the image $a_1 T + \dotsb + a_nT^n$ of $T$ in
$\Gamma_n(S;T)$.  The endomorphism of $\TT_S$ is then an automorphism exactly when $a_1 T +
\dotsb + a_n T^n$ is invertible under composition in $\Gamma_n(S;T)$, that is,
when $a_1 \in \Gamma(S)^\times$.  We have shown the following.

\begin{prop}\label{st:AutTT=repble}
As set-valued functors, $\sEnd(\TT_n)$ is canonically represented by
\[
   \Spec\ZZ[a_1,a_2,\dotsc,a_n] = \AA_\ZZ^n,
\]
and $\sAut(\TT_n)$ is canonically represented by the open subscheme 
\[
   \Spec \ZZ[a_1,a_1^{-1},a_2,\dotsc, a_n]
\]
of $\AA_\ZZ^n$. \hfill $\square$
\end{prop}

Explicitly, the monoid structure on $\Spec\ZZ[a_1,\dotsc,a_n] = \AA_\ZZ^n$ obtained from the identification with $\sEnd(\TT_n)$ is given by composition of polynomials $a_1T + \dotsb + a_nT^n$ in the truncated polynomial ring $\ZZ[a_1, \dotsc, a_n][T]/(T)^{n+1}$.

\begin{rk}\label{rk:AutTT_n_filt}
$\sAut(\TT_n)$ admits a decreasing filtration of closed sub-group schemes
\[
   \sAut(\TT_n) =: \A_0^{\TT_n}
      \supset \A_1^{\TT_n}
      \supset \dotsb
      \supset \A_{n-1}^{\TT_n}
      \supset \A_n^{\TT_n} := 1
\]
defined on $S$-valued points by
\[
   \A_i^{\TT_n}(S) = \{\,T + a_{i+1}T^{i+1} + \dotsb + a_nT^n \mid a_{i+1},\dotsc,a_n\in \Gamma(S)\,\}, \quad 1 \leq i \leq n-1.
\]
Said differently, $\A_i^{\TT_n}$ is just the kernel of the homomorphism 
$\sAut(\TT_n) \to \sAut(\TT_i)$ induced by the identification $\TT_n^{(i)} \ciso
\TT_i$, $0 \leq i \leq n$.  One verifies at once that the map on points
\[
   T+ a_{i+1}T^{i+1} + \dotsb + a_nT^n \mapsto a_{i+1}
\]
specifies an isomorphism of $\ZZ$-groups
\[
   \A_i^{\TT_n}/\A_{i+1}^{\TT_n} \isoarrow
   \begin{cases}
      \GG_m, & i=0;\\
      \GG_a, & 1\leq i \leq n-1.
   \end{cases}
\]
We shall return to the $\A_i^{\TT_n}$'s in \s\ref{ss:aut_buds_ht_h}.
\end{rk}


\subsection{Buds}\label{ss:buds}

We now introduce the infinitesimal versions of the formal Lie groups, namely the
\emph{buds}; these are the algebro-geometric analogs of the bud laws.  Let $S$ 
be a scheme and $n \geq 0$.

\begin{defn} \label{def:nbud} 
A \emph{(commutative, $1$-parameter) $n$-bud over $S$} consists of an $n$-germ
$X$ over $S$ \eqref{def:ngerm} equipped with a morphism of $S$-schemes, which 
we think of as a multiplication map,
\[
   F\colon (X\fib S X)^{(n)} \to X,
\]
satisfying the constraints
\begin{enumerate}
\renewcommand{\theenumi}{I}
\item\label{it:id}
   (identity) the section $S \xra{\sigma} X$ is a left and right identity
   for $F$, i.e.\ both compositions in the diagram
   \[
      \xymatrix{
         X  \ar[r]^-{(\sigma\pi, \id_X)} \ar[d]_-{(\id_X, \sigma\pi)}
            & (X \fib S X)^{(n)} \ar[d]^-{F}\\
         (X \fib S X)^{(n)} \ar[r]^-F
            & X}
   \]
   equal $\id_X$;
\renewcommand{\theenumi}{A}
\item\label{it:assoc}
   (associativity) $F$ is associative on points of $(X \fib S X \fib S 
   X)^{(n)}$, i.e.\ the restrictions of $F \times \id_X$ and $\id_X \times F$ to
   $(X \fib S X \fib S X)^{(n)}$ yield a commutative diagram
   \[
      \xymatrix@C=1.2cm{
         (X \fib S X \fib S X)^{(n)} \ar[r]^-{\id_X \times F} \ar[d]_-{F \times \id_X}
            & (X \fib S X)^{(n)} \ar[d]^-{F}\\
         (X \fib S X)^{(n)} \ar[r]^-{F}
            & X;}
	\]
   and
\renewcommand{\theenumi}{C}
\item\label{it:comm}
   (commutativity) $F$ is commutative, i.e.\ letting $\tau\colon X \fib S X \to
   X \fib S X$ denote the transposition map $(x,y) \mapsto (y,x)$ and
   restricting $\tau$ to $(X\fib S X)^{(n)}$, $F$ is $\tau$-equivariant,
   i.e.\ the diagram
   \[
		\xymatrix@C=.2cm@R=4ex{
			(X \fib S X)^{(n)} \ar[rr]^-\tau \ar[rd]_-F
				& & (X \fib S X)^{(n)} \ar[ld]^-F\\
			& X
			}
	\]
	commutes.
\end{enumerate}
\end{defn}

Of course, the infinitesimal neighborhoods in the definition are all taken with
respect to the sections induced by $\sigma$. It is an immediate consequence of 
the functoriality of $\Inf$ that the diagrams in the definition are well 
defined.

\begin{rk}\label{rk:nbud.almost.group}
Heuristically, the multiplication map $F$ in the definition is almost a
commutative monoid scheme structure on $X$ over $S$, but $F$ is defined only on
points of a certain subfunctor of the product $X \fib S X$. On the other hand,
recall from \eqref{rk:inf_prods} that $(X \fib S X)^{(n)}$ is the honest
product of $X$ with itself in the category $\nInf(S)$.  Hence $n$-buds over
$S$ are commutative monoids in $\nInf(S)$.  In fact, we'll verify in
\s\ref{ss:nbuds=gps} that the $n$-buds are precisely the $n$-germs endowed with
a commutative \emph{group} structure in $\nInf(S)$.
\end{rk}

\begin{rk}\label{rk:general_nbuds}
Just as we noted for formal Lie groups \eqref{rk:general_fg}, one can certainly
consider $n$-buds on many parameters, or without the commutativity constraint
\eqref{it:comm}. But as always, we shall only be concerned with the commutative,
$1$-parameter case. So we shall always abusively understand ``$n$-bud'' in the
sense of \eqref{def:nbud}, except where explicitly stated otherwise.
\end{rk}

\begin{rk}
In the trivial case $n=0$, we have $X \isoarrow S$ and $F$ is forced to
be the canonical isomorphism $(X \fib S X)^{(0)} \isoarrow X$.
\end{rk}

\begin{eg}\label{eg:trivnbud}
Consider the $n$-germ $\TT$ over $S$ \eqref{eg:trivngerm}. Then, quite 
analogously to \eqref{eg:trivfg}, to give an $n$-bud structure 
$(\TT \fib S \TT)^{(n)} \to \TT$ with the given section as identity is to give
an $n$-bud law $F$ over $\Gamma(S)$ in the classical sense \eqref{def:nbud.law}.  Hence, to give an $n$-bud law is to give an $n$-bud with a choice of coordinate.  We write $\TT^F = \TT_S^F$ for the bud structure on $\TT$ obtained from $F$.
\end{eg}

\begin{eg}\label{eg:bud_eg's} 
Everything in \eqref{eg:fg_eg's} admits an obvious analog for buds.  In particular, let us signal the following.
\begin{itemize}
\item
   The \emph{additive $n$-bud $\GG_a^{(n)} = \GG_{a,S}^{(n)}$ over $S$} is the 
   $n$th infinitesimal neighborhood of $\GG_a$ at the identity, that is, the
   $n$-bud $\TT_{n,S}^F$ for
   \[
      F(T_1,T_2) = T_1 + T_2
   \]
   the additive bud law \eqref{eg:add&mult_laws}.
\item
   The \emph{multiplicative $n$-bud $\GG_m^{(n)} = \GG_{m,S}^{(n)}$ over $S$} is
   the $n$th infinitesimal neighborhood of $\GG_m$ at the identity, that is, the
   $n$-bud $\TT_S^F$ for
\[
	F(T_1,T_2) = T_1 + T_2 + T_1T_2
\]
the multiplicative bud law \eqref{eg:add&mult_laws}.
\end{itemize}
\end{eg}

\begin{rk}\label{rk:trunc_buds}
For $F$ a group law (resp., $m$-bud law with $m \geq n$), let us denote by
$F^{(n)}$ the $n$-bud law obtained from $F$ by discarding terms of degree $ \geq
n + 1$.  Then we have $(\wh\AA^F)^{(n)} \ciso \TT_{n}^{F^{(n)}}$ (resp., 
$(\TT_{m}^F)^{(n)} \ciso \TT_{n}^{F^{(n)}}$).  More generally,  any truncation of a formal Lie group or of a bud is a
bud, since truncation $\nInf[\infty](S) \to \nInf(S)$ and $\nInf[m](S) \to 
\nInf(S)$, $m\geq n$, preserves finite (including empty) products.
\end{rk}

There is an obvious notion of morphism of buds.

\begin{defn} \label{def:germbudmorph} 
A \emph{morphism} $f\colon X \to Y$ of $n$-buds over $S$ is a morphism of monoid
objects in $\nInf(S)$, that is, a morphism of the underlying $n$-germs such that
\[
   \xymatrix@C+4ex{
      (X \fib S X)^{(n)} \ar[r]^-{(f\times f)^{(n)}} \ar[d]
         & (Y \fib S Y)^{(n)} \ar[d]\\
      X \ar[r]^-{f}
         & Y
   }
\]
commutes.
\end{defn}

\begin{eg}\label{eg:morph.of.triv.buds}
Analogously to \eqref{eg:trivfg_morph}, to give a morphism of buds $\TT_S^F \to
\TT_S^G$ over $S$ is to give a homomorphism of bud laws $F \to G$
\eqref{def:nbud.law}.
%
\end{eg}

\begin{defn}\label{def:(nbuds).and.Bn}
We define $\nbuds(S)$ to be the category of $n$-buds and bud morphisms over $S$.  We define $\Bn(S)$ to be the groupoid of $n$-buds and bud isomorphisms over $S$.
\end{defn}

\begin{rk}\label{rk:buds.base.change}
Given an $n$-bud $X$ over $S$ and a base change $f\colon S' \to S$, we know $X'
:= S' \fib S X$ is at least an $n$-germ by \eqref{st:ngerm.base.change}. But
then the base change of the multiplication map makes $X'$ into a bud, since
infinitesimal neighborhoods \eqref{rk:Inf_bc} and fibered products are
compatible with base change. Hence $\nbuds$ and $\Bn$ define a fibered category
and a CFG, respectively, over \Sch.

Note that when $X = \TT_S^F$, one has $X' \ciso \TT_{S'}^{F'}$, where $F'$ is 
the bud law over $\Gamma(S')$ obtained by applying $f^\#$ to the coefficients of
$F$.
\end{rk}

%
%
%
%
%
%
%
%

\subsection{Buds as group objects}\label{ss:nbuds=gps}

Fix a base scheme $S$. We remarked in \eqref{rk:nbud.almost.group} that
$n$-buds over $S$ are almost commutative monoids in the category of $S$-schemes,
and are honest commutative monoids in the category $\nInf(S)$ \eqref{def:nInf}. 
As promised, we'll now see that $n$-buds are honest \emph{group} objects in 
$\nInf(S)$.

\begin{prop}\label{st:nbuds=gps} 
The $n$-buds over $S$ are precisely the $n$-germs over $S$ endowed with a
commutative group structure in $\nInf(S)$. The $n$-bud morphisms over $S$ are
precisely the homomorphisms of group objects in $\nInf(S)$.
\end{prop}

\begin{proof}
All we need to show is that every $n$-bud $X$ is automatically equipped
with an inverse morphism $X \to X$. Since the inverse is unique if it exists, it
suffices to find the inverse locally on $S$. Hence we may assume $X$ has trivial
underlying germ \eqref{st:ngermloctriv}, so that we may assume $X = \TT_S^F$ for
some $n$-bud law $F(T_1,T_2) \in \Gamma_n(S;T_1,T_2)$ \eqref{eg:trivnbud}.  Now
use (\ref{rk:inverse}, \ref{eg:morph.of.triv.buds}).
%
\end{proof}

We obtain the following as a formal consequence.

\begin{cor}\label{st:nbudmorphs=abgp} 
For $n$-buds $X$ and $Y$ over $S$, the set of bud morphisms $X \to Y$ is
naturally an abelian group. Moreover, composition of bud morphisms is bilinear. 

\hfill $\square$
\end{cor}

Explicitly, bud morphisms $X \to Y$ are added as elements of
$\Hom_{\nInf(S)}(X,Y)$ under the group structure coming from $Y$. The content
of the corollary is that bud morphisms form a \emph{subgroup} of
$\Hom_{\nInf(S)}(X,Y)$.

\begin{rk}
The category of $n$-buds over $S$ is not additive for $n\geq 1$,
since the product of $n$-germs, whether taken in the category of pointed sheaves
or $\nInf(S)$, is not again an $n$-germ.  But the problem is only that we've 
restricted to the $1$-parameter case:  commutative $n$-buds in the general sense of \eqref{rk:general_nbuds}, with no constraint on the number of parameters, do form an additive category.
\end{rk}

\section{Basic moduli theory}\label{s:moduli_theory}

We now begin to consider the basic moduli theory of the stacks of $n$-germs, 
$n$-buds, formal Lie varieties, and formal Lie groups.

\subsection{The stack of $n$-germs}\label{ss:ngerms=B(AT_n)}

In this section we show that the moduli stack of $n$-germs \Gn, $n \geq 1$ is equivalent to
the classifying algebraic stack $B\bigl(\sAut(\TT_n)\bigr)$, with $\sAut(\TT_n)$
as defined in \eqref{def:AutTT_n}. 

\begin{prop}\label{st:ngerms=stack}
\Gn is a stack over \Sch for the fpqc topology.
\end{prop}

\begin{proof}
We have to check that objects and morphisms descend.  It is clear that \Gn is
a stack for the Zariski topology since we can glue schemes and morphisms of
schemes.   So we may restrict to the case of a base scheme $S$ and a faithfully
flat quasi-compact morphism $f\colon S' \to S$.  Our argument from here will be a
straightforward application of the descent theory in \cite{sga1}*{VIII}.

Descent for morphisms of germs along $f$ is an immediate consequence of descent 
for morphisms of schemes \cite{sga1}*{VIII 5.2}.  To check descent for objects,
%
%
let $X'$ be an $n$-germ over $S'$ equipped with a descent datum. Then $X'$
is affine over $S'$ \eqref{st:nbudaff}. So $X'$ descends to a scheme $X$ affine
over $S$ \cite{sga1}*{VIII 2.1}. By descent for morphisms, the section for $X'$
descends to a section $\sigma$ for $X$. It remains to show that $X$ is
$n$-infinitesimal along $\sigma$ and that constraint \eqref{it:lf1} in the 
definition of $n$-germ \eqref{def:ngerm} holds for $X$. We may assume $X = 
\Spec \A$ for some quasi-coherent $\O_S$-algebra \A. Let $\I := \ker[\A \to 
\O_S]$ be the ideal corresponding to $\sigma$. Then $\A \ciso \O_S \oplus \I$, 
so \I certainly pulls back to the ideal corresponding to the given section $S' 
\to X'$. Moreover, since arbitrary base change preserves quotients of 
quasi-coherent modules, and flat base change preserves powers of ideals, we 
conclude that the pullback of $\I^i/\I^{i+1}$, $i = 1,\dotsc$, $n$, is locally 
free of rank $1$, and that the pullback of $\I^{n+1}$ is $0$. Since $S' \to S$ 
is faithfully flat, we conclude that $\I^i/\I^{i+1}$ itself is locally free of 
rank $1$ \cite{sga1}*{VIII 1.10}, and that $\I^{n+1} = 0$.
\end{proof}

\begin{thm}\label{st:ngerms=B(AT_n)}
\Gn is equivalent to the classifying stack $B\bigl(\sAut(\TT_n)\bigr)$.  In
particular, \Gn is algebraic.
\end{thm}

\begin{proof}
By \eqref{st:ngermloctriv} and the previous proposition, \Gn is a gerbe (for any topology between the Zariski and fpqc topologies, inclusive) over $\Spec\ZZ$.  Moreover, $\TT_\ZZ$ \eqref{eg:trivngerm} is an object of \Gn over $\Spec \ZZ$, so that \Gn is neutral.  The first assertion now follows from \cite{lamb00}*{3.21}.

As for algebraicity, we need note only that $\sAut(\TT_n)$ is a smooth,
separated group scheme of finite presentation over \ZZ by
\eqref{st:AutTT=repble}, and that the quotient of any algebraic space by such a
group scheme is algebraic \cite{lamb00}*{4.6.1}.  (More generally, one may
weaken the smoothness hypothesis by requiring only that, in its place, $G$ be 
flat and of finite presentation \cite{lamb00}*{10.13.1}.)
\end{proof}

\begin{rk}\label{rk:B(Aut(T)).indep.of.top}
Of course, for an arbitrary group sheaf $G$ on a site \C, the stack $B(G)$ 
depends on the topology on \C.  By the theorem, $B\bigl(\sAut(\TT_n)\bigr)$ is 
independent of the choice of topology on \Sch between the Zariski and fpqc topologies, 
inclusive.  In particular, every fpqc-torsor for $\sAut(\TT_n)$ is in fact a 
Zariski-torsor.
\end{rk}

\subsection{Bud structures on trivial germs}\label{ss:bud.str.on.germs}

To prepare for our discussion of the moduli stack of $n$-buds in the next
section, recall from 
\eqref{st:ngermloctriv} that every bud has locally trivial underlying germ.  
Hence the classification of bud structures on $\TT_S$ \eqref{eg:trivngerm} 
assumes an important 
role in the theory.  Let $n \geq 1$.

\begin{defn}\label{def:laz_n}
We define $L_n$ to be the presheaf of sets on \Sch
\[
   L_n\colon S \mapsto \{\text{$n$-bud structures on}\ \TT_{n,S}\}.
\]
\end{defn}

Fix a base scheme $S$. By \eqref{eg:trivnbud}, to give an $n$-bud structure on $\TT_S$ is to give an $n$-bud law $F \in \Gamma_n(S;T_1,T_2)$. So by Lazard's theorem \eqref{st:Laz.thm}, $\laz_n \iso \AA_\ZZ^{n-1}$ for $n\geq 1$, but the isomorphism is not canonical. In the trivial case $n=0$, we have $\laz_0 \ciso \Spec \ZZ$.

We have chosen the notation $L_n$ in honor of Lazard.

\begin{rk}\label{rk:AutTTn.acts.on.Ln}
The functor $\sAut(\TT_n)$ \eqref{def:AutTT_n} acts naturally on $L_n$ as
``changes of coordinate'': given a bud structure $\TT_S^F$ and a germ
automorphism $f$ of $\TT_S$, transport of structure along $f$ determines a bud
structure $\TT_S^G$, and $f$ is tautologically a bud isomorphism $\TT_S^F
\isoarrow \TT_S^G$. Explicitly, denoting by $f^\#$ the map on global sections of
$\TT_S$, we have $G(T_1,T_2) = f^\#\bigl[ F \bigl(f^\#{}^{-1}(T_1),
f^\#{}^{-1}(T_2) \bigr)\bigr]$.
\end{rk}

\subsection{The stack of $n$-buds}\label{ss:nbuds=AT_nbsBLT_n}

In this section we show that the moduli stack of $n$-buds \Bn, $n\geq 1$, is
equivalent to the quotient algebraic stack $\sAut(\TT_n)\bs \laz_n$, with the
schemes $\sAut(\TT_n)$ and $\laz_n$ as defined in \eqref{def:AutTT_n} and
\eqref{def:laz_n}, respectively, and with the action of $\sAut(\TT_n)$ on $L_n$
as described in \eqref{rk:AutTTn.acts.on.Ln}.

\begin{prop}\label{st:nbuds=stack}
\Bn is a stack over \Sch for the fpqc topology.
\end{prop}

\begin{proof}
As in \eqref{st:ngerms=stack}, we have to check that morphisms and objects
descend.  The only new ingredient for $n$-buds is the multiplication map.

Descent for morphisms of buds is an immediate consequence of descent for
morphisms of schemes \cite{sga1}*{VIII 5.2} and of compatibility of infinitesimal
neighborhoods \eqref{rk:Inf_bc} and of fibered products with base change, which
imply that a germ morphism locally compatible with the multiplication maps 
is globally compatible.


For objects, suppose we're given an $n$-bud fpqc-locally on $S$ equipped with a
descent datum.  By \eqref{st:ngerms=stack}, we get an $n$-germ $X$ over $S$. 
Again using that infinitesimal neighborhoods and fibered products are compatible
with base change, the local multiplication maps descend to a multiplication
map on $X$, and we're done.
\end{proof}

\begin{thm}\label{st:nbuds=AT_nbsBLT_n}
\Bn is equivalent to the quotient stack $\sAut(\TT_n)\bs \laz_n$.  In
particular, \Bn is algebraic.
\end{thm}

\begin{proof}
We shall apply \cite{lamb00}*{3.8} to the tautological morphism $\pi\colon L_n \to \Bn$.  By \eqref{st:ngermloctriv} and \eqref{eg:trivnbud}, $\pi$ is locally essentially surjective (for the Zariski topology, hence for any finer topology).  Moreover, it is clear from the definitions that the maps 
\[
   \sAut(\TT_n) \times L_n \xrra{\pr_{L_n}}{a} L_n,
\]
where $a$ denotes the action map described in \eqref{rk:AutTTn.acts.on.Ln}, induce an isomorphism 
\[
   \sAut(\TT_n) \times L_n \isoarrow L_n \fib\Bn L_n.
\]
The first assertion now follows, and, as in the proof of \eqref{st:ngerms=B(AT_n)}, the algebraicity assertion is immediate from \cite{lamb00}*{4.6.1}.
\end{proof}

\begin{rk}
As in \eqref{rk:B(Aut(T)).indep.of.top}, we deduce from the theorem that the
quotient stack $\sAut(\TT_n) \bs L_n$ is independent of the topology on \Sch
between the Zariski and fpqc topologies, inclusive.
\end{rk}

\subsection{Ind-infinitesimal sheaves}\label{ss:inf_inf_sheaves}

To prepare for our discussion of the moduli stacks of formal Lie varieties and 
of formal Lie groups in the next two sections, it will be convenient to first
dispense with a few generalities on ind-infinitesimal sheaves \eqref{def:nInf}.
Let $S$ be a scheme, and recall the truncation functor 
\[
   \nInf[\infty](S) \to \nInf(S)
\]
of \eqref{rk:truncation}.
Since infinitesimal neighborhoods are compatible with base change
\eqref{rk:Inf_bc}, truncation preserves pullbacks (up to
isomorphism), so that it specifies a morphism of fibered categories
\[
   \nInf[\infty] \to \nInf.
\]
Given $m \geq n$, the diagram
\[
	\xymatrix{
		{\nInf[\infty]} \ar[r]^-{-^{(m)}} \ar@/_4ex/[rr]_-{-^{(n)}}
			& {\nInf[m]} \ar[r]^-{-^{(n)}}
			& {\nInf}
		}
\]
commutes up to canonical isomorphism (or possibly on the nose, depending on one's choice of definitions).  Similarly, truncation $\nInf \to \nInf$ is canonically isomorphic to the identity functor.  Hence we may form the limit $\ilim_n \nInf$ of the fibered categories $\nInf$, $n\geq 0$, with respect to the truncation functors, and we obtain a natural arrow
\[\tag{$*$}\label{disp:Inf_lim_arrow}
   \nInf[\infty] \to \ilim_n \,\nInf.
\]
Let us emphasize that the limit $\ilim_n \nInf$ is taken in the sense of bicategories; see the appendix for a basic introduction, especially \s\ref{ss:lim_stacks} for limits of fibered categories.

\begin{prop}\label{st:lim_nInf}
The arrow \eqref{disp:Inf_lim_arrow} is an equivalence of fibered categories.
\end{prop}

\begin{proof}
Full faithfulness is an immediate consequence of the definition of
ind-infini\-tesimal and the universal mapping property of a colimit.  For
essential surjectivity, recall (\ref{eg:lims_on_1-cats}, \ref{rk:Fib_obwise})
that an object in $\ilim_n \nInf$ over the scheme $S$ is a family
$(X_n,\varphi_{mn})$, where $X_n \in\ob\nInf(S)$ for all $n$, and
$\varphi_{mn}\colon X_m^{(n)} \isoarrow X_n$ for all $m\geq n$, subject to a
natural cocycle condition for every $l \geq m \geq n$.  For $m \geq n$, consider the composite
\[\tag{$**$}\label{disp:pre_colim}
   X_n \xyra[\sim]{\varphi_{mn}^{-1}} X_m^{(n)} \xinj{\can} X_m.
\]
As $m$ and $n$ vary, the cocycle condition on the $\varphi_{mn}$'s says that the various composites \eqref{disp:pre_colim} are organized into a diagram indexed on the totally ordered set $\NN\cup\{0\}$.  Let us take the colimit sheaf $X := \clim_n X_n$.  Since $X$ is a colimit of a filtered diagram of sheaves, and all transition morphisms in the diagram are monomorphisms, the canonical arrow $X_n \to X$ is a monomorphism for all $n$.  Since $X_n$ is $n$-infinitesimal, $X_n \inj X$ factors through $X^{(n)}$ for all $n$.  It is now easy to verify that $X_n \inj X^{(n)}$ is an isomorphism.  Hence $X$ is ind-infinitesimal.  Moreover, it is easy to verify that the arrows $X_n \isoarrow X^{(n)}$, $n\geq 0$, define an isomorphism from $(X_n,\varphi_{mn})$ to the image of $X$ in $\ilim_n \nInf$.
\end{proof}

%

\subsection{The stack of formal Lie varieties}\label{ss:flv_stack} 


We now come to the moduli stack of formal Lie varieties.  Let $S$ be a scheme.

\begin{defn}
We define $\FLV(S)$ to be the groupoid of formal Lie varieties and isomorphisms over $S$.
\end{defn}

\begin{rk}
It is clear that the base change of a formal Lie variety is again a formal Lie variety.  Hence \FLV defines a CFG over \Sch.  Moreover, it is clear from the definition of formal Lie variety \eqref{def:fv} that \FLV is a stack for the Zariski topology.  
%
In fact, \FLV is a stack for the fpqc topology; this would not be hard to prove in a direct fashion, but we shall deduce it in \eqref{st:FLV=fpqc_stack} from the analogous statement for the stack of $n$-germs \Gn \eqref{st:ngerms=stack}.
\end{rk}

Our first task for this section is to obtain the analog of \eqref{st:ngerms=B(AT_n)} for \FLV.  Recall the formal Lie variety $\wh\AA_S$ of \eqref{eg:whAA}.

\begin{defn}\label{def:Aut_whAA}
We define $\sAut(\wh\AA)$ to be the presheaf of groups on \Sch
\[
   \sAut(\wh\AA)\colon S \mapsto \Aut_{\FLV(S)}(\wh\AA_S).
\]
\end{defn}

We can describe $\sAut(\wh\AA)$ in a way quite analogous to \eqref{st:AutTT=repble}.  Indeed, in analogy with $\TT_{n,S}$, to give an automorphism of $\wh\AA_S$ is to give a power series $a_1T + a_2T^2 + \dotsb \in \Gamma(S)[[T]]$ with $a_1$ a unit.  So we deduce the following.

\begin{prop}\label{st:Aut_whAA=repble}
As a set-valued functor, $\sAut(\wh\AA)$ is canonically represented by the open subscheme $\Spec \ZZ[a_1,a_1^{-1},a_2,a_3\dotsc]$ of $\AA_\ZZ^\infty$.
\hfill $\square$
\end{prop}

\begin{thm}\label{st:FLV=B()}
$\FLV \approx B\bigl(\sAut(\wh\AA)\bigr)$, where the right-hand side denotes the classifying stack with respect to the Zariski topology.
\end{thm}

\begin{proof}
The proof is essentially the same as that of \eqref{st:ngerms=B(AT_n)}:  \FLV is
plainly a gerbe over $\Spec\ZZ$ for the Zariski topology, and $\Spec\ZZ 
\xra{\wh\AA{}_\ZZ} \FLV$ defines a section.
\end{proof}

\begin{rk}\label{rk:B()_top_indep}
Once we see in \eqref{st:FLV=fpqc_stack} that \FLV is a stack for the fpqc topology, it will follow that $B\bigl(\sAut(\wh\AA)\bigr)$ is independent of the choice of topology between the Zariski and fpqc topologies, inclusive.  In particular, every fpqc-torsor for $\sAut(\wh\AA)$ is in fact a Zariski-torsor.
\end{rk}

\begin{rk}\label{rk:FLV_not_alg}
Unlike \Gn \eqref{st:ngerms=B(AT_n)}, \FLV is not algebraic: informally, the 
group $\sAut(\wh\AA)$ is ``too big'', i.e.\ it is not of finite type.
\end{rk}

Let us now turn to the relation between the stacks \FLV and \Gn, $n\geq 0$.  
Recall that the truncation functors induce an equivalence $\nInf[\infty] 
\xra{\approx} \ilim_n\nInf$  \eqref{st:lim_nInf}.  As noted in
\eqref{rk:trunc_germs}, this equivalence restricts to an arrow
\[\tag{$*$}\label{disp:FLV->lim}
   \FLV \to \ilim_n\Gn.
\]
As in the previous section, we emphasize that the limit in the display is taken 
in the sense of bicategories; see the appendix.

\begin{thm}\label{st:FLV=limGn}
The arrow \eqref{disp:FLV->lim} is an equivalence of stacks.
\end{thm}

\begin{proof}
Let $(X_n,\varphi_{mn})$ be an object in $\ilim_n\Gn$, say over the scheme $S$. 
As in the proof of \eqref{st:lim_nInf}, the $\varphi_{mn}$'s organize the $X_n$'s 
into a diagram
\[\tag{$**$}\label{disp:Gn_colim_diag}
   S \ciso X_0 \inj X_1 \inj X_2 \inj \dotsb
\]
indexed on the totally ordered set $\NN\cup\{0\}$, such that the inclusion $X_n \inj X_m$ identifies $X_n \xra[\sim]{\varphi_{mn}^{-1}} X_m^{(n)}$ for all $m\geq n$.   In light of \eqref{st:lim_nInf} and its proof, all we have to show is that the colimit sheaf $X:= \clim_nX_n$ is a formal Lie variety.  

The proof is essentially the same as that of \eqref{st:ngermloctriv}.
Localizing (in the Zariski topology) on $S$, we may assume that $S$ is an affine scheme $\Spec A$ and that $X_1$ has trivial conormal bundle.  We claim that $X \iso \Spf A[[T]]$.  Indeed, by \eqref{st:nbudaff}, we may assume that $X_n = \Spec B_n$, $n\geq 0$, for $B_n$ an augmented $A$-algebra, say with augmentation ideal $I_n$ satisfying $I_n^{n+1} = 0$.  Then \eqref{disp:Gn_colim_diag} translates into a diagram of augmented $A$-algebras
\[
   A \ciso B_0 \twoheadleftarrow B_1 \twoheadleftarrow B_2 \twoheadleftarrow \dotsb
\]
such that $B_m \surj B_n$ identifies $B_m/I_m^{n+1} \isoarrow B_n$ for all $m\geq n$.  Now, by assumption the conormal module $I_1$ (here $I_1^2 = 0$) for $B_1$ is a free $A$-module of rank $1$.  Let $t_1$ be any generator.  Now successively choose a lift $t_2\in B_2$ of $t_1$, a lift $t_3 \in B_3$ of $t_2$, and so on.  Since $I_n/I_n^2 \ciso I_1$ is free for all $n$, the proof of \eqref{st:ngermloctriv} shows exactly that the map
\[
   \xymatrix@R=0ex{
      A[[T]] \ar[r] & B_n\\
      \smash{T}\vphantom{t_n} \ar@{|->}[r] & t_n
   }
\]
induces $A[[T]]/(T)^{n+1}\isoarrow B_n$.  Hence we obtain a topological isomorphism 
\[
   A[[T]] \isoarrow \ilim_n B_n.
\]
The theorem follows.
\end{proof}

\begin{cor}\label{st:FLV=fpqc_stack}
\FLV is a stack over \Sch for the fpqc topology.
\end{cor}

\begin{proof}
Immediate from \eqref{st:ngerms=stack}, \eqref{st:lim_stacks}, and the theorem.
\end{proof}

\begin{rk}
Our definition of formal Lie variety \eqref{def:fv} has a kind of built-in local triviality for the Zariski topology.  Though one may be tempted to consider formulating the local triviality condition with respect to other topologies, the corollary says that the notion of formal Lie variety is independent of the choice of topology for local triviality between the Zariski and fpqc topologies, inclusive.
\end{rk}

\subsection{The stack of formal Lie groups}\label{ss:flg_stack}

Now that we have discussed the stack of formal Lie varieties, we turn to the moduli stack of formal Lie groups.  Let $S$ be a scheme.

\begin{defn}
We define $\flg(S)$ to be the category of formal Lie groups and homomorphisms over $S$.
We define $\FLG(S)$ to be the groupoid of formal Lie groups and isomorphisms over $S$.
\end{defn}

\begin{rk}\label{rk:flg_base_change}
Since formal Lie varieties are stable under base change, it is clear that formal Lie groups are stable under base change.  Hence \flg defines a fibered category, and \FLG a CFG, over \Sch. 
We shall verify in \eqref{st:FLG=fpqc_stack} that \FLG is a stack for the fpqc
topology.  In some sense, this almost follows in a purely formal fashion from
the analogous result for \FLV \eqref{st:FLV=fpqc_stack}.  The only difficulties
arise when trying to glue locally given multiplication maps $X \fib S X \to X$: 
first, strictly speaking, the product of formal Lie varieties is not a
($1$-parameter) formal Lie variety, hence not an object of \FLV; and second,
\FLV only parametrizes \emph{isomorphisms} between objects, not arbitrary
morphisms.  Of course, these difficulties are merely artifacts of our choices to
work only with $1$-parameter formal Lie varieties and only with stacks of
groupoids.  We would have no trouble if we had worked from the outset with
pointed formal Lie varieties in the \emph{general} sense of \eqref{rk:general_fv}, with
no constraint on the number of parameters, and with stacks of \emph{categories}.

In analogy with \eqref{rk:buds.base.change}, given a group law $F$ over
$\Gamma(S)$ and a base change $f\colon S' \to S$, one has $\wh\AA_S^F \fib S S'
\ciso \wh\AA_{S'}^{F'}$, where $F'$ is the
group law over $\Gamma(S')$ obtained by applying $f^\#$ to the coefficients of
$F$.
\end{rk}

Our first goal in this section is to prove the analog of \eqref{st:nbuds=AT_nbsBLT_n} for \FLG.  Recall the formal Lie variety $\wh\AA_S$ of \eqref{eg:whAA}.

\begin{defn}
We define $L$ to be the presheaf of sets on \Sch
\[
   L\colon S \mapsto \{\text{formal Lie group structures on}\ \wh\AA_S\}.
\]
\end{defn}

In analogy with $L_n$ \eqref{def:laz_n}, we have $L \iso \Spec\ZZ[a_1,a_2,\dotsc] = \AA_\ZZ^\infty$ by Lazard's theorem  \eqref{st:Laz.thm}, but the isomorphism is not canonical.  Just as in \eqref{rk:AutTTn.acts.on.Ln}, $\sAut(\wh\AA)$ acts naturally on $L$ as ``changes of coordinate''.  Just as in \eqref{st:nbuds=AT_nbsBLT_n}, we deduce the following.

\begin{thm}
$\FLG \approx \sAut(\wh\AA)\bs L$. \hfill $\square$
\end{thm}

\begin{rk}
It is an immediate consequence of \eqref{rk:B()_top_indep} that the quotient stack $\sAut(\wh\AA)\bs L$ is independent of the choice of topology between the Zariski and fpqc topologies, inclusive.
\end{rk}

\begin{cor}\label{st:FLG=fpqc_stack}
\FLG is a stack for the fpqc topology. \hfill $\square$
\end{cor}

\begin{rk}
In analogy with \eqref{rk:FLV_not_alg}, unlike the moduli stack of $n$-buds \Bn
\eqref{st:nbuds=AT_nbsBLT_n}, \FLG is not algebraic.
\end{rk}

In analogy with the previous section, let us now turn to the relation between
the stacks \FLG and \Bn, $n\geq 0$. By \eqref{rk:trunc_buds}, we may form the limit
$\ilim_n \Bn$ of the $\Bn$'s with respect to the truncation functors, and we
obtain an arrow
\[\tag{$*$}\label{disp:FLG->lim}
   \FLG \to \ilim_n\Bn.
\]
As in previous sections, we emphasize that the limit is taken in the sense of bicategories;
see the appendix.

\begin{thm}\label{st:FLG=limBn}
The arrow \eqref{disp:FLG->lim} is an equivalence of stacks.
\end{thm}

\begin{proof}
We could prove the theorem in a fairly direct fashion entirely analogous to 
\eqref{st:FLV=limGn}.  Instead, we can obtain the theorem as a more-or-less 
formal consequence of \eqref{st:FLV=limGn}.  Indeed, for any base scheme $S$ and
$n\geq 0$, the category $\nInf(S)$ contains finite products 
\eqref{rk:inf_prods}, and truncation between the $\nInf(S)$'s for varying $n$ 
preserves finite products.  Hence passing to the limit of the $\nInf(S)$'s 
commutes with taking commutative group objects \eqref{st:gp_obs_in_lim}.  Now 
use \eqref{st:FLV=limGn} and the identification of $\Bn(S)$ with objects of 
$\Gn(S)$ endowed with a commutative group structure \eqref{st:nbuds=gps}.
\end{proof}

\section{The height stratification: buds}\label{s:ht_strat_buds}

Fix a prime number $p$ once and for all.  We shall now begin to study
the algebraic geometry of the classical notion of
\emph{height} for formal group laws and bud laws over rings of
characteristic $p$.  The essential feature of the theory is a
\emph{stratification}, relative to $p$, on the moduli stack
$\FLG$ of formal Lie groups and on the stacks $\Bn$ of $n$-buds.  We shall begin
by working in the context of buds; we shall turn to formal Lie groups in
\s\ref{s:ht_strat_flg}.

In order to reduce clutter, we shall choose not to embed $p$ in our notation,
though one certainly obtains a different stratification for each choice 
of $p$.
%
%
%

\subsection{Review of formal group laws II}\label{ss:reviewII}

We devote this section to a brief survey of some of the classical algebraic
theory of formal group laws and bud laws related to height, as well as some
somewhat technical refinements to the material in \s\ref{ss:reviewI}.  We shall
use \cite{laz55} and \cite{froh68} as our main references, although there are
many other sources available.  As for \s\ref{ss:reviewI}, the reader may wish to
skip this section and refer back only as needed.

Let $A$ be a ring and $F$ an $n$-bud law, $n\geq 1$, or a formal group law.

\begin{defn}\label{def:[p]_F}
For $m \in \NN$, we define $[m]_F$ to be the element of $A[T]/(T)^{n+1}$ in the bud law case, and of $A[[T]]$ in the group law case,
\[
	[m]_F(T) := F( \cdots F(F(\underbrace{T,T),T), \cdots, T}_{m\ T\text{'s}}).
\]
\end{defn}

In other words, $[1]_F(T) = T$, and $[m+1]_F(T) = F\bigl([m]_F(T), T\bigr)$.

\begin{rk}\label{rk:[p].leading.coeff=p}
Since
\[
   F(T_1,T_2) = T_1 + T_2 + \text{(higher order terms)},
\]
we have
\[
   [m]_F(T) = mT + \text{(higher order terms)}.
\]
\end{rk}

We will be almost exclusively interested in $[p]_F$ for varying bud or group laws $F$.  If $A$ is of characteristic $p$, then we conclude from \eqref{rk:[p].leading.coeff=p} that $[p]_F = 0$ or has lowest degree term of degree $\geq 2$.  In fact, a much stronger statement holds.

\begin{prop}\label{st:[p].law.form}
Suppose $A$ has characteristic $p$.  Let $f\colon F \to G$ be a homomorphism of bud laws or formal group laws over $A$.  Then $f$ is $0$ or of the form $a_{p^h}T^{p^h} + a_{2p^h}T^{2p^h} + \dotsb$ for some nonnegative integer $h$ and some nonzero $a_{p^h}\in A$.  In particular, $[p]_F$ is $0$ or of the form $a_{p^h}T^{p^h} + a_{2p^h}T^{2p^h} + \dotsb$, $a_{p^h} \neq 0$.
\end{prop}

\begin{proof}
\cite{froh68}*{I \s3 Theorem 2(ii)}, which works for bud laws as well as formal
group laws.  Note that $[p]_F$ is plainly a homomorphism $F \to F$.
\end{proof}

The proposition motivates the following definition.  Let $h$ be a nonnegative
integer.

\begin{defn}\label{def:law.height}
$F$ has \emph{height $h$} if $[p]_F$ is of the form $a_{p^h}T^{p^h} +
\text{(higher order terms)}$ for some $a_{p^h} \in A^\times$.
\end{defn}

\begin{rk}
In the bud law case, height $h$ only makes sense for $n \geq p^h$.
\end{rk}

\begin{rk}
By \eqref{rk:[p].leading.coeff=p}, bud laws and group laws of positive height only occur over rings of characteristic $p$.  If $F$ is of positive height $h$, then $F$ is a truncated polynomial in the bud law case, and power series in the group law case, in $T^{p^h}$ \eqref{st:[p].law.form}.  At the other extreme, $F$ has height $0$ $\iff$ $A$ is a $\ZZ[\frac 1 p]$-algebra.
\end{rk}

\begin{rk}
The definition of height in \eqref{def:law.height} differs from the one in the classical literature over rings of characteristic $p$; compare e.g.\ \cite{laz55}*{p.\ 266} or \cite{froh68}*{p.\ 27}.  Explicitly, $F$ has height $h$ in the classical sense if 
\[
   [p]_F(T) = a_{p^h}T^{p^h} + \text{(higher order terms)}
\] 
with $a_{p^h} \neq 0$, whereas we require, in addition, that $a_{p^h}$ be a
\emph{unit}.  In the classical sense, every $F$ is of finite height or satisfies
$[p]_F = 0$.  But in the sense of \eqref{def:law.height}, there certainly exist
$F$ for which $[p]_F \neq 0$ but the height is not defined.  The modified 
version of the classical definition allows one to obtain better behavior with
regard to change of base ring.
\end{rk}

\begin{rk}
In the group law case, it is customary to say that $F$ is of \emph{infinite
height} if $[p]_F = 0$.  For example, over a ring of characteristic $p$, the
additive law $F(T_1,T_2) = T_1 + T_2$ has infinite height.  We remark that this
notion would fit naturally into the geometric theory of formal Lie groups, but
we will not have occasion to introduce it here.
\end{rk}

Before continuing with our discussion of $[p]$, it is now convenient to digress
for a moment on some technical matters underpinning some of the results we
surveyed in \s\ref{ss:reviewI}.  Quite generally, let $P \in A[T_1,T_2]$ be any
polynomial.

\begin{defn}\label{def:SPC}
We say that $P$ is a \emph{symmetric polynomial $2$-cocycle}, or \emph{SPC} for
short, if
\begin{enumerate}
\renewcommand{\theenumi}{S}
\item\label{it:symmetric}
   $P(T_1,T_2) = P(T_2,T_1)$ in $A[T_1,T_2]$; and
\renewcommand{\theenumi}{C}
\item\label{it:cocycle}
   $P(T_2,T_3) - P(T_1+T_2,T_3) + P(T_1,T_2+T_3) - P(T_1,T_2) = 0$ in $A[T_1,T_2,T_3]$.
\end{enumerate}
\end{defn}

\begin{rk}\label{rk:gp_cohom}
To explain the terminology, recall that the group cohomology $H^\bullet(A,A)$
can be computed by the cochain complex $\bigl(C^\bullet(A,A),
\delta^\bullet\bigr)$, where $C^i := C^i(A,A) := \Hom_\Sets(A^i,A)$, $i\geq 0$,
and $\delta^i\colon C^i \to C^{i+1}$ sends
\[
   \begin{split}
   f(a_1,\dotsc,a_i) \overset{\delta^i}\mapsto {}& f(a_2,\dotsc,a_{i+1})\\
      &\qquad+ \sum_{j=1}^i (-1)^j f(a_1,\dotsc,a_{j-1}, a_j + a_{j+1},
               a_{j+2},\dotsc, a_{i+1})\\
      &\qquad\qquad + (-1)^{i+1} f(a_1,\dotsc,a_i).
   \end{split}
\]
For each $i$, the polynomial ring $A[T_1,\dotsc,T_i]$ maps in the obvious way to
$C^i$. Hence \eqref{it:cocycle} in the definition justifies ``polynomial
$2$-cocycle''.  Plainly \eqref{it:symmetric} justifes ``symmetric''.
\end{rk}

\begin{eg}\label{eg:B_m_and_C_m}
For $m\geq 2$, let
\[
   \lambda(m) :=
   \begin{cases}
      l, & \text{$l$ is a prime and $m$ is a power of $l$;}\\
      1, & \text{otherwise.}
   \end{cases}
\]
Then we define Lazard's polynomials $B_m\in \ZZ[T_1,T_2]$ by
\[
   B_m(T_1,T_2) := (T_1+T_2)^m - T_1^m -T_2^m
\]
and $C_m\in \ZZ[T_1,T_2]$ by
\[
   C_m(T_1,T_2) := \frac{1}{\lambda(m)} B_m(T_1,T_2).
\]
Since $\delta^1(-T^m) = B_m$, we deduce at once that $B_m$ and $C_m$ are SPC's
over \ZZ, homogenous of degree $m$.  Hence $B_m$ and $C_m$ define SPC's over
any ring.
\end{eg}

Note that over a ring of characteristic $p$ and for $m$ a power of $p$, we have
$B_m = 0$.  But the SPC $C_m$ is more interesting.  The following is one of
Lazard's most important technical results, valid for any ring $A$ and for any $m
\geq 2$.

\begin{thm}[Lazard]\label{st:mult_C_m}
Let $P$ be a homogenous SPC of degree $m$ over $A$.  Then there exists a unique
$a\in A$ such that $P = aC_m$.
\end{thm}

\begin{proof}
It is easy to compute directly that $C_m$ is \emph{primitive}, which implies the
unicity of $a$.  Existence of $a$ is much harder; see \cite{laz55}*{Lemme 3} or
\cite{froh68}*{III \s1 Theorem 1a}.
\end{proof}

SPC's occur in the theory of formal group laws and bud laws in the following
simple way; we shall encounter them in a somewhat different context in
\s\ref{ss:aut_buds_ht_h}. Let $m\geq 2$, and let $F$ and $G$ be group laws or
$n$-bud laws over $A$
whose respective $(m-1)$-truncations $F^{(m-1)}$ and $G^{(m-1)}$
\eqref{rk:trunc_buds} are equal, $m\leq n-1$ in the bud case.  Then a fairly
simple calculation shows that $F^{(m)}$ and $G^{(m)}$ differ by a homogenous SPC
of degree $m$.  Hence \eqref{st:mult_C_m} implies the following.

\begin{cor}[Lazard]\label{st:laz.comp.thm}
There exists a unique $a\in A$ such that
\begin{xxalignat}{3}
\phantom{\square}& & F^{(m)} &= G^{(m)} + aC_m. & \square
\end{xxalignat}
\end{cor}


The corollary plays an important role in Lazard's construction of a universal
group law $U$ over $\ZZ[t_1,t_2,\dotsc]$ we cited in \eqref{st:Laz.thm}.  We
will use the following refinement of \eqref{st:Laz.thm} for calculations in
\s\ref{ss:ht_>=h_buds}. It describes the group law constructed by Lazard in the
proofs of
\cite{laz55}*{Th\'eor\`emes II and III}.

\begin{thm}[Lazard]\label{st:Laz.thm.refinement}
The universal formal group law $U(T_1,T_2)$ of \eqref{st:Laz.thm} may be chosen 
such that for all $n\geq 1$,
\begin{enumerate}
\renewcommand{\theenumi}{\roman{enumi}}
\item\label{it:U^(n)_coeffs}
   the coefficients of the truncation $U^{(n)}$ involve only $t_1,\dotsc,
t_{n-1}$;
\item\label{it:U^(n)_univ}
   $U^{(n)}$, regarded as defined over $\ZZ[t_1,\dotsc,t_{n-1}]$ by \eqref{it:U^(n)_coeffs}, is a universal $n$-bud law; and
\item\label{it:U^(n)_smoothness}
   there is an equality of elements in $\ZZ[t_1,\dotsc,t_{n-1},s][T_1,T_2]/
(T_1,T_2)^{n+1}$,
\begin{xxalignat}{3}
\phantom{\square}& &
\begin{aligned}[b]
   U^{(n)}(t_1,\dotsc,t_{n-2}, t_{n-1} + s)(T_1,T_2)
     &- U^{(n)}(t_1,\dotsc,t_{n-1},)(T_1,T_2)\\
     &\qquad\qquad\qquad\qquad\qquad = sC_n(T_1,T_2).
\end{aligned}
& &\square
\end{xxalignat}
\end{enumerate}
\end{thm}


We now return to our discussion of $[p]$.  We will use the following lemma to
study the coefficients of $[p]$ in \s\ref{ss:ht_>=h_buds}. Let $F$, $G$, and
$m$ be as in
\eqref{st:laz.comp.thm}, so that $F^{(m)} = G^{(m)} + aC_{m}$ for a unique $a\in
A$.

\begin{lem}\label{st:[p].difference.lem}
For any $k\geq 1$,
\[
   [k]_{F^{(m)}}(T) = [k]_{G^{(m)}}(T) + \frac{k^m-k}{\lambda(m)} \cdot aT^m.
\]
In particular, for $k=p$ and $m$ of the form $p^h$, we have
\[
   [p]_{F^{(p^h)}}(T) = [p]_{G^{(p^h)}}(T) + \bigl(p^{p^h-1} - 1\bigr)
aT^{p^h}.
\]
\end{lem}

\begin{proof}
\cite{laz55}*{Lemme 6} or \cite{froh68}*{III \s1 Lemma 4}.
\end{proof}

In later sections, we will be especially interested in group laws of a given
height for which $[p]$ is as simple as possible.
Fix $h\geq 1$.

\begin{thm}\label{st:[p]_H=T^p^h.over.Fp}
There exists a formal group law $H_h$ over $\FF_p$ such that $[p]_{H_h}(T) = T^{p^h}$.
\end{thm}

\begin{proof}
See, for example, \cite{froh68}*{III \s2 Theorem 1}.
\end{proof}

\begin{ntn}\label{ntn:H}
We fix once an for all a formal group law $H_h$ as in \eqref{st:[p]_H=T^p^h.over.Fp}.  As we shall often work with respect to fixed $h$, we often write just $H$ for $H_h$, provided no confusion seems likely.
\end{ntn}

If $A$ is of characteristic $p$, then we may view $H$ as defined over $A$.  Let $n\geq p^h$, and consider the $n$-truncation $H^{(n)}$.

\begin{prop}\label{st:ht.h.fin.et}
Let $F$ be an $n$-bud law (resp., formal group law) of height $h$ over $A$.
Then there exists a finite \'etale extension ring (resp., a countable 
ascending union of finite \'etale extension rings) $B$ of $A$ such that $F 
\iso H^{(n)}$ (resp., $F \iso H$) over $B$.
\end{prop}

\begin{proof}
We'll proceed by extracting some arguments from the proofs of the statements 
leading up to the proof of Theorem 2 in \cite{froh68}*{III \s2}.  There is also
a somewhat cleaner version of the proof sketched in \cite{mil03}*{10.4}.

It suffices to consider the bud law case; the group law case then follows by
considering the various truncations $F^{(n)}$ for higher and higher $n$.  To
begin, the proof of \cite{froh68}*{III \s2 Lemma 3} shows that there exists a
finite \'etale extension ring $B$ of $A$ such that $F$ is isomorphic over $B$ to
an $n$-bud law $G$ for which $[p]_G(T) = T^{p^h}$.  A bit more explicitly, 
whereas \cite{froh68} takes $A$ to be a separably closed field and proceeds by 
finding solutions to certain (separable) equations in $A$, one can proceed in 
our case by formally adjoining solutions to certain (separable) equations to 
$A$, that is, one can obain $B$ as an iterated extension ring of (finitely many)
rings of the form $A[X]$ modulo a separable polynomial.

The remaining step is to show that over any ring, any two bud laws for which
$[p](T) = T^{p^h}$ are isomorphic.  This is probably best and most simply
seen via a direct argument, but it can be gleaned from \cite{froh68}*{III \s2}
by  combining arguments (suitably adapted to the bud case) in the proofs of Lemma 2, Proposition 3, and Theorem 2.
\end{proof}

%
%
%
%

\begin{rk}
In particular, the ring $B$ in the theorem is faithfully flat over $A$.
\end{rk}

\subsection{Multiplication by $p$}\label{ss:mult.by.p}

Let $X$ be an $n$-bud or formal Lie group over the scheme $S$.  Then (using \eqref{st:nbudmorphs=abgp} in the bud case) the endomorphisms of $X$ form an abelian group.  So we may make the following definition.

\begin{defn}\label{def:[p]_X}
We define $[p]_X$ to be the element of $\End_{\nbuds(S)}(X)$ in the bud case, and of $\End_{\flg(S)}(X)$ in the formal Lie group case, $p\,\id_X$.
\end{defn}

\begin{rk}\label{rk:[p].for.triv.buds}
When $X = \TT_S^F$ for an $n$-bud law $F \in \Gamma_n(S;T_1,T_2)$ \eqref{eg:trivnbud}, the morphism
\[
	\xymatrix{
		X \ar[r]^-{[p]_X} & X
		}
\]
corresponds to the $\O_S$-algebra map
\[
	\xymatrix@R=0ex{
		\O_S[T]/(T)^{n+1} & \O_S[T]/(T)^{n+1} \ar[l] \\
		\smash{[p]_F(T)}\vphantom{(} & \smash{T,}\vphantom{(} \ar@{|->}[l]
		}
\]
with $[p]_F$ as defined in \eqref{def:[p]_F}.  There is an obvious analogous statement when $F$ is a group law.
\end{rk}

\begin{rk}\label{rk:[p].compat.truncation}
One checks at once that truncation functors are
additive functors on the category of formal Lie groups and on the various categories of buds.  Hence truncation preserves
$[p]$: if $X$ is a formal Lie group or an $m$-bud with $m\geq n$, then $[p]_X^{(n)} = [p]_{X^{(n)}}$.
\end{rk}

\begin{rk}\label{rk:[p].compat.bc}
Similarly, consider the category of $n$-buds or of formal Lie groups over $S$. Then for any morphism
$f\colon S' \to S$, the base change functor $-\fib S S'$ is additive, hence
preserves $[p]$.  When $X = \TT_S^F$ or $X = \wh\AA_S^F$, recall that the multiplication law on $X
\fib S S'$ is given by the law $F'$ obtained by applying $f^\#$ to the
coefficients of $F$ (\ref{rk:buds.base.change}, \ref{rk:flg_base_change}).  Hence $[p]_{F'}$ is obtained
by applying $f^\#$ to the coefficients of $[p]_F$.
\end{rk}

\subsection{Zero loci of line bundles}\label{ss:lb's}

In the next section we'll describe the height stratification on \Bn as arising
from a succession of \emph{zero loci of sections of line bundles}.  Our aim in
this section is to dispense with a few of the basic preliminaries.  The material
we shall discuss is general in nature and is independent of our earlier work.

Let \Vect[1] denote the fibered category on \Sch that assigns to each scheme $S$
the category of locally free $\O_S$-modules of rank $1$ and all module
homomorphisms (with the usual pullback functors).  Then \Vect[1] is an fpqc
stack \cite{sga1}*{VIII 1.12}.  Note that the underlying moduli stack of
\Vect[1], obtained by discarding the non-Cartesian morphisms, is just
$B(\GG_m)$.  Let \F be a fibered category over \Sch.

\begin{defn}\label{def:lb}
A \emph{line bundle} on \F is a $1$-morphism $\L \colon \F \to \Vect[1]$ in 
$\Fib\bigl(\Sch\bigr)$.  A \emph{morphism} $\L \to \L'$ of line bundles on \F is
a $2$-morphism $\L \to \L'$ in $\Fib\bigl(\Sch\bigr)$.
\end{defn}

When \F is an algebraic stack, one recovers the usual notion of line bundle on
\F; see \cite{lamb00}*{13.3} (strictly speaking, \cite{lamb00} would take 
$\Vect[1](S)$ to be the \emph{opposite} of the category of locally free 
$\O_S$-modules of rank $1$, but let us not belabor this point).

\begin{eg}\label{eg:O_X}
For any fibered category \F, we denote by $\O_\F$ the line bundle that assigns
to each $X\in\ob\F$ over the scheme $S$ the trivial line bundle $\O_S$,
and to each morphism $\mu\colon Y \to X$ over $f\colon T \to S$ the Cartesian morphism in \Vect[1] corresponding to the canonical isomorphism $\O_T \isoarrow f^* \O_S$.
\end{eg}

Let \L be a line bundle on  the fibered category \F.

\begin{defn}\label{def:sectionoflb}\hfill
\begin{enumerate}
\renewcommand{\theenumi}{\alph{enumi}}
\item
   A \emph{global section}, or just \emph{section, of \L} is a morphism
   $a\colon \O_\F \to \L$.
\item
   Given a section $a$ of \L, the \emph{zero locus of $a$} is the full
   subcategory  $V(a)$ of \F whose objects $X$ over the scheme $S$ are those
   for which $a_X\colon \O_S \to \L_X$ is the $0$ map.
\end{enumerate}
\end{defn}

Let $a$ be a section of \L.  The basic result of interest to us is the
following.  The proof is straightforward.

\begin{prop}\label{st:V(a)=clsdimm}\hfill
\begin{enumerate}
\renewcommand{\theenumi}{\roman{enumi}}
\item
   $V(a)$ is a sub-fibered category of \F, and the inclusion functor 
   $V(a) \to \F$ is a closed immersion.
\item
   If \F is a CFG, or stack, or algebraic stack, then so is $V(a)$. \hfill
   $\square$
\end{enumerate}
\end{prop}

\subsection{The height stratification on the stack of $n$-buds}
\label{ss:nbudheight}

In this section, we translate the classical notion of height for bud laws to
the geometric setting by defining the \emph{height stratification} on the stack
\Bn.  We shall define the analogous stratification on the stack of formal Lie
groups in \s\ref{ss:ht_strat_flg}.

Let $X$ be an $n$-bud over the scheme $S$ with section $\sigma\colon S \to X$.
Let
\[
   \I_X := \ker\bigl[\sigma^{-1}\O_X  \xra{\sigma^\sharp} \O_S\bigr].
\]
Then the endomorphism $[p]_X$ of $X$ \eqref{def:[p]_X} determines an
endomorphism $[p]_{\I_X}$ of $\I_X$.

Now let $h$ be a nonnegative integer, and assume $n \geq p^h$.  

\begin{defn}\label{def:bud.ht.>=h}
We say $X$ has \emph{height $\geq h$} if the endomorphism $[p]_{\I_X}\colon \I_X \to
\I_X$ has image in $\I_X^{p^h}$.  We denote by $\Bn^{\geq h}$ the full
subcategory of \Bn of $n$-buds of height $\geq h$.
\end{defn}

\begin{eg}\label{eg:triv_bud_ht>=h}
Let $X = \TT_S^F$ for some $n$-bud law $F$ over $\Gamma(S)$ \eqref{eg:trivnbud}.
Then we see from \eqref{rk:[p].for.triv.buds} that $X$ has height $\geq h$
$\iff$ $[p]_F \in T^{p^h} \cdot \Gamma_n(S; T)$.  So our terminology is
compatible with \eqref{def:law.height}.
\end{eg}

\begin{rk}
Recall that $[p]$ \eqref{rk:[p].compat.bc}
is compatible with base
change. Hence height $\geq h$
is stable under base change.  Hence $\Bn^{\geq h}$ is fibered over \Sch.
\end{rk}

\begin{rk}\label{rk:ht.>=h.compat.truncation}
Similarly, height $\geq h$ is stable under truncation (provided we don't truncate below $n = p^h$).  More precisely, $X$ has height $\geq h$ $\iff$ $X^{(p^h)}$ has height $\geq h$.
\end{rk}

\begin{rk}
Of course, for fixed $n$, there are only finitely many values (possibly none) of $h$ for which height $\geq h$ makes sense.  So we get a finite decreasing chain $\Bn = \Bn^{\geq 0} \ctnsneq \Bn^{\geq 1} \ctnsneq \Bn^{\geq 2} \ctnsneq \dotsb$.
\end{rk}

We shall see in a moment that the inclusion $\Bn^{\geq h} \inj \Bn$ is a
\emph{closed immersion}.  First, some notation.

By definition of $n$-germ \eqref{def:ngerm} for $n\geq 1$, the conormal sheaf $\I_X/\I_X^2$
associated to the section $S \to X$ is a line bundle on $S$. Moreover, since the
conormal sheaf associated to a section is compatible with base change on $S$
\cite{egaIV4}*{16.2.3(ii)}, formation of the conormal bundle defines a line
bundle on \Bn.

\begin{defn}\label{def:L}
We denote by \L the line bundle on \Bn associating to each bud its associated conormal sheaf.  We define $\L_h := \L|_{\Bn^{\geq h}}$.
\end{defn}

\begin{rk}
Clearly, the same construction defines a line bundle $\L'$ on the stack \Gn
of $n$-germs, and \L is just the pullback of $\L'$ along the natural forgetful
morphism $\Bn \to \Gn$.
\end{rk}

\begin{rk}\label{rk:L_indep_of_n}
Strictly speaking, \L and $\L_h$ depend on $n$.  But the dependence on $n$ is
largely superficial:  since the conormal sheaf of an immersion
depends only on the $1$st infinitesimal neighborhood, formation
of \L is compatible with truncation $\Bn[m] \to \Bn$, and similarly for $\L_h$.
In other words, up to canonical isomorphism, we may construct \L on $\Bn$ by
first constructing \L on $\Bn[1]$ and then pulling back along $\Bn \to \Bn[1]$; and similarly for $\L_h$, replacing $\Bn[1]$ by $\Bn[p^h]$.
So, to avoid clutter, we shall abuse notation and simply suppress the $n$ when
writing \L and $\L_h$.
\end{rk}

When $X$ has height $\geq h$, the map $[p]_{\I_X}\colon \I_X \to \I_X^{p^h}$
induces a map of line bundles
\[\tag{$*$}\label{disp:pre-v's}
   \I_X/\I_X^2 \to \I_X^{p^h}/\I_X^{p^h+1}.
\]
But $\I_X^{p^h}/\I_X^{p^h+1} \ciso (\I_X/\I_X^2)^{\tensor p^h}$
\eqref{st:freesuccquots}.  So we may express \eqref{disp:pre-v's} as
\[
	(\L_h)_X \to (\L_h)_X^{\tensor p^h},
\]
or as a section
\[\tag{$**$}\label{disp:(v_h)_X}
	(v_h)_X \colon \O_S \to (\L_h)_X^{\tensor p^h-1}.
\]
Since $[p]$ is compatible with pullbacks, we may make the following
definition.

\begin{defn}\label{def:v_h}
We denote by $v_h$ the section of $\L_h^{\tensor p^h-1}$ defined by
\eqref{disp:(v_h)_X}.
\end{defn}

\begin{eg}\label{eg:a_0}
By \eqref{rk:[p].leading.coeff=p}, $v_0$ is just the section $p$ of $\O_{\Bn}$.
\end{eg}

\begin{eg}\label{rk:triv_bud_v_h}
Of course, we've taken pains to express $v_h$ in a coordinate-free way, so that
it is, in some sense, canonical.  But when $X$ admits a coordinate, $\L_X$
is trivial, and $v_h$ can be understood more concretely.  Precisely, suppose $X
= \TT_S^F$ for some $n$-bud law $F$ over $\Gamma(S)$.  Then $\I_X = T\cdot
\O_S[T]/(T)^{n+1}$, and there is an obvious trivialization
\[
   \xymatrix@R=0ex{
      \O_S \ar[r]^-\sim  & \I_X/\I_X^2 = \L_X\\
      1 \ar@{|->}[r]     & \smash[b]{\text{image of}}\ T.
   }
\]
The displayed trivialization induces a natural trivialization of
$\L_X^{\tensor p^h-1}$.

Now suppose $X$ has height $\geq h$, so that  $[p]_F$ is of the form
\[
   a_{p^h}T^{p^h} + \text{(higher order terms)} \qquad \eqref{eg:triv_bud_ht>=h}.
\]
Then, under our trivialization of $(\L_h)_X^{\tensor p^h-1} = \L_X^{\tensor
p^h-1}$, $(v_h)_X$ corresponds exactly to $a_{p^h} \in \Gamma(S)$.
\end{eg}

\begin{rk}\label{rk:v_h_indep_of_n}
Just as for $\L_h$, $v_h$ implicitly depends on $n$.  But $v_h$ is essentially independent of $n$ in the same sense as $\L_h$ is \eqref{rk:L_indep_of_n}.
\end{rk}

\begin{prop}\label{st:nbuds^>=h=V(a_h)}
Assume $h \geq 1$.  Then $\Bn^{\geq h}$ is the zero locus
\eqref{def:sectionoflb} in $\Bn^{\geq h-1}$ of the section $v_{h-1}$.
\end{prop}

\begin{proof}
Let $X$ be an $n$-bud of height $\geq h-1$.  We claim
\begin{align*}
   X \in \ob V(v_{h-1})
      &\iff (v_{h-1})_X = 0\\
      &\iff \text{$[p]_{\I_X}$ carries $\I_X$ into $\I_X^{p^{h-1}+1}$}\\
      &\overset{(\dag)}{\iff} \text{$[p]_{\I_X}$ carries $\I_X$ into
                   $\I_X^{p^h}$}\\
      &\iff X\in\ob \Bn^{\geq h}.
\end{align*}
Only the ``$\Longleftrightarrow$'' marked $(\dag)$ requires proof.
For this, the implication $\Longleftarrow$ is trivial.  For the implication
$\Longrightarrow$, the assertion is local on $S$, so we may assume $X$ is
of the form $\TT_S^F$ \eqref{eg:trivnbud}.  Now use \eqref{st:[p].law.form} and
\eqref{rk:[p].for.triv.buds}.
\end{proof}

\begin{cor}\label{st:ht.>=h.clsd.imm}
$\Bn^{\geq h}$ is an algebraic stack for the fpqc toplogy, and the inclusion functor $\Bn^{\geq h} \to \Bn$
is a closed immersion.
\end{cor}

\begin{proof}
Apply \eqref{st:nbuds=stack}, \eqref{st:nbuds=AT_nbsBLT_n}, 
\eqref{st:V(a)=clsdimm}, and the proposition.
\end{proof}

\begin{rk}\label{rk:Bn^>=1=Bn.x.FFp}
Combining \eqref{eg:a_0} and \eqref{st:nbuds^>=h=V(a_h)}, we see that $\Bn^{\geq
1}$ is precisely the stack of $n$-buds over $\FF_p$-schemes.
\end{rk}

\begin{rk}\label{rk:>=h_refinement}
The proposition says that the property of height $\geq h$ depends only on a
bud's $p^{h-1}$-truncation.  So we could extend the notion of height $\geq h$
to $n$-buds for $n \geq p^{h-1}$, but this added bit of generality would offer
no real advantage to us.
\end{rk}

\subsection{The stack of height $\geq h$ $n$-buds}\label{ss:ht_>=h_buds}

Let $h\geq 1$ and $n \geq p^h$.  Our aim in this section is to obtain a
description of $\Bn^{\geq h}$ \eqref{def:bud.ht.>=h} analogous to the
description of \Bn in \eqref{st:nbuds=AT_nbsBLT_n}.

As a warm-up, note first that the case $h = 1$ is easy: $\Bn^{\geq 1} \approx \Bn \tensor \FF_p$ by \eqref{rk:Bn^>=1=Bn.x.FFp}, so $\Bn^{\geq 1} \approx \sAut(\TT_n)_{\FF_p} \bs (L_n)_{\FF_p}$ by \eqref{st:nbuds=AT_nbsBLT_n}, where $(L_n)_{\FF_p} := L_n\tensor \FF_p \iso \AA_{\FF_p}^{n-1}$ and $\sAut(\TT_n)_{\FF_p} := \sAut(\TT_n) \tensor \FF_p$ is an open subscheme of $\AA_{\FF_p}^n$.


To treat the case of general $h$, let us begin by constructing a convenient
presentation of $\Bn^{\geq h}$. Let $A$ be the polynomial ring
$\ZZ_{(p)}[t_1,\dotsc,t_{n-1}]$, and let $F$ be a universal (for
$\ZZ_{(p)}$-algebras) $n$-bud law over $A$ such that the truncation $F^{(n')}$
satisfies \eqref{it:U^(n)_coeffs}--\eqref{it:U^(n)_smoothness} in
\eqref{st:Laz.thm.refinement} for all $n'\leq n$. Let $a_h$ denote the
coefficient of $T^{p^h}$ in
$[p]_F(T) \in A[T]/(T)^{n+1}$ \eqref{def:[p]_F}.
%
%
%
%
%
%
%
For $h\geq 1$, we define
\begin{itemize}
\item
   $A_h := A/(p,a_1,\dotsc,a_{h-1})$;
\item
   $F_h := \text{the reduction of $F$ over $A_h$}$; and
\item
   $X_h := \TT_{A_h}^{F_h}$.
\end{itemize}
By \eqref{st:[p].law.form} and \eqref{eg:triv_bud_ht>=h}, $X_h$ has height
$\geq h$.  So $X_h$ determines (up to isomorphism) a classifying map
\[\tag{$*$}\label{disp:Ah_pres}
   \Spec A_h \xra{X_h} \Bn^{\geq h}.
\]

\begin{thm}\label{Bn^>=h_quot}
$\sAut(\TT_n)_{\FF_p}$ acts naturally on $\Spec A_h$, and \eqref{disp:Ah_pres} identifies 
\[
   \sAut(\TT_n)_{\FF_p} \bs \Spec A_h \approx \Bn^{\geq h}.
\]
\end{thm}

\begin{proof}
The theorem is clear for $h = 1$, since plainly $X_1$ identifies $\Spec A_1 \isoarrow (L_n)_{\FF_p}$.  Now, for any $h$, the theorem is equivalent by \cite{lamb00}*{3.8} to
\begin{itemize}
\item \eqref{disp:Ah_pres} is locally essentially surjective;
\item $\sAut(\TT_n)_{\FF_p}$ acts on $\Spec A_h$; and
\item the $\sAut(\TT_n)_{\FF_p}$-action induces
   \[
       \sAut(\TT_n)_{\FF_p} \times \Spec A_h 
          \ciso \Spec A_h \fib{\Bn^{\geq h}} \Spec A_h.
   \]
\end{itemize}
So we need note just that the diagram
\[
   \vcenter{
   \xymatrix{
      \Spec A_{h+1} \ar@{^{(}->}[r] \ar[d]_-{X_{h+1}}
         & \Spec A_{h} \ar[d]^-{X_{h}}\\
      \Bn^{\geq h+1} \ar@{^{(}->}[r]
         & \Bn^{\geq h}
      }
      }
\]
is Cartesian by \eqref{rk:triv_bud_v_h} and \eqref{st:nbuds^>=h=V(a_h)}, and everything follows by induction on $h$.
\end{proof}

\begin{rk}\label{rk:univ_ht_>=h_bud}
Though we didn't need it explicitly in the proof of the theorem, the ring $A_h$ and the bud law $F_h$ admit an obvious modular interpretation:  namely, $F_h$ is \emph{universal} amongst $n$-bud buds of height $\geq h$.  On the other hand, one could use universality of $F_h$ to give an alternative proof of the theorem quite along the lines of \eqref{st:nbuds=AT_nbsBLT_n}, without appealing to induction.
\end{rk}

The theorem leads us to consider closely the ring $A_h$.  The essential observation is the following result on $A$.

\begin{prop}\label{st:poly_rg}
The map of polynomial rings $\ZZ_{(p)}[u_1,\dotsc,u_{n-1}] \to A$ determined by
\[\label{disp:u_i_mapsto}\tag{$**$}
   u_i \mapsto
   \begin{cases}
      a_h & i = p^h-1\ \text{for $h = 1$, $2,\dotsc$;}\\
      t_i & \text{otherwise}
   \end{cases}
\]
is an isomorphism.  In particular, $A/(a_1,\dotsc,a_{h-1})$ is a polynomial ring
over $\ZZ_{(p)}$ on $n-h$ variables.
\end{prop}

\begin{proof}
Combining \eqref{st:[p].difference.lem} with the form of $F$ described in
\eqref{st:Laz.thm.refinement}, we deduce that
\[
   a_h(t_1,\dotsc,t_{p^j-2},t_{p^h-1} + s) - a_h(t_1,\dotsc,t_{p^j-1})
   = (p^{p^h-1}-1)s
\]
as elements in $\ZZ_{(p)}[t_1,\dotsc,t_{p^h-1},s]$.  Hence
\[
   a_h(t_1,\dotsc,t_{p^h-1}) = (p^{p^h-1}-1)t_{p^h-1} + (\text{terms involving}\
t_1,\dotsc,t_{p^h-2}).
\]
But, for $h\geq 1$, $p^{p^h-1}-1$ is a \emph{unit} in any $\ZZ_{(p)}$-algebra.
The proposition now follows easily.
\end{proof}

\begin{cor}\label{st:A_h=aff}
$\Spec A_h \iso \AA_{\FF_p}^{n-h}$.\hfill $\square$
\end{cor}

\begin{rk}\label{rk:univ_ht_>=h_law}
There are natural analogs of \eqref{st:poly_rg} and \eqref{st:A_h=aff} in the group law setting: if
$U$ is a universal
(for $\ZZ_{(p)}$-algebras) formal group law over $\ZZ_{(p)}[t_1,t_2,\dotsc]$ of
the form described in \eqref{st:Laz.thm.refinement}, and we again denote by
$a_h$ the coefficient of $T^{p^h}$ in $[p]_U(T)$, then
\begin{itemize}
\item
   the map
$\ZZ_{(p)}[u_1,u_2,\dotsc] \to \ZZ_{(p)}[t_1,t_2,\dotsc]$ specified by
\eqref{disp:u_i_mapsto} is an isomorphism of polynomial rings; and
\item
   $\ZZ_{(p)}[t_1,t_2,\dotsc]/(a_1,\dotsc,a_{h-1})$ is a polynomial ring over
$\ZZ_{(p)}$ on the images of the $t_i$ for $i\neq p^1-1$, $p^2-1,\dotsc$,
$p^{h-1}-1$.
\end{itemize}
Moreover, in analogy with \eqref{rk:univ_ht_>=h_bud}, the reduction of $U$ over
\[
   \ZZ_{(p)}[t_1,t_2,\dotsc]/(p,a_1,\dotsc,a_{h-1})
      \ciso \FF_p[t_1,t_2,\dotsc]/(\ol a _1,\dotsc,\ol a _{h-1}),
\]
where $\ol a _i$ denotes the reduction of $a_i$ mod $p$, is plainly of height
$\geq h$, and indeed \emph{universal} amongst group laws of
height $\geq h$.
\end{rk}

\begin{rk}
It would also be possible to obtain \eqref{st:A_h=aff} essentially from
Landweber's classification of invariant prime ideals in $MU_*$
\cite{land73}*{2.7}, or by considering \emph{$p$-typical} group laws over $\ZZ_{(p)}$.
We have chosen our approach for its more direct emphasis on elementary 
properties of $[p]$.
\end{rk}

%
%
%

We will now conclude the section by applying some of our results thus far to
considerations of smoothness and dimension of the algebraic stacks \Bn and
$\Bn^{\geq h}$. We shall use freely the language of \cite{lamb00}, but let us
state here explicitly the notion of \emph{relative dimension} of a morphism. We
will not (and \cite{lamb00} does not) attempt to define the relative dimension
of an arbitrary locally finite type morphism of algebraic stacks $f\colon \X \to
\Y$. We can, however, give a satisfactory definition when $f$ is smooth. Indeed,
if $\xi$ is a point of \X \cite{lamb00}*{5.2}, then let $\Spec L \to \Y$ be any
representative of $f(\xi)$, set $\X_L := \Spec L \fib\Y \X$, and let $\wt\xi$ be
any point of $\X_L$ lying over $\xi$. Then $\X_L$ is a locally Noetherian
algebraic stack, and the \emph{relative dimension of $f$ at $\xi$} is
the integer $\dim_\xi f := \dim_{\wt\xi} \X_L$ \cite{lamb00}*{11.14}. It is
straightforward to verify that the definition is independent of the choices
made.

\begin{thm}\label{st:nbuds_sm}
For $n \geq 1$, \Bn is smooth over $\Spec\ZZ$ of relative dimension $-1$ at every point.  For $h\geq 1$ and $n \geq p^h$, $\Bn^{\geq h}$ is smooth over $\Spec \FF_p$ of relative dimension $-h$ at every point.
\end{thm}

\begin{proof}
The assertion for $\Bn$ is immediate from the definitions and from $\Bn\approx 
\sAut(\TT_n)\bs \laz_n$ \eqref{st:nbuds=AT_nbsBLT_n}, since $\laz_n \iso 
\AA_\ZZ^{n-1}$ and $\sAut(\TT_n)$ is an open subscheme of $\AA_\ZZ^n$ 
\eqref{st:AutTT=repble}.  The assertion for $\Bn^{\geq h}$ is similarly
immediate from \eqref{Bn^>=h_quot} and \eqref{st:A_h=aff}.
\end{proof}

\subsection{The stratum of height $h$ $n$-buds}\label{ss:ht_h_buds}

In this section we study the \emph{strata} of the height stratification on \Bn, or, in other words, the notion of (exact) height for buds.  Let $h \geq 1$ and $n \geq p^{h+1}$.

\begin{defn}\label{def:bud.ht.h}
We denote by $\Bn^h$ the algebraic stack obtained as the open complement of
$\Bn^{\geq h+1}$ in $\Bn^{\geq h}$.  We call the objects of $\Bn^h$ the $n$-buds
of \emph{height $h$}, or sometimes of \emph{exact height $h$}.
\end{defn}

\begin{rk}\label{rk:ht.h.compat.truncation}
Since formation of infinitesimal neighborhoods is compatible with base change,
it's clear from the definitions that height $h$ is stable under truncation
(provided we don't truncate below $n = p^{h+1}$).  More precisely, an $n$-bud $X$ has height $h$ $\iff$ $X^{(p^{h+1})}$ has height $h$.
\end{rk}

Let $S$ be a scheme. We can give a more concrete description of the $n$-buds of
height $h$ over $S$ as follows. Since $\Bn^h$ is a stack, height $h$ is a local
condition. So, by \eqref{st:ngermloctriv}, the essential case to consider is the
$n$-bud $\TT_S^F$ \eqref{eg:trivnbud} for some bud law $F$ over $\Gamma(S)$. In
this case, we have the pleasing result that the notion of height in
\eqref{def:bud.ht.h} agrees exactly with that for bud laws
\eqref{def:law.height}.

\begin{prop}\label{st:hthtrivgerm}
$\TT_S^F$ has height $h$ $\iff$ $[p]_F$ \eqref{def:[p]_F} is of the form
\[
   [p]_F(T) = a_{p^h}T^{p^h} + a_{2p^h}T^{2p^h} + \dotsb,
      \quad a_{p^h} \in \Gamma(S)^\times.
\]
\end{prop}

\begin{proof}
Let $X := \TT_S^F$.  By \eqref{st:[p].law.form} and \eqref{eg:triv_bud_ht>=h}, $X$ has height $\geq
h$ $\iff$ $[p]_F$ is of the asserted form, only with no constraint on $a_{p^h}$.
Now, \eqref{rk:triv_bud_v_h} and \eqref{st:nbuds^>=h=V(a_h)} say that the closed
subscheme $Z := \Spec\O_S/ a_{p^h}\O_S$ of $S$ is \emph{universal} amongst all
$S$-schemes $S'$ with the property that $S' \fib S X$ is an $n$-bud of height
$\geq h+1$.  So $X$ is of height $h$ $\iff$ $\Spec \O_S/ a_{p^h}\O_S =
\emptyset$ $\iff$ $a_{p^h}$ is a unit.
\end{proof}

\begin{rk}
In our discussion of line bundles in \s\ref{ss:lb's}, we could have introduced
the notion of principal open substack associated to a section of a line bundle
as a sort-of complement to the notion of zero locus of a section of a line
bundle.  Then \eqref{st:hthtrivgerm} would show, in essence, that $\Bn^h$ can be
obtained as the principal open substack in $\Bn^{\geq h}$ associated to $v_h$
\eqref{def:v_h}.  The details are straightforward, but we shall decline to
pursue them further.
\end{rk}

\begin{rk}\label{rk:ht_h_need_n>=p^h}
In analogy with \eqref{rk:>=h_refinement}, \eqref{st:hthtrivgerm}
allows us to extend the notion of exact height $h$ to $n$-buds for $n \geq p^h$,
but the added bit of generality again offers no real advantage to us.
\end{rk}

\begin{rk}\label{rk:ht_confusion}
We caution that potential confusion lurks in definitions \eqref{def:bud.ht.>=h}
and \eqref{def:bud.ht.h}: to say that a bud has ``height $\geq h$'' is
\emph{not} to say that it has ``height $h'$ for some $h' \geq h$''.  For
example, if $[p]_F(T) = a_{p^h}T^{p^h} + a_{2p^h}T^{2p^h} + \dotsb$ with 
$a_{p^h}$ a nonzero nonunit, then $\TT_S^F$ will have height $\geq h$ but will
not have a well-defined height.
\end{rk}

Our goal for the remainder of the section is to obtain a characterization of
$\Bn^h$.  Recall the formal group law $H = H_h$ of \eqref{ntn:H} and its
$n$-truncation $H^{(n)}$ \eqref{rk:trunc_buds}.  By \eqref{st:hthtrivgerm}, the 
$n$-bud $\TT_S^{H^{(n)}}$ has height $h$.

\begin{defn}\label{def:AutH^(n)}
We define $\sAut(H^{(n)})$ to be the presheaf of groups on $\Sch{}_{/\FF_p}$
\[
   \sAut(H^{(n)})\colon S \mapsto 
     \Aut_{\Gamma(S)}(H^{(n)})
     \ciso \Aut_{\nbuds(S)}\bigl(\TT_{n,S}^{H^{(n)}}\bigr).
\]
\end{defn}

Our theorem will assert that $\Bn^h$ is the classifying stack of
$\sAut(H^{(n)})$.  Now, up to this point, the classifying stacks we've
encountered have been essentially independent of the choice of topology; see
\eqref{rk:B(Aut(T)).indep.of.top} and \eqref{rk:B()_top_indep}.  But our theorem
would fail if we only considered $\sAut(H^{(n)})$-torsors for, for example, the
Zariski topology.  It will be convenient to formulate the theorem in terms of
the finite \'etale topology \cite{sga3-1}*{IV 6.3} instead.  Quite generally,
given a group $G$ over $\Spec \FF_p$, we write $B_\fet(G)$ for the stack over
$\Sch{}_{/\FF_p}$ of $G$-torsors for the finite \'etale topology.

\begin{thm}\label{st:Bn^h=B(AutH^(n))}
$\Bn^h \approx B_\fet\bigl(\sAut(H^{(n)})\bigr)$.
\end{thm}

\begin{proof}
By \eqref{st:ngermloctriv}, \eqref{eg:trivnbud}, \eqref{st:[p]_H=T^p^h.over.Fp}, and \eqref{st:ht.h.fin.et}, $\Bn^h$ is a neutral gerbe over $\Spec \FF_p$ for the finite \'etale topology, with section provided by $\TT_{\FF_p}^{H^{(n)}}$.  So apply \cite{lamb00}*{3.21}.
\end{proof}

\begin{rk}
Since \Bn is a stack for the fpqc topology \eqref{st:nbuds=stack}, so is its
locally closed substack $\Bn^h$.  Hence we deduce that 
$B\bigl(\sAut(H^{(n)})\bigr)$ is independent of the topology on 
$\Sch{}_{/\FF_p}$ between the finite \'etale and fpqc topologies, inclusive.  
In particular, every fpqc-torsor for $\sAut(H^{(n)})$ is in fact a 
finite-\'etale-torsor.
\end{rk}

\subsection{Automorphisms and endomorphisms of buds of height
$h$}\label{ss:aut_buds_ht_h}

Let $h\geq 1$ and let $S$ be a scheme.  For $n \geq p^{h+1}$, by
\eqref{st:Bn^h=B(AutH^(n))}, every $n$-bud of height $h$ over $S$ is isomorphic
finite-\'etale locally to $\TT_S^{H^{(n)}}$, with $H = H_h$ the formal group law
of \eqref{ntn:H} and $H^{(n)}$ its $n$-trun\-ca\-tion \eqref{rk:trunc_buds}.  So
we are naturally led to consider closely the group $\sAut(H^{(n)})$
\eqref{def:AutH^(n)}.  We shall devote this section to investigating some
aspects of its structure.  It will be convenient, especially in later sections,
to work as much as possible with regard to any $n \geq 1$; but our main results
here will require $n \geq p^{h+1}$ (or at least $n \geq p^h$, granting
\eqref{rk:ht_h_need_n>=p^h}), so that height $h$ makes sense.

To begin, let $n \geq 1$, and recall the schemes $\sAut(\TT_n)_{\FF_p}$ and
$(L_n)_{\FF_p}$ from \s\ref{ss:ht_>=h_buds} obtained by reducing $\sAut(\TT_n)$
\eqref{def:AutTT_n} and $L_n$ \eqref{def:laz_n}, respectively, mod $p$.


\begin{lem}\label{st:AutH^(n)=clsd_sbgp}
$\sAut(H^{(n)})$ is canonically represented by a closed sub-group scheme of
$\sAut(\TT_n)_{\FF_p}$.
\end{lem}

\begin{proof}
Roughly, the point is just that $\sAut(H^{(n)})$ naturally sits inside $\sAut(\TT_n)_{\FF_p}$ as the stabilizer of $H^{(n)}$.  Precisely,  $\sAut(\TT_n)_{\FF_p}$ acts naturally on $(L_n)_{\FF_p}$ by \eqref{rk:AutTTn.acts.on.Ln}, and we have a Cartesian square
\[
   \xymatrix{
      \sAut(H^{(n)}) \ar[r] \ar[d] \ar@{}[rd] |*{\square}
         & \sAut(\TT_n)_{\FF_p} \ar[d] \\
      \Spec \FF_p \ar[r]_-{H^{(n)}}
         & (L_n)_{\FF_p},
   }
\]
where the bottom horizontal arrow is the classifying map determined by $H^{(n)}$ and the right vertical arrow is defined on points by $f \mapsto f \cdot H^{(n)}$.
\end{proof}

By \eqref{st:nbuds_sm} and \eqref{st:Bn^h=B(AutH^(n))}, for $n \geq p^{h+1}$,
the classifying stack $B\bigl( \sAut(H^{(n)}) \bigr)$ 
is an open substack of a stack smooth of relative dimension $-h$ over
$\Spec\FF_p$.  Hence $B\bigl( \sAut(H^{(n)}) \bigr)$ is itself smooth of
relative dimension $-h$ over $\Spec \FF_p$.  Hence it natural to ask if
$\sAut(H^{(n)})$ is smooth of dimension $h$ over $\Spec \FF_p$.

We shall answer the question in the affirmative in \eqref{st:AutH^(n)=sm_dim_h}
below.  To prepare, let $n \geq 1$, and recall the
$\A_\bullet^{\TT_n}$-filtration on $\sAut(\TT_n)$ from \eqref{rk:AutTT_n_filt}.
Let $(\A_\bullet^{\TT_n})_{\FF_p}$ denote the filtration on $\sAut(\TT_n)_{\FF_p}$
obtained by base change to $\FF_p$.

\begin{defn}\label{def:A^H^(n)}
We define $\A_\bullet^{H^{(n)}}$ to be the intersection filtration on $\sAut(H^{(n)})$:
\[
	\A_i^{H^{(n)}} := \sAut(H^{(n)}) \fib{\sAut(\TT_n)_{\FF_p}} (\A_i^{\TT_n})_{\FF_p},
	\qquad i= 0,\ 1,\dotsc,\ n.
\]
\end{defn}

Concretely, $\A_0^{H^{(n)}} = \sAut(H^{(n)})$, and $\A_i^{H^{(n)}}$ is given on
points by
\[
   \A^{H^{(n)}}_i(S) := \biggl\{\, f \in \Aut_{\Gamma(S)}(H^{(n)}) \biggm|
   	\parbox{37ex}{\centering $f(T)$ is of the form \\
				   	  $T + a_{i+1}T^{i+1} + a_{i+2}T^{i+2} + \dotsb + a_nT^n$} 
				   	  \,\biggr\}.
\]
By \eqref{rk:AutTT_n_filt} and \eqref{st:AutH^(n)=clsd_sbgp}, $\A^{H^{(n)}}_0 /
\A^{H^{(n)}}_1$ embeds as a closed subscheme of $\GG_m = \GG_{m,\FF_p}$, and
$\A^{H^{(n)}}_i / \A^{H^{(n)}}_{i+1}$ embeds as a closed subscheme of $\GG_a =
\GG_{a,\FF_p}$ for $i = 1$, $2,\dotsc$, $n-1$.

Our main result for the section is following calculation of the successive
quotients of the $\A^{H^{(n)}}_\bullet$-filtration for $n \geq p^{h+1}$.
Let $l$ be the nonnegative integer such that $p^l \leq n < p^{l+1}$.

\begin{thm}\label{st:AutH^(n)_succ_quot}
We have an identification of presheaves
\[
   \A^{H^{(n)}}_i / \A^{H^{(n)}}_{i+1} \ciso
   \begin{cases}
      \mu_{p^h-1}, & i = 0;\\
      \GG_a^{\Fr_{p\smash{^h}}}, & i = p-1,\ p^2-1,\dotsc,\ p^{l-h}-1;\\
      \GG_a, & i = p^{l-h+1}-1,\ p^{l-h+2}-1, \dotsc,\ p^l-1;\\
      0, & \text{otherwise};
   \end{cases}
\]
where $\mu_{p^h-1} \subset \GG_m$ is the sub-group scheme of $(p^h-1)$th roots
of unity, and \linebreak $\GG_a^{\Fr_{p\smash{^h}}} \subset \GG_a$ is the
sub-group scheme of fixed points for the $p^h$th-power Frobenius operator.
\end{thm}

In other words, for any $\FF_p$-scheme $S$, we have
\[
   \mu_{p^h-1}(S) = \bigl\{\, a\in \Gamma(S)^\times \mid a^{p^h-1} = 1 \,\bigr\}
   \;\quad \text{and} \;\quad
   \GG_a^{\Fr_{p\smash{^h}}}(S) = \bigl\{\, a \in\Gamma(S) \mid a^{p^h} = a \,\bigr\}.
\]
Hence $\mu_{p^h-1}$ and $\GG_a^{\Fr_{p\smash{^h}}}$ are respresented,
respectively, by
\[
   \Spec \FF_p[T]/(T^{p^h-1}-1)
   \;\quad \text{and} \;\quad
   \Spec \FF_p[T]/(T^{p^h}-T).
\]
Hence both $\mu_{p^h-1}$ and $\GG_a^{\Fr_{p\smash{^h}}}$ are finite \'etale
groups over $\Spec \FF_p$.

Before proceeding to the proof of the theorem, let us first signal an immediate
consequence.  We continue with $n \geq p^{h+1}$.

\begin{cor}\label{st:AutH^(n)=sm_dim_h}
$\sAut(H^{(n)})$ is smooth of dimension $h$ over $\Spec \FF_p$.
\end{cor}

\begin{proof}
By the theorem, $\sAut(H^{(n)})$ is obtained as an iterated extension of smooth
groups, so is smooth. Moreover, the $\A_\bullet^{H^{(n)}}$-filtration has
precisely $h$ successive quotients of dimension $1$, and all other successive
quotients of dimension $0$. So the dimension assertion follows from
\cite{demga70}*{III \s3 5.5(a)}.
\end{proof}

We shall devote the rest of the section to the proof of
\eqref{st:AutH^(n)_succ_quot}.  One could extract the proof from a careful
analysis of some of the statements and arguments in \cite{laz55}*{\s IV} or in
\cite{froh68}*{I \s3, III \s2}.  But, for sake of clarity, it seems preferable
to give a reasonably self-contained proof here.  Our presentation has profited
significantly from notes we received from Spallone on a course of Kottwitz.

To prove \eqref{st:AutH^(n)_succ_quot}, it will be somewhat more convenient to
translate the problem into one concerning endomorphisms of $H^{(n)}$, as opposed
to automorphisms.  Let $n \geq 1$.

\begin{defn}
We define $\sEnd(H^{(n)})$ to be the presheaf of (noncommutative) rings on
$\Sch{}_{/\FF_p}$
\[
   \sEnd(H^{(n)})\colon
      S \mapsto \End_{\Gamma(S)}(H^{(n)})
      \ciso \End_{\nbuds(S)}\bigl(\TT_{n,S}^{H^{(n)}}\bigr).
\]
\end{defn}

Recall that the $S$-points of $\sEnd(H^{(n)})$ are the truncated polynomials
$f(T) \in \Gamma_n(S;T)$ that ``commute'' with $H^{(n)}$ in the sense of
\eqref{def:nbud.law}.  The ring structure on points of $\sEnd(H^{(n)})$ is
described explicitly in terms of $H^{(n)}$ in \eqref{rk:End_A(F)}.

\begin{defn}\label{def:I^H^(n)}
For $i=0$, $1,\dotsc$, $n$, we denote by $\I^{H^{(n)}}_i$ the subpresheaf of $\sEnd(H^{(n)})$ defined on points by
\[
   \I^{H^{(n)}}_i(S) := \biggl\{\, f \in \End_{\Gamma(S)}(H^{(n)}) \biggm|
      \parbox{33ex}{\centering $f(T)$ is of the form \\
                    $a_{i+1}T^{i+1} + a_{i+2}T^{i+2} + \dotsb + a_nT^n$} 
                    \,\biggr\}.
\]
\end{defn}

One verifies immediately that $\I^{H^{(n)}}_i$ is a presheaf of \emph{$2$-sided ideals} in $\sEnd(H^{(n)})$ for all $i$, and we have a decreasing filtration
\[
   \sEnd(H^{(n)}) = \I^{H^{(n)}}_0 \supset \I^{H^{(n)}}_1 
      \supset \dotsb \supset \I^{H^{(n)}}_{n-1} \supset \I^{H^{(n)}}_n = 0.
\]

For any $n \geq 1$, we can relate the $\I^{H^{(n)}}_\bullet$-filtration of
$\sEnd(H^{(n)})$ to the $\A^{H^{(n)}}_\bullet$-filtration of $\sAut(H^{(n)})$ as
follows.  By \eqref{rk:+.in.End_A(F)}, the map on points
\[
   f \mapsto \id +_{H^{(n)}} f,
\]
where $\id(T) = T$ \eqref{rk:End_A(F)}, defines a morphism of \emph{set}-valued presheaves
\[\tag{$*$}\label{disp:I^i->A^i}
   \I_i^{H^{(n)}} \to \A_i^{H^{(n)}}, \qquad 1 \leq i \leq n.
\]

\begin{lem}\label{st:I^i->A^i_isom}
The arrow \eqref{disp:I^i->A^i} is an isomorphism of presheaves of sets.
\end{lem}

\begin{proof}
The inverse is given by addition with $i_{H^{(n)}}$ \eqref{rk:inverse}.
\end{proof}

In a moment, we shall exploit the lemma to express the successive quotients of the $\A^{H^{(n)}}_\bullet$-filtration in terms of the successive quotients of the $\I^{H^{(n)}}_\bullet$-filtration.  But we first need another lemma.  Quite generally, let $R$ be a possibly noncommutative ring with unit, and let $I \subset R$ be a $2$-sided ideal such that $1+I\subset R^\times$.

\begin{lem}\label{st:noncomm}\hfill
\begin{enumerate}
\renewcommand{\theenumi}{\roman{enumi}}
\item\label{it:R/(1+I)}
   The natural map $R^\times/(1+I) \to (R/I)^\times$ is an isomorphism of
   groups.
\item\label{it:(1+I)/(1+J)}
   Let $J$ be a $2$-sided ideal such that $I^2\subset J \subset I$.  Then the map 
   \[\tag{$**$}\label{disp:(1+I)/(1+J)}
      i \mapsto 1 + i \mod 1 + J
   \]
   induces an isomorphism of groups $I/J \isoarrow (1+I)/(1+J)$.
\end{enumerate}
\end{lem}

\begin{proof}
\eqref{it:R/(1+I)} Immediate.

\eqref{it:(1+I)/(1+J)} It plainly suffices to show that \eqref{disp:(1+I)/(1+J)}
defines a group homomorphism $I \to (1+I)/(1+J)$.  That is, given $i$ and $i'\in
I$, we must find $j\in J$ such that $(1+i+i')(1+j) = (1+i)(1+i')$.  Take
$j:=(1+i+i')^{-1}ii'$.
\end{proof}

The two previous lemmas yield the following as an immediate consequence.

\begin{lem}\label{st:A_succ_quot_vs_I_succ_quot}
The natural arrow 
\[
   \A^{H^{(n)}}_0/\A^{H^{(n)}}_1 = \sAut(H^{(n)})/\A^{H^{(n)}}_1 
      \to 
      \bigl( \sEnd(H^{(n)})/\I^{H^{(n)}}_1 \bigr)^\times = (\I^{H^{(n)}}_0/\I^{H^{(n)}}_1)^\times
\]
is an isomorphism of presheaves of abelian groups.  For $1\leq i \leq n-1$, the arrow \eqref{disp:I^i->A^i} induces an isomorphism of presheaves of abelian groups 
\begin{xxalignat}{3}
   &\phantom{\square} & \I_i^{H^{(n)}}/\I_{i+1}^{H^{(n)}}
        & \isoarrow \A^{H^{(n)}}_i/\A^{H^{(n)}}_{i+1}.
   & \square
\end{xxalignat}
\end{lem}

This last lemma reduces \eqref{st:AutH^(n)_succ_quot} to the calculation of the
successive quotients of the $\I^{H^{(n)}}_\bullet$-filtration.  To prepare, it
will be convenient to formulate a few general lemmas on homomorphisms between
bud laws.  We continue with $n \geq 1$.

Let $A$ be a ring, and fix $n$-bud laws $F$ and $G$ over $A$.  Let $f\in
A[T]/(T)^{n+1}$ have $0$ constant term. We define $\partial f \in
A[T_1,T_2]/(T_1,T_2)^{n+1}$ to measure the failure of $f$ to be a homomorphism
$F \to G$:
\[\tag{$\dag$}\label{disp:partial_f}
   (\partial f) (T_1,T_2) := f\bigl(F(T_1,T_2)\bigr) -
                              G\bigl(f(T_1),f(T_2)\bigr).
\]
In keeping with our notation for truncated bud and group laws \eqref{rk:trunc_buds}, we write $f^{(m)}$ for the image of $f$
in $A[T]/(T)^{m+1}$, and $(\partial f)^{(m)}$ for the image of $\partial f$ 
in $A[T_1,T_2]/(T_1,T_2)^{m+1}$, $m\leq n$.

As an easy first lemma, we shall consider the effect of perturbations to $f$ on
$\partial f$.  Let $a\in A$, and put $g(T) := f(T) + aT^m$.  Recall the
polynomial $B_m$ of \eqref{eg:B_m_and_C_m}.

\begin{lem}\label{st:partial_f_perturb}
$(\partial g)^{(m)} = (\partial f)^{(m)} + aB_m$.
\end{lem}

\begin{proof}
Without loss of generality, we may assume $m = n$.  Then we just compute
\[
   g\bigl(F(T_1,T_2)\bigr)
      = f\bigl(F(T_1,T_2)\bigr) + aF(T_1,T_2)^n
      = f\bigl(F(T_1,T_2)\bigr) + a(T_1+T_2)^n
\]
and
\[
   G\bigl(g(T_1),g(T_2)\bigr)
      = G\bigl(f(T_1)+aT^n,f(T_2)+aT^n\bigr)
      = G\bigl(f(T_1),f(T_2)\bigr) + aT_1^n + aT_2^n
\]
and subtract.
\end{proof}

For the next two lemmas, assume $m\geq 2$, and suppose that $f^{(m-1)}$ is a
homomorphism $F^{(m-1)} \to G^{(m-1)}$, so that $(\partial f)^{(m-1)} = 0$.
Recall the polynomial $C_m$ of \eqref{eg:B_m_and_C_m}.

\begin{lem}\label{st:partial_f_trunc}
There exists a unique $a\in A$ such that $(\partial f)^{(m)} = aC_m$.
\end{lem}

\begin{proof}
We shall show that $(\partial f)^{(m)}$ is a homogenous SPC \eqref{def:SPC} of
degree $m$ and appeal to \eqref{st:mult_C_m}. Without loss of generality, we may
assume $m = n$.  Since $F$ and $G$ are commutative, it is clear that $\partial
f$ is symmetric. Since $f^{(n-1)}$ is a homomorphism, it is clear that $\partial
f$ is homogenous of degree $n$. So we just have to show that $\partial f$
satisfies the cocycle condition (\ref{def:SPC}\ref{it:cocycle}). We shall do so
by exploiting the associativity of $F$ and $G$.

Replacing $T_2$ by $F(T_2,T_3)$ in \eqref{disp:partial_f}, we obtain an equality
of elements in the ring $A[T_1,T_2,T_3]/(T_1,T_2,T_3)^{n+1}$,
\[
   (\partial f)\bigl(T_1,F(T_2,T_3)\bigr) 
      = f\Bigl(F\bigl(T_1,F(T_2,T_3)\bigr)\Bigr) 
               - G\Bigl(f(T_1),f\bigl(F(T_2,T_3)\bigr)\Bigr).
\]
Replacing $f\bigl( F(T_2,T_3)\bigr)$ with $G\bigl(f(T_2),f(T_3)\bigr) + (\partial f)(T_2,T_3)$ in the display, we obtain
\begin{multline*}
   (\partial f)\bigl(T_1,F(T_2,T_3)\bigr) 
      = f\Bigl(F\bigl(T_1,F(T_2,T_3)\bigr)\Bigr)\\ 
      - G\Bigl(f(T_1),G\bigl(f(T_2),f(T_3)\bigr)
                            +(\partial f)(T_2,T_3)\Bigr).
\end{multline*}
Since $\partial f$ is homogenous of degree $n$, the left-hand side of this last
display is
\[
   (\partial f)(T_1,T_2+T_3),
\]
and the right-hand side is
\[
   f\Bigl(F\bigl(T_1,F(T_2,T_3)\bigr)\Bigr)
   - G\Bigl(f(T_1),G\bigl(f(T_2),f(T_3)\bigr)\Bigr)
   - (\partial f)(T_2,T_3).
\]

Analogously, replacing $T_1$ by $F(T_1,T_2)$ and $T_2$ by $T_3$ in
\eqref{disp:partial_f}, one obtains  a second equality in
$A[T_1,T_2,T_3]/(T_1,T_2,T_3)^{n+1}$,
\begin{multline*}
   (\partial f)(T_1+T_2,T_3)
      = f\Bigl(F\bigl(F(T_1,T_2),T_3\bigr)\Bigr)\\
      - G\Bigl(G\bigl(f(T_1),f(T_2)\bigr),f(T_3)\Bigr)
      - (\partial f)(T_1,T_2).
\end{multline*}
Now subtract equations and use the associativity of $F$ and $G$.
\end{proof}

We continue with $a$ as in the previous lemma.  Recall $\lambda(m)$ from
\eqref{eg:B_m_and_C_m}.

\begin{lem}\label{st:[p]_commutator}
For any $k\geq 1$,
\[
   (f^{(m)} \circ [k]_{F^{(m)}})(T) = ([k]_{G^{(m)}}\circ f^{(m)})(T) +
\frac{k^m-k}{\lambda(m)} \cdot aT^m.
\]
In particular, for $k=p$, $m$ of the form $p^j$, and $A$ of characteristic $p$,
we have
\[
   (f^{(p^j)} \circ [p]_{F^{(p^j)}})(T) = ([p]_{G^{(p^j)}}\circ f^{(p^j)})(T) -
aT^m.
\]
\end{lem}

\begin{proof}
The proof is entirely similar to that given in \cite{laz55}*{Lemme 6} or in
\cite{froh68}*{III \s1 Lemma 4}, but for convenience, we'll write out the details.  Without loss of generality, we may assume $m=n$.

We induct on $k$. The
assertion is clear for $k=1$. So assume the assertion holds for $k$.  We have
\begin{alignat*}{2}
   (f \circ [k+1]_{F})(T)
      = {}& f\Bigl( F\bigl(T,[k]_F(T)\bigr)\Bigr)\\
      = {}& G\bigl( f(T), (f\circ [k]_F)(T) \bigr)
                + (\partial f)\bigl(T,[k]_F(T)\bigr)
          & &\qquad\eqref{disp:partial_f}.
\end{alignat*}
Let us examine separately the two terms appearing in the last expression.  We
have
\begin{alignat*}{2}
   G\bigl( f(T), (f\circ [k]_F)(T) \bigr)
      &= G\Bigl( f(T), ([k]_G \circ f)(T)
                  + \tfrac{k^n-k}{\lambda(n)}aT^n \Bigr)
         & &\qquad\text{(induction)}\\
      &= G\bigl( f(T), ([k]_G \circ f)(T) \bigr)
        + \frac{k^n-k}{\lambda(n)}aT^n\\
      &= ([k+1]_G\circ f)(T) + \frac{k^n-k}{\lambda(n)}aT^n.
\end{alignat*}
Moreover,
\begin{alignat*}{2}
   (\partial f)\bigl(T,[k]_F(T)\bigr)
      &= aC_n\bigl(T,[k]_F(T)\bigr)
         & &\qquad\eqref{st:partial_f_trunc}\\
      &= aC_n(T,kT)
         & &\qquad\eqref{rk:[p].leading.coeff=p}\\
      &= a\frac{(k+1)^n-k^n-1}{\lambda(n)}T^n.
\end{alignat*}
The lemma follows at once.
\end{proof}

We are now ready to compute the $\I_i^{H^{(n)}}/\I_{i+1}^{H^{(n)}}$'s.
We denote by $\OO = \OO_{\FF_p}$ the tautological ring scheme structure on
$\AA^1_{\FF_p}$, and by $\OO^{\Fr_{\smash{p^h}}}$ the sub-ring scheme of \OO of
fixed points for the $p^h$th-power Frobenius operator:
\[
   \OO^{\Fr_{\smash{p^h}}}(S) = \bigl\{\, a \in\Gamma(S) \bigm| a^{p^h} = a \,\bigr\}.
\]
Quite as in \eqref{rk:AutTT_n_filt}, one verifies at once that the map on points
\[\tag{$\ddag$}\label{disp:I^i->G_a}
   a_{i+1}T^{i+1} + \dotsb + a_nT^n \mapsto a_{i+1}
\]
specifies a monomorphism of presheaves of rings
\[
   \I_0^{H^{(n)}}/\I_1^{H^{(n)}} \inj \OO, \qquad i=0,
\]
and a monomorphism of presheaves of abelian groups
\[
   \I^{H^{(n)}}_i / \I^{H^{(n)}}_{i+1} \inj \GG_a, \qquad 1\leq i \leq n-1.
\]
Assume $n \geq p^{h+1}$, and again let $l$ be the nonnegative integer such
that $p^l \leq n < p^{l+1}$.

\begin{thm}\label{st:EndH^(n)_succ_quot}
For $0\leq i \leq n-1$, \eqref{disp:I^i->G_a} induces an identification of 
presheaves
\[
   \I^{H^{(n)}}_i / \I^{H^{(n)}}_{i+1} \ciso
   \begin{cases}
      \OO^{\Fr_{\smash{p^h}}}, & i = 0;\\
      \GG_a^{\Fr_{p\smash{^h}}}, & i = p-1,\ p^2-1,\dotsc,\ p^{l-h}-1;\\
      \GG_a, & i = p^{l-h+1}-1,\ p^{l-h+2}-1, \dotsc,\ p^l-1;\\
      0, & \text{otherwise.}
   \end{cases}
\]
\end{thm}

\begin{proof}
Let $A$ be a ring of characteristic $p$, and let $I_i:= \I^{H^{(n)}}_i(A)$, $0
\leq i \leq n$. For $i\neq 0$, $p-1$, $p^2-1,\dotsc$, $p^l-1$, we have
$I_i/I_{i+1} = 0$ by \eqref{st:[p].law.form}.  So we are left to compute the
quotients for $i$ of the form $p^j-1$.

As a first step, let $f(T) = a_1T + \dotsb a_nT^n$ be any endomorphism of
$H^{(n)}$ over $A$. Let $A^{\Fr_{\smash{p^h}}} := \OO^{\Fr_{\smash{p^h}}}(A)$. Since $f$ must commute with $[p]_{H^{(n)}}(T) = T^{p^h}$, we deduce $a_i
\in A^{\Fr_{\smash{p^h}}}$ for $ip^h \leq n$. In particular, the map
\eqref{disp:I^i->G_a} carries $I_i/I_{i+1}$ into
$A^{\Fr_{p\smash{^h}}}$ for $i = 0$, $p-1,\dotsc$, $p^{l-h} -1$, as asserted. So
we are reduced to showing the following: given $a_{p^j}T^{p^j} \in
A[T]/(T)^{n+1}$, with $a_{p^j}\in A^{\Fr_{\smash{p^h}}}$ in case $j \leq l-h$ and no
constraint on $a_{p^j}$ in case $j > l-h$, we can add terms of degree $>p^j$ to obtain an endomorphism of $H^{(n)}$.

We shall proceed by induction on the degree of the term to be added.  We work with $\partial$ \eqref{disp:partial_f} computed with regard to $F = G = H^{(n)}$.  To begin, let $f(T)= a_{p^j}T^{p^j}$.  Then
\begin{align*}
   (\partial f)^{(p^j)}(T_1,T_2)
      &= f\bigl(H^{(p^j)}(T_1,T_2)\bigl)
        - H^{(p^j)}\bigl( f(T_1),f(T_2) \bigr)\\
      &= a_{p^j}(T_1+T_2)^{p^j} -a_{p^j}T_1^{p^j} -a_{p^j}T_2^{p^j}\\
      &= 0.
\end{align*}
Hence $f$ defines an endomorphism of $H^{(p^j)}$.  We must now show that if $g(T) = a_{p^j}T^{p^j}+\dotsb + a_{m-1}T^{m-1}$ defines an endomorphism of $H^{(m-1)}$, $p^j+1 \leq m \leq n$, then we can always add a term of degree $m$ to $g$ to obtain an endomorphism of $H^{(m)}$.

\emph{Case 1: $m$ is not a power of $p$.} By \eqref{st:partial_f_trunc},
$(\partial g)^{(m)} = aC_m$ for a unique $a\in A$. Now, since $A$ is of
characteristic $p$ and $m$ is not a power of $p$, $C_m$ is a unit multiple of
$B_m$ over $A$. Hence by \eqref{st:partial_f_perturb} we can find (a unique) $a_m \in A$ such that $g(T) + a_mT^m$ defines an endomorphism of $H^{(m)}$.

Note that if $a_{p^j},\dotsc$, $a_{m-1}\in A^{\Fr_{\smash{p^h}}}$, then $a_m \in A^{\Fr_{\smash{p^h}}}$ too, since $H$ is defined over $\FF_p$, $A^{\Fr_{\smash{p^h}}}$ is a subring of $A$, and $a_m$ is uniquely determined.

\emph{Case 2: $m$ is a power of $p$.} In this case $B_m = 0$ over $A$, so the
method of case 1 breaks down.  But we claim that, in fact, $g$ is already an
endomorphism of $H^{(m)}$.  By \eqref{st:partial_f_trunc} and
\eqref{st:[p]_commutator}, it suffices to show that $g\circ [p]_{H^{(m)}} =
[p]_{H^{(m)}} \circ g$ in $A[T]/(T)^{m+1}$.  We shall show that the stronger
statement $g\circ [p]_{H^{(n)}} = [p]_{H^{(n)}} \circ g$ in $A[T]/(T)^{n+1}$
holds.  We consider separately the subcases $j \leq l-h$ and $j > l-h$.  If $j
\leq l-h$, then recall $a_{p^j} \in A^{\Fr_{\smash{p^h}}}$.  Hence by induction
all the coefficients of $g$ lie in $A^{\Fr_{\smash{p^h}}}$.  Hence $g$ commutes
with $[p]_{H^{(n)}}$.  If $j > l-h$, then the argument is even easier: $g \circ
[p]_{H^{(n)}}$ and $[p]_{H^{(n)}} \circ g$ both only involve terms of degrees
$\geq p^{j+h}$.  So both are $0$ in $A[T]/(T)^{n+1}$ by definition of $l$.
\end{proof}

At last we obtain the proof of \eqref{st:AutH^(n)_succ_quot}.

\begin{proof}[Proof of \eqref{st:AutH^(n)_succ_quot}]
Clear from \eqref{st:A_succ_quot_vs_I_succ_quot} and
\eqref{st:EndH^(n)_succ_quot}, noting for the $i=0$ case that $\mu_{p^h-1}$ sits
naturally inside $\OO^{\Fr_{\smash{p^h}}}$ as the subfunctor of units.
\end{proof}

\begin{rk}
One verifies immediately that the maps
\[
   \A^{H^{(n)}}_i / \A^{H^{(n)}}_{i+1} \to
   \begin{cases}
      \GG_m, & i=0;\\
      \GG_a, & 1\leq i \leq n-1
   \end{cases}
\] 
induced by \eqref{rk:AutTT_n_filt} and the maps
\[
   \I^{H^{(n)}}_i / \I^{H^{(n)}}_{i+1} \to
   \begin{cases}
      \OO, & i=0;\\
      \GG_a, & 1 \leq i \leq n-1
   \end{cases}
\]
of the previous theorem are compatible with the identifications of \eqref{st:A_succ_quot_vs_I_succ_quot}.
\end{rk}

\section{The height stratification: formal Lie groups}\label{s:ht_strat_flg}

We continue working with respect to fixed prime $p$.

\subsection{The height stratification on the stack of formal Lie
groups}\label{ss:ht_strat_flg}

In this section we introduce the height stratification on the stack of formal Lie groups, quite in analogy with the height stratification on \Bn, $n \geq 1$.

Let $h \geq 0$.  We denote by $\FLG^{\geq h}$ the full sub-fibered category of \FLG rendering the diagram
\[\tag{$*$}\label{disp:FLG^>=h_Cart_diag}
   \vcenter{
   \xymatrix{
      \FLG^{\geq h} \ar[r] \ar[d]
         & \FLG \ar[d]\\
      \Bn[p^h]^{\geq h} \ar[r]
         & \Bn[p^h]
      }
      }
\]
Cartesian; here, as usual, the right vertical arrow denotes truncation. Abusing
notation, we denote again by \L the pullback to \FLG of the line bundle \L on
$\Bn[1]$ \eqref{def:L}. Similarly, we abusively denote by $\L_h$ the restriction
of \L to $\FLG^{\geq h}$; then $\L_h$ is canonically isomorphic to the pullback
to $\FLG^{\geq h}$ of the line bundle $\L_h$ on $\Bn[p^h]^{\geq h}$
\eqref{def:L}. We abusively denote again by $v_h$ the section $\O_{\FLG^{\geq
h}} \to \L_h^{\tensor p^h-1}$ over $\FLG^{\geq h}$ obtained by pulling back the
section $v_h\colon \O_{\Bn[p^h]^{\geq h}} \to \L_h^{\tensor p^h-1}$ from
$\Bn[p^h]^{\geq h}$ \eqref{def:v_h}.

The fibered category $\FLG^{\geq h}$ and the various sections $v_i$ are related
in the following simple way. Let $X$ be a formal Lie group over the base scheme $S$.

\begin{prop}\label{st:ht>=h_tfae}
The following are equivalent.
\begin{enumerate}
\renewcommand{\theenumi}{\roman{enumi}}
\item
   $X$ is an object in $\FLG^{\geq h}$.
\item
   The $p^h$-bud $X^{(p^h)}$ has height $\geq h$.
\item
   For any $n\geq p^h$, the $n$-bud $X^{(n)}$ has height $\geq h$.
\item
   $X$ is an object in each of the successive zero loci \eqref{def:sectionoflb} $V(v_0)$, $V(v_1)\text{,}\dotsc$, $V(v_{h-1})$.
\end{enumerate}
\end{prop}

\begin{proof}
\eqref{rk:ht.>=h.compat.truncation} and \eqref{st:nbuds^>=h=V(a_h)}.
\end{proof}

\begin{defn}
$X$ has \emph{height $\geq h$} it it satisfies the equivalent conditions of \eqref{st:ht>=h_tfae}.
\end{defn}

%

\begin{eg}\label{eg:triv_ht>=h}
Quite as for buds \eqref{eg:triv_bud_ht>=h}, given a formal group law $F$ over $\Gamma(S)$, the formal Lie group $\wh\AA_S^F$ \eqref{eg:trivfg} has height $\geq h$ $\iff$ $[p]_F \in T^{p^h} \cdot \Gamma(S)[[T]]$.  So our terminology again comports with \eqref{def:law.height}.
\end{eg}

\begin{rk}
Many of the above definitions are independent of fixed choices we've made.  For example, the proposition says that we could have just as well defined $\FLG^{\geq h}$ by replacing the diagram \eqref{disp:FLG^>=h_Cart_diag} with one in which $p^h$ is everywhere replaced by any $n \geq p^h$.  Up to canonical isomorphism, \eqref{rk:L_indep_of_n} says we could have defined \L as the pullback to \FLG of the line bundle \L on $\Bn$, for any $n \geq 1$; and similarly for $\L_h$, for any $n \geq p^h$.  Analogously, \eqref{rk:v_h_indep_of_n} says that we could have defined $v_h$ as the pullback of the section $v_h$ over $\Bn^{\geq h}$, for any $n \geq p^h$.
\end{rk}

\begin{rk}
Just as for buds, $\FLG^{\geq 0} = \FLG$, and $\FLG^{\geq 1}$ is the
stack of formal Lie groups over $\FF_p$-schemes.
\end{rk}

\begin{prop}
$\FLG^{\geq h}$ is a stack for the fpqc topology, and the inclusion functor
$\FLG^{\geq h} \to \FLG$ is a closed immersion.
\end{prop}

\begin{proof}
The diagram \eqref{disp:FLG^>=h_Cart_diag} is Cartesian.  So the first assertions follows because $\Bn[p^h]$ \eqref{st:nbuds=stack}, \FLG \eqref{st:FLG=fpqc_stack}, and $\Bn[p^h]^{\geq h}$ \eqref{st:ht.>=h.clsd.imm} are fpqc stacks.  And the second assertion follows because $\Bn[p^h]^{\geq h} \to \Bn[p^h]$ is a closed immersion \eqref{st:ht.>=h.clsd.imm}.
\end{proof}

%
%

\begin{rk}
As for buds, we obtain a decreasing filtration of closed substacks $\FLG =
\FLG^{\geq 0} \ctnsneq \FLG^{\geq 1} \ctnsneq \FLG^{\geq 2} \ctnsneq \dotsb.$ By
contrast with the bud case \eqref{rk:ht.>=h.compat.truncation}, the filtration
for \FLG is of infinite length.
\end{rk}

\subsection{The stack of height $\geq h$ formal Lie groups}

In this section we collect some characterizations of $\FLG^{\geq h}$ analogous to previous results on $\Bn^{\geq h}$ and \FLG.

Let us first consider an analog to the description of $\Bn^{\geq h}$ in \eqref{Bn^>=h_quot}.  Let $U$ be a universal (for $\ZZ_{(p)}$-algebras) formal group law over $\ZZ_{(p)}[t_1,t_2,\dotsc]$ as in \eqref{rk:univ_ht_>=h_law}. Recall that, in the notation of \eqref{rk:univ_ht_>=h_law}, the reduction of $U$ over the ring
\[
   B_h := \ZZ_{(p)}[t_1,t_2,\dotsc]/(p,a_1,\dotsc,a_{h-1})
      \ciso \FF_p[t_1,t_2,\dotsc]/(\ol a _1,\dotsc,\ol a _{h-1})
\]
is a universal group law of height $\geq h$, and $B_h$ is a polynomial ring over $\FF_p$ on the images of the $t_i$ for $i\neq p^1-1$, $p^2-1,\dotsc$,
$p^{h-1}-1$.  Let $\sAut(\wh\AA)_{\FF_p} := \sAut(\wh\AA) \tensor \FF_p$, with $\sAut(\wh\AA)$ as in \eqref{def:Aut_whAA}.  Just as in \eqref{Bn^>=h_quot}, we deduce the following.

\begin{thm}
$\sAut(\wh\AA)_{\FF_p}$ acts naturally on $\Spec B_h$, and we have
$\FLG^{\geq h} \approx \sAut(\wh\AA)_{\FF_p} \bs \Spec B_h$.\hfill $\square$
\end{thm}

In analogy with \eqref{st:FLG=limBn}, let us next consider the relation between the stacks $\FLG^{\geq h}$ and $\Bn^{\geq h}$, $n\geq p^h$.  By \eqref{rk:ht.>=h.compat.truncation}, we may form the limit
$\ilim_{n\geq p^h} \Bn^{\geq h}$ of the $\Bn^{\geq h}$'s with respect to the truncation functors.  By \eqref{st:ht>=h_tfae}, truncation determines an arrow
\[\tag{$*$}\label{disp:FLG^>=h->lim}
   \FLG^{\geq h} \to \ilim_{n \geq p^h} \Bn^{\geq h}.
\]
As always, we emphasize that the limit is taken in the sense of bicategories;
see the appendix.

\begin{thm}\label{st:FLG>=h=lim}
The arrow \eqref{disp:FLG^>=h->lim} is an equivalence of stacks.
\end{thm}

\begin{proof}
Combine \eqref{st:FLG=limBn}, \eqref{st:ht>=h_tfae}, and \eqref{st:initial_lim}.
\end{proof}

\subsection{The stratum of height $h$ formal Lie groups I}\label{ss:FLG_ht_h_I}

In this section, in analogy with \s\ref{ss:ht_h_buds}, we begin to study the \emph{strata} of the height stratification on \FLG, or, in other words, the notion of (exact) height for formal Lie groups.  Let $X$ be a formal Lie group over the base scheme $S$.

\begin{prop}\label{st:ht_h_tfae}
The following are equivalent.
\begin{enumerate}
\renewcommand{\theenumi}{\roman{enumi}}
\item
   The $p^{h+1}$-bud $X^{(p^{h+1})}$ has height $h$.
\item
   For any $n \geq p^{h+1}$, the $n$-bud $X^{(n)}$ has height $h$.
\item
   $X$ is an object in the open complement of $\FLG^{\geq h+1}$ in $\FLG^{\geq h}$.
\end{enumerate}
\end{prop}

\begin{proof}
\eqref{rk:ht.h.compat.truncation}.
\end{proof}

\begin{defn}\label{def:FLG_ht_h}
$X$ has \emph{height $h$}, or \emph{exact height $h$}, if it satisfies the equivalent conditions of \eqref{st:ht_h_tfae}.  We denote by $\FLG^h$ the substack of \FLG of formal Lie groups of height $h$.
\end{defn}


%
%
%
%

\begin{eg}
Quite as for buds \eqref{st:hthtrivgerm}, if $X = \wh\AA_S^F$ for the formal group law $F$ over $\Gamma(S)$ \eqref{eg:trivfg}, then the notion of height $h$ for $X$ recovers precisely that for $F$ \eqref{def:law.height}.
\end{eg}

\begin{rk}
The caution of \eqref{rk:ht_confusion} still applies: to
say that a formal Lie group has ``height $\geq h$'' is \emph{not} to say that it has ``height $h'$ for some $h' \geq h$''.
\end{rk}

\begin{rk}[Relation to $p$-Barsotti-Tate groups]
Our notion of height for formal Lie groups is related to, but \emph{not}
strictly compatible with, the notion of height for $p$-Barsotti-Tate, or 
$p$-divisible, groups.  In rough form,
the difference is that (exact) height for formal Lie groups is a
\emph{locally closed} condition, whereas height for Barsotti-Tate groups is a
\emph{fiberwise} condition.  For example, if $X$ is a formal Lie group of height
$h$ in the sense of \eqref{def:FLG_ht_h}, then $X$ is a Barsotti-Tate group of
height $h$ in the sense of Barsotti-Tate groups.  But the converse can easily fail. 
For example, $\wh\GG _m$ is a Barsotti-Tate group of height $1$ over any base
scheme on which $p$ is locally nilpotent.  But $\wh\GG _m$ has height $1$ as
a formal Lie group exactly when $p$ is honestly $0$.

Similar examples exist for any height $h > 1$.  Let $A$ be a ring and $I\subset
A$ a nonzero nilideal such that $B := A/I$ is of characteristic $p$.  Let $F$ be
a formal group law of height $h$ over $B$.  As $A$ is necessarily a
$\ZZ_{(p)}$-algebra, we may apply \eqref{rk:univ_ht_>=h_law} to lift $F$ to a
group law $\wt F$ over $A$ such that the coefficients in $[p]_{\wt F}(T)$ of
$T^p$, $T^{p^2},\dotsc$, $T^{p^{h-1}}$ are any elements of $I$ that we like.  In
particular, we can ensure that $\wt F$ is not of height $\geq h$.  But it is
easy to verify that $\wh\AA_A^{\wt F}$ \eqref{eg:trivfg} is a Barsotti-Tate
group of height $h$.
%
\end{rk}

\begin{rk}[Relation to $p$-typical formal group laws]\label{rk:p-typ}
Let us digress for a moment to briefly discuss $BP$-theory and $p$-typical
formal group laws.  We refer to
\cite{rav86} for general background, especially to \cite{rav86}*{App.\ 2} for
the relevant group law theory.  Recall that $BP_*$ and the ring $W := BP_*[t_0,
t_0^{-1}, t_1, t_2,\dotsc]$ admit a natural Hopf algebroid structure such that
the associated internal groupoid in the category of affine $\ZZ_{(p)}$-schemes
\[\tag{$*$}\label{disp:BP_gpd}
   \Spec W \rra \Spec BP_*
\]
represents $p$-typical formal group laws and the isomorphisms between them. In
particular, letting \X denote the stackification of \eqref{disp:BP_gpd}, there
is a natural morphism $f\colon \X \to \FLG \tensor \ZZ_{(p)}$, and one verifies
just as in \cite{nau07}*{34(2)} that $f$ is an equivalence. Hence the height
stratification on \FLG induces a stratification on \X, or in other words, a
stratification on $\Spec BP_*$ by \emph{invariant} closed subschemes.

Now, recall that $BP_* \iso \ZZ_{(p)}[u_1,u_2,\dotsc]$, where for convenience we
take the $u_i$'s to be the Araki generators and set $u_0 := p$. Recall also
Landweber's ideals $I_0 := 0$ and $I_h := (u_0,u_1,\dotsc,u_{h-1})$, $h > 0$, in
$BP_*$. Then for all $h \geq 0$, \emph{the closed substack $\FLG^{\geq h}
\tensor \ZZ_{(p)}$ in $\FLG \tensor \ZZ_{(p)} \approx \X$ corresponds to the
ideal $I_h \subset BP_*$}; one may deduce this essentially from Landweber's
classification of invariant prime ideals in $BP_*$
\citelist{\cite{land73}*{2.7}\cite{land76}*{6.2}}, or in a more direct fashion
from the formula \cite{rav86}*{A2.2.4} (this formula is the only point where our
particular choice of the Araki generators enters). In particular, our
notion of (exact) height agrees with Pribble's \cite{prib04}*{4.5}. The identification of the height
stratification and the $I_h$-stratification on \X is also noted in
\cite{nau07}*{\s6 pp.\ 25--26}; one verifies immediately that Naumann's
definition of the height stratification agrees with ours.

This said, let us note that our notion of height is not completely compatible
with the notion of height for $BP_*$-algebras in \cite{hovstr05}*{4.1}.  Namely, given a $BP_*$-algebra $A$, consider the composite
\[
   \Spec A \to \Spec BP_* \to \FLG.
\]
From the point of view of this paper, it would be reasonable to say that $A$ is a $BP_*$-algebra of height $h$ if the displayed composite factors through $\FLG^h$.  But, as noted in \cite{nau07}*{24}, $A$ has height $h$ in the sense of \cite{hovstr05} if it satisfies the strictly weaker condition that $h$ is the smallest nonnegative integer for which the composite factors through the open substack $\FLG - \FLG^{\geq h+1}$ of \FLG (\cite{hovstr05} defines $A$ to have height $\infty$ if the composite fails to factor through  $\FLG - \FLG^{\geq h+1}$ for any $h$).
\end{rk}

We shall next formulate a characterization of $\FLG^h$ analogous to
\eqref{st:Bn^h=B(AutH^(n))}.  Recall the formal group law $H = H_h$ of
\eqref{ntn:H}. 

\begin{defn}\label{def:AutH}
We define $\sAut(H)$ to be the presheaf of groups on $\Sch{}_{/\FF_p}$
\[
   \sAut(H)\colon S \mapsto 
     \Aut_{\Gamma(S)}(H)
     \ciso \Aut_{\flg(S)}\bigl(\wh\AA_S^H\bigr).
\]
\end{defn}

Whereas in \eqref{st:Bn^h=B(AutH^(n))} we were led to consider torsors for the finite \'etale topology, we shall now need to consider $\sAut(H)$-torsors for the fpqc topology.  Given a group $G$ over $\Spec \FF_p$, we write $B_\fpqc(G)$ for the stack over $\Sch{}_{/\FF_p}$ of
$G$-torsors for the fpqc topology.

\begin{thm}\label{st:FLG^h=B(AutH)} 
$\FLG^h \approx B_\fpqc\bigl(\sAut(H)\bigr)$.
\end{thm}

\begin{proof}
Essentially identical to the proof of \eqref{st:Bn^h=B(AutH^(n))}.
\end{proof}

\begin{rk}
The statement of the theorem is not entirely sharp:  by \eqref{st:ht.h.fin.et}, it would suffice to replace the fpqc topology by the topology on $\Sch{}_{/\FF_p}$ generated by the Zariski topology and all maps $\Spec B \to \Spec A$ between affine schemes obtained as a limit of surjective finite \'etale maps $\dotsb \to \Spec B_2 \to \Spec B_1 \to \Spec A$.
\end{rk}

We shall study the group $\sAut(H)$ and its relation to the groups  $\sAut(H^{(n)})$ for varying $n \geq p^{h+1}$ in the next section.  

We shall conclude this section by formulating another characterization of the stack $\FLG^h$, this time the obvious analog of \eqref{st:FLG>=h=lim}.  By \eqref{rk:ht.h.compat.truncation}, we may form the limit
$\ilim_{n\geq p^{h+1}} \Bn^h$ of the $\Bn^h$'s with respect to the truncation functors.  By \eqref{st:ht_h_tfae}, truncation determines an arrow
\[\tag{$**$}\label{disp:FLG^h->lim}
   \FLG^h \to \ilim_{n \geq p^{h+1}} \Bn^h.
\]
As always, we emphasize that the limit is taken in the sense of bicategories;
see the appendix.  As in \eqref{st:FLG>=h=lim}, only replacing the reference to \eqref{st:ht>=h_tfae} with \eqref{st:ht_h_tfae}, we obtain the following.

\begin{thm}\label{st:FLG^h=lim}
The arrow \eqref{disp:FLG^h->lim} is an equivalence of stacks. \hfill $\square$
\end{thm}

%
%

\subsection{Automorphisms and endomorphisms of formal Lie groups of height $h$}\label{ss:hthauts}

Let $h \geq 1$.  Our result $\FLG^h \approx B_\fpqc\bigl( \sAut(H)\bigr)$
\eqref{st:FLG^h=B(AutH)}, with $\FLG^h$ the stratum in \FLG of formal Lie groups
of height $h$, leads us to consider closely the $\FF_p$-group $\sAut(H)$
\eqref{def:AutH}.  We shall devote this section to investigating some aspects of
its structure and of its relation to the groups $\sAut(H^{(n)})$
\eqref{def:AutH^(n)}, $n \geq 1$.  We shall ultimately apply our final
result of this section, \eqref{isom_pro-obs}, to obtain another characterization
of $\FLG^h$ in \s\ref{ss:modstackhth}.

Let us begin with the analog of \eqref{st:AutH^(n)=clsd_sbgp} for $\sAut(H)$. 
Recall the \ZZ-group $\sAut(\wh\AA)$ \eqref{def:Aut_whAA}, and let
$\sAut(\wh\AA)_{\FF_p} := \sAut(\wh\AA)\tensor \FF_p$.  Quite as in
\eqref{st:AutH^(n)=clsd_sbgp}, we obtain the following.

\begin{lem}
$\sAut(H)$ is canonically represented by a closed sub-group scheme of
$\sAut(\wh\AA)_{\FF_p}$. \hfill $\square$
\end{lem}

Quite as in \s\ref{ss:aut_buds_ht_h}, although we will ultimately be interested
in automorphisms of $H$, we shall accord the endomorphisms of $H$ a more
fundamental role.  

\begin{defn}
We define $\sEnd(H)$ to be the presheaf of (noncommutative) rings on
$\Sch{}_{/\FF_p}$
\[
   \sEnd(H)\colon
      S \mapsto \End_{\Gamma(S)}(H)
      \ciso \End_{\flg(S)}\bigl(\wh\AA_S^H\bigr).
\]
\end{defn}

The ring structure on points of $\sEnd(H)$ is described explicitly in terms of
$H$ in \eqref{rk:End_A(F)}.

The $\I_\bullet^{H^{(n)}}$-filtration on $\sEnd(H^{(n)})$ \eqref{def:I^H^(n)}
admits a natural analog for $\sEnd(H)$, as follows.

\begin{defn}
For $i=0$, $1$, $2,\dotsc$, we denote by $\I^H_i$ the subpresheaf of $\sEnd(H)$
defined on points by
\[
   \I^H_i(S) := \biggl\{\, f \in \End_{\Gamma(S)}(H) \biggm|
      \parbox{32.6ex}{\centering $f(T)$ is of the form \\
                    $a_{i+1}T^{i+1} + \text{(higher order terms)}$} 
                    \,\biggr\}.
\]
\end{defn}

Quite as for $\I_i^{H^{(n)}}$, one verifies immediately that $\I^H_i$ is a
presheaf of \emph{$2$-sided ideals} in $\sEnd(H)$ for all $i$, and we have
a decreasing filtration
\[
   \sEnd(H) = \I^H_0 \supset \I^H_1 \supset \I^H_2 \supset \dotsb,
\]
this time of infinite length.


We now wish to introduce the analog for $\sAut(H)$ of the
$\A_\bullet^{H^{(n)}}$-filtration on $\sAut(H^{(n)})$ \eqref{def:A^H^(n)}. We
could do
so by mimicking the definition of the $\A_\bullet^{H^{(n)}}$-filtration in the
obvious way: there is a natural filtration on $\sAut(\wh\AA)$ in plain analogy
with \eqref{rk:AutTT_n_filt}, hence an induced filtration on
$\sAut(\wh\AA)_{\FF_p}$, hence an intersection filtration on $\sAut(H)$.
Instead, we will just use directly the $\I_\bullet^H$-filtration on $\sEnd(H)$.

\begin{defn}\label{def:A^H}
We define $\A_i^H$ to be the subpresheaf of $\sEnd(H)$
\[
   \A_i^H :=
	\begin{cases}
	   \sAut(H), & i=0\\
		T +_H \I_i^H, & i=1,\ 2,\ldots.
	\end{cases}
\]
\end{defn}

Concretely, reasoning as in \eqref{st:I^i->A^i_isom}, $\A_i^H$ is given on
points by
\[
	\A^H_i(S) := \biggl\{\, f \in \Aut_{\Gamma(S)}(H) \biggm|
		\parbox{37ex}{\centering $f(T)$ is of the form \\
						  $T + a_{i+1}T^{i+1} + \text{(higher order terms)}$} 
						  \,\biggr\}.
\]
It is immediate that $\A_i^H$ is a \emph{normal subgroup} in $\sAut(H)$ for 
all $i$, and we have a decreasing filtration
\[
   \sAut(H) = \A_0^H \supset \A_1^H \supset \A_2^H \supset \dotsb.
\]

Let us now turn to the relation between $\sEnd(H)$ and the $\sEnd(H^{(n)})$'s, 
and between $\sAut(H)$ and the $\sAut(H^{(n)})$'s.
For any $m \geq n \geq 1$, truncation of $H$ induces a commutative 
diagram of presheaves of rings
\[\tag{$*$}\label{disp:EndH_diag}
   \vcenter{
	\xymatrix@C-7ex@R-1ex{
   	& \sEnd(H) \ar[rd] \ar[ld]\\
		\sEnd(H^{(m)}) \ar[rr] & & \sEnd(H^{(n)}).
	}
	}
\]

\begin{prop}\label{st:EndH_and_AutH_limits}
For all $i \geq 0$, the diagram \eqref{disp:EndH_diag} induces
\begin{enumerate}
\renewcommand{\theenumi}{\roman{enumi}}
\item\label{it:I_i=lim_I^H^n}
   $\I_i^H \isoarrow \ilim_{n \geq1} \I_i^{H^{(n)}}$, where we take
	$\I_i^{H^{(n)}} := 0$ for $i \geq n$; and
\item\label{it:A_i=lim_A^H^n}
   $\A_i^H \isoarrow \ilim_{n \geq 1} \A_i^{H^{(n)}}$, where we take
	$\A_i^{H^{(n)}} := 1$ for $i \geq n$.
\end{enumerate}
Moreover,
\begin{enumerate}
\renewcommand{\theenumi}{\roman{enumi}}
\setcounter{enumi}{2}
\item\label{it:I_i=lim_I_n}
   $\I_i^H \isoarrow \ilim_{n \geq i} \I_i^H/\I_n^H$ and
\item \label{it:A_i=lim_A_n}
   $\A_i^H \isoarrow \ilim_{n \geq i} \A_i^H/\A_n^H$.
\end{enumerate}
In particular, $\sEnd(H)$ (resp.\ $\sAut(H)$) is complete and separated with 
respect to the $\I^H_\bullet$- (resp.\ $\A^H_\bullet$-) topology.
\end{prop}

\begin{proof}
Before anything else, it is clear from the definitions that truncation carries
$\I_i^H$ and $\I_i^{H^{(m)}}$ into $\I_i^{H^{(n)}}$, $m \geq n$, so that the 
limit and arrow in \eqref{it:I_i=lim_I^H^n} are well-defined; and analogously for 
\eqref{it:A_i=lim_A^H^n}.

\eqref{it:I_i=lim_I^H^n}\: The case $i = 0$ is clear because $\wh\AA^H$
\eqref{eg:trivfg} is ind-infinitesimal: precisely, use \eqref{st:FLG=limBn},
\eqref{rk:Fib_obwise}, and \eqref{rk:lims_via_cat_lims}.  The case $i > 0$ is
then clear because, for all $n \geq i$, the inverse image of $\I_i^{H^{(n)}}$ in
$\sEnd(H)$ is $\I_i^H$.

\eqref{it:A_i=lim_A^H^n}\: Immediate from \eqref{it:I_i=lim_I^H^n} and, when 
$i > 0$, from \eqref{st:I^i->A^i_isom}.

\eqref{it:I_i=lim_I_n}\: Immediate from \eqref{it:I_i=lim_I^H^n}, since for $n
\geq i$, $\I_i^H/\I_n^H$ identifies with the image of $\I_i^H$ in
$\I_i^{H^{(n)}}$.

\eqref{it:A_i=lim_A_n}\: Immediate from \eqref{it:A_i=lim_A^H^n}, since for $n
\geq i$, $\A_i^H/\A_n^H$ identifies with the image of $\A_i^H$ in
$\A_i^{H^{(n)}}$.
\end{proof}

As a consequence of the proposition and of our earlier calculation of the the
successive quotients of the $\I_\bullet^{H^{(n)}}$-filtration
\eqref{st:EndH^(n)_succ_quot}, we now obtain the successive quotients of the
$\I_\bullet^H$-filtration.  For any $i$ and any $n \geq i+1$, we have
monomorphisms 
\[\tag{$**$}\label{disp:I^H_succ_quot->G_a} 
   \I_i^H/\I_{i+1}^H 
	   \inj \I^{H^{(n)}}_i / \I^{H^{(n)}}_{i+1} 
		\inj
		\begin{cases}
			\OO, & i = 0;\\
			\GG_a, & i > 0;\\
		\end{cases}
\]
plainly the composite is independent of the choice of $n$.

\begin{cor}\label{st:EndH_succ_quot}
The diagram \eqref{disp:I^H_succ_quot->G_a} induces an identification of 
presheaves
\[
   \I^H_i / \I^H_{i+1} \ciso
   \begin{cases}
   	\OO^{\Fr_{\smash{p^h}}}, & i = 0;\\
   	\GG_a^{\Fr_{p\smash{^h}}}, & i = p-1,\ p^2-1,\ p^3-1,\dotsc;\\
   	0, & \text{otherwise.}
   \end{cases}
\]
\end{cor}

\begin{proof}
Fix $i$.  For any $n \geq i+1$, we have an exact sequence of presheaves
\[
   0 
	   \to \I^{H^{(n)}}_{i+1} 
		\to \I^{H^{(n)}}_i 
		\to \I^{H^{(n)}}_i / \I^{H^{(n)}}_{i+1}
		\to 0.
\]
It follows from \eqref{st:EndH^(n)_succ_quot} that
\begin{itemize}
\item
	$(\I^{H^{(n)}}_{i+1})_{n \geq i+1}$ satisfies the
	Mittag-Leffler condition as a diagram of pre\-sheaves of abelian groups; and
\item 
   as $n$ increases, $\I^{H^{(n)}}_i / \I^{H^{(n)}}_{i+1}$ is eventually
	constant of the asserted value.
\end{itemize}
Now take the limit over $n$ and use \eqref{st:EndH_and_AutH_limits}.
\end{proof}

In an entirely similar fashion, using \eqref{st:AutH^(n)_succ_quot} in place of
\eqref{st:EndH^(n)_succ_quot}, and using the Mittag-Leffler condition for
not-necessarily-abelian groups, we obtain the successive quotients
of the $\A_\bullet^H$-filtration.

\begin{cor}\label{st:AutH_succ_quots}
We have an identification of presheaves
\begin{xxalignat}{3}
   \parbox{5ex}{\hfill \\ \hfill \\ \phantom{$\square$}} & &
	\A^H_i / \A^H_{i+1} &\ciso
	\begin{cases}
		\mu_{p^h-1}, & i = 0;\\
		\GG_a^{\Fr_{p\smash{^h}}}, & i = p-1,\ p^2-1,\ p^3-1,\dotsc;\\
		0, & \text{otherwise.}
	\end{cases}
	& &
	\parbox{5ex}{\flushright \hfill \\ \smallskip
	             \hfill \\ \smallskip
					 $\square$}
\end{xxalignat}
\end{cor}

In the rest of the section we shall study the following quotient groups, which
appear in \eqref{st:EndH_and_AutH_limits}, and their relation to the 
$\sEnd(H^{(n)})$'s and $\sAut(H^{(n)})$'s.

\begin{defn}\label{def:E_n_and_U_n}
We define $\E_n^H$ to be the presheaf quotient ring $\sEnd(H)/\I_n^H$, and 
$\U_n^H$ to be the subpresheaf of units in $\E_n^H$.
\end{defn}

In other words, by \eqref{st:noncomm}, $\U_n^H \ciso \sAut(H)/\A_n^H$.

\begin{rk}\label{rk:E_n_and_U_n_explicit}
By \eqref{st:EndH_succ_quot} and \eqref{st:AutH_succ_quots}, $\E_n^H$ and
$\U_n^H$ can be obtained from finitely many iterated extensions of finite 
\'etale groups. Hence both are
finite \'etale over $\Spec \FF_p$.  In fact, it is easy to write down explicit
representing schemes.  To fix ideas, consider $\E_n^H$.  For all $i \geq
0$, the exact sequence of presheaves
\[
   0 \to \I_{i+1}^H \to \I_i^H \xra\can \I_i^H/\I_{i+1}^H \to 0
\]
has representable cokernel.  Hence the quotient map ``$\can$'' admits a section
in the category of \emph{set-valued} presheaves.  Hence $\I_i^H \iso \I_{i+1}^H
\times (\I_i^H/\I_{i+1}^H)$ as presheaves of sets.  Now, the possible nontrivial
values of $\I_i^H/\I_{i+1}^H$, namely $\OO^{\Fr_{p\smash{^h}}}$ and
$\GG_a^{\Fr_{p\smash{^h}}}$, both have underlying scheme $\Spec
\FF_p[T]/(T^{p^h} - T)$.  Hence, letting $l$ denote the integer such that $p^l
\leq n < p^{l+1}$, we deduce that $\E_n^H$ is representable by
\[\tag{$\sharp$}\label{disp:E_n^H_representing_scheme}
   \Spec\FF_p[T_0,\dotsc,T_l]/
	   (T_0^{p^h}-T_0^{\vphantom{p^h}},\dotsc,T_l^{p^h}-T_l^{\vphantom{p^h}}).
\]
We can even specify a natural representation: $S$-points of
\eqref{disp:E_n^H_representing_scheme} are canonically identified with ordered
$(l+1)$-tuples of elements $a\in\Gamma(S)$ satisfying $a^{p^h} = a$, and we can
take the map from $\E_n^H$ to \eqref{disp:E_n^H_representing_scheme} specified
on points by sending the class of $f(T)$ to the coefficients of $T$,
$T^p,\dotsc$, $T^{p^l}$.

Similarly, $\U_n^H$ is representable by 
\[
   \Spec\FF_p[T_0^{\vphantom{-1}},T_0^{-1},T_1^{},\dotsc,T_l^{}]/ 
	   (T_0^{p^h} - T_0^{\vphantom{p^h}}, \dotsc, T_l^{p^h} - T_l^{\vphantom{p^h}}).
\]
\end{rk}

%
%

\begin{rk}\label{rk:O_D=End(H)}
Let us digress for a moment to make a remark on the $\U_n^H$'s.  Let $D$ denote
the central division algebra over $\QQ_p$ of dimension $h^2$ and Hasse invariant
$\frac 1 h$.  Let $\O_D$ denote the maximal order in $D$.  Then a classical
theorem of Dieudonn\'e \cite{dieu57}*{Th\'eor\`eme 3} and Lubin \cite{lub64}*{5.1.3} in the theory of
formal group laws asserts that $\O_D \iso \End_{\FF_{p^h}}(H)$ as topological
rings, where $\End_{\FF_{p^h}}(H)$ has the $\I_\bullet^H(\FF_{p^h})$-topology;
precisely, one has $p^r\O_D \iso \I_{p^{rh}-1}^H(\FF_{p^h})$ for all $r\geq 0$.
Hence $\O_D^\times \iso \Aut_{\FF_{p^h}}(H) \ciso \ilim_n \U_n^H(\FF_{p^h})$ as
pro-finite groups.

The finite algebraic group $\U_n^H$ and the abstract finite group
$\U_n^H(\FF_{p^h})$ are closely related: indeed, the former is a \emph{twist}
over $\Spec \FF_p$ of the latter.  Precisely, for any abstract group $G$ and
ring $A$, write $G_A$ for the corresponding constant group scheme over $\Spec
A$.  Then $\U_n^H$ is not constant over $\Spec \FF_p$, but it becomes isomorphic
to $\bigl(\U_n^H(\FF_{p^h})\bigr)_{\FF_{p^h}}$ after the base change $\Spec
\FF_{p^h} \to \Spec \FF_p$, as we see very explicitly from
\eqref{rk:E_n_and_U_n_explicit}.
\end{rk}

Our work so far furnishes a number of immediate relations between the $\E_n^H$'s
and the $\sEnd(H^{(n)})$'s, and between the $\U_n^H$'s and the
$\sAut(H^{(n)})$'s.  To fix ideas, let us consider the $\E_n^H$'s and the
$\sEnd(H^{(n)})$'s.  For example, for all $n \geq 1$, $\E_n^H$ is identified
with the image of $\sEnd(H)$ in $\sEnd(H^{(n)})$.  And by
\eqref{st:EndH_and_AutH_limits}, the $\E_n^H$'s and the $\sEnd(H^{(n)})$'s have
the same limit, namely $\sEnd(H)$, endowed with the same topology.  Our final
goal for the section will be to show that a yet stronger statement holds:
namely, that the $\E_n^H$'s and the $\sEnd(H^{(n)})$'s determine
\emph{isomorphic pro-objects}; and similarly for the $\U_n^H$'s and the
$\sAut(H^{(n)})$'s.

Precisely, let $\prolim_{\!\!\!\!n} \E_n^H$ be the pro-ring scheme obtained from
the diagram
\[
	\cdots \to \E_3^H \to \E_2^H \to \E_1^H,
\]
and $\prolim_{\!\!\!\!n} \sEnd(H^{(n)})$ be the pro-ring scheme obtained from
the diagram
\[
	\cdots \to \sEnd(H^{(3)}) 
		\to \sEnd(H^{(2)}) 
		\to \sEnd(H^{(1)}).
\]
The natural inclusions $\E_n^H \inj \sEnd(H^{(n)})$ for $n \geq 1$ plainly 
determine a morphism of pro-objects
\[
	\alpha\colon \prolim_n \E_n^H 
		\to \prolim_n \sEnd(H^{(n)}).
\]

We shall show that $\alpha$ is an isomorphism by exhibiting an explicit inverse
$\beta$.  To define $\beta$, we must define $\beta_n\colon \prolim_{\!\!\!\!  m}
\sEnd(H^{(m)}) \to \E_n^H$ for each $n \geq 1$.  For this, let $l$ be the
integer such that $p^l \leq n < p^{l+1}$, and take any $m \geq p^{l+h}$.
Consider the natural map 
\[\tag{$\flat$}\label{disp:beta_n_prelim}
   \sEnd(H^{(m)}) \to \sEnd(H^{(n)})
\]
induced by truncation.  By \eqref{st:EndH^(n)_succ_quot},
\eqref{st:EndH_succ_quot}, and choice of $m$, the image of
\eqref{disp:beta_n_prelim} in $\sEnd(H^{(n)})$ identifies with $\E_n^H$.  Hence
\eqref{disp:beta_n_prelim} induces $ \prolim_{\!\!\!\!  m} \sEnd(H^{(m)}) \to
\E_n^H$, which we take as the desired $\beta_n$.  It is clear that the
$\beta_n$'s are compatible as $n$ varies, so that we obtain the desired $\beta$.

Analogously, we may form the pro-algebraic groups 
\[
   \prolim_n \U_n^H
	\quad\text{and}\quad
   \prolim_n \sAut(H^{(n)}),
\]
and we obtain morphisms
\[
   \prolim_n \U_n^H  \xrla{\alpha'}{\beta'} \prolim_n \sAut(H^{(n)}).
\]

\begin{thm}\label{isom_pro-obs}
The morphisms $\alpha$ and $\beta$ (resp., $\alpha'$ and $\beta'$) are inverse 
isomorphisms of pro-objects.
\end{thm}

\begin{proof}
Everything is elementary from what we've already said.
\end{proof}

\subsection{Functoriality of quotient stacks}\label{ss:copreshB(G)'s}

In the next section we shall need to interpret \eqref{isom_pro-obs} in terms of
the classifying stacks $B(\U_n)$ and $B\bigl(\sAut(H^{(n)})\bigr)$.  We shall
now pause a moment to record the following fact for use 
then: let \C be a site, and let \D denote the category of paris $(G,X)$, where 
$X$ is a sheaf on \C and $G$ is a group sheaf on \C acting on $X$ (on the left,
say). Then passing 
to the quotient stack defines a morphism (in the sense of bicategories; see
\s\ref{ss:bicat_morphs}) from \D to the $2$-category of stacks over \C.

Roughly, the essential observation is simply that stackification defines a bicategory
morphism $\Fib(\C) \to \St(\C)$.  More to the point, note that we have a canonical 
$2$-functor $\D \to \Fib(\C)$ sending $(G,X)$ to the presheaf of groupoids
\[
   G \times X \xrra{\pr_X}{a} X,
\]
where $a$ denotes the action map.  Hence the result of composing with 
stackification is to send $(G,X) \mapsto G\bs X$.


In particular, taking $X$ to be the sheaf with constant value $\{*\}$,
we see that $G \mapsto B(G)$ defines a morphism from group sheaves on \C to
stacks.

\subsection{The stratum of height $h$ formal Lie groups II}\label{ss:modstackhth}

In this section we apply the work of the previous two sections to give another
characterization of the stack $\FLG^h$ of formal Lie groups of height $h$, $h
\geq 1$.  
Recall the algebraic groups $\U_n^H$, $n \geq 1$, of \eqref{def:E_n_and_U_n}.

\begin{thm}\label{st:FLG^h=limU_n}
$\dsp \FLG^h \approx \ilim_n B_\fet(\U_n^H)$.
\end{thm}

\begin{proof}
The proof just consists of stringing together some of our previous results.  By
\s\ref{ss:copreshB(G)'s} and \eqref{st:proF}, the isomorphism of pro-objects
\[
   \prolim_{n} \sAut(H^{(n)})
	   \isoarrow \prolim_{\!\!\!\!n} \U_n^H
\]
from \eqref{isom_pro-obs} induces an equivalence of stacks
\[
   \ilim_n B_\fet \bigl( \sAut(H^{(n)}) \bigr) 
	   \xra\approx \ilim_n B_\fet(\U_n^H).
\]
By \eqref{st:Bn^h=B(AutH^(n))}, we have an equivalence $\Bn^h \approx B_\fet
\bigl( \sAut(H^{(n)}) \bigr)$ for $n \geq p^{h+1}$, plainly compatible with
truncation on the $\Bn^h$ side and with the transition maps induced by
$\prolim_{\!\!\!\!n} \sAut(H^{(n)})$ on the $B_\fet \bigl( \sAut(H^{(n)})
\bigr)$ side.  Now use \eqref{st:FLG^h=lim}.
\end{proof}

%
%
%
%

\begin{rk}\label{rk:O_D=End(H).analog}
One may consider the equivalences
\[
   B_\fpqc\bigl(\sAut(H)\bigr) \approx \FLG^h \approx \ilim_n B_\fet(\U_n^H)
\]
combined from \eqref{st:FLG^h=B(AutH)} and \eqref{st:FLG^h=limU_n} to be a
stack analog of the theorem $\O_D^\times \iso \Aut_{\FF_q}(H)$ discussed in
\eqref{rk:O_D=End(H)}.  Indeed, $\U_n^H$ becomes constant after the base change
$\Spec \FF_{p^h} \to \Spec \FF_p$, and we obtain equivalences over $\FF_{p^h}$
\[
   B_\fpqc\bigl(\sAut(H)_{\FF_{p^h}}\bigr)
      \approx \ilim_n B\bigl((\U_n^H)_{\FF_{p^h}}\bigr)
      \approx \ilim B(\O_D^\times/N),
\]
where the limit on the right runs through the open normal subgroups $N$ of
$\O_D^\times$.
\end{rk}

\section{Valuative criteria}\label{s:val_crit}
In this section we shall conduct a basic investigation of some properties of the
stacks \FLG and \Bn, $n\geq 1$, related to valuative criteria.  As in previous
sections, we work with the notion of height relative to fixed prime $p$.

\begin{thms}\label{thm:Bn.univ.closed}
$\Bn$ is universally closed over $\Spec \ZZ$, and for all $h \geq 1$ and $n \geq
p^h$, $\Bn^{\geq h}$ is universally closed over $\Spec \FF_p$.
\end{thms}

\begin{proof}
The proof is the same in all cases, so let's just consider \Bn over $\Spec \ZZ$.
We apply the valuative criterion in \cite{lamb00}*{7.3}.  Let $\O$ be a
valuation ring and $K$ its field of fractions.  Let $X$ be an $n$-bud
over $K$.  Then $X$ admits a coordinate, so we may assume $X$ is given by a bud
law 
\[
   F(T_1,T_2) = T_1 + T_2 + \sum_{2\leq i+j \leq n} a_{ij}T_1^iT_2^j,
	\qquad a_{ij} \in K.
\]
For changes of coordinate of the form $f(T) = \lambda T$ for nonzero $\lambda
\in K$,
we obtain
\[
   f\bigl[ F\bigl(f^{-1}(T_1), f^{-1}(T_2)\bigr) \bigr]
		= T_1 + T_2 + \sum_{2\leq i+j \leq n} a_{ij}\lambda^{1-i-j} T_1^iT_2^j.
\]
So, by taking $\lambda$ of sufficiently negative valuation, we see that $F$ is 
$K$-isomorphic to a bud law defined over \O.
\end{proof}

\begin{rks}\label{rk:sep}
\Bn is not \emph{proper} over \ZZ because it is not separated.  Indeed, let \O 
be a valuation ring with fraction field $K$.  Then the natural functor
\[
   \Bn(\O) \to \Bn(K)
\]
is faithful but not full.  For example, for the additive $n$-bud
$\GG_a^{(n)}$ \eqref{eg:bud_eg's} we have
\[
   \Aut_\O (\GG_a^{(n)})
      \ctndneq \Aut_K (\GG_a^{(n)}),
\] 
since
the latter contains automorphisms of the form $f(T) = \lambda T$ for $\lambda$
of nonzero valuation.

Similarly, $\Bn^{\geq h}$ is not separated over $\FF_p$.
\end{rks}

\begin{egs}
The following may be taken as an exhibition of the non-sep\-ar\-atedness of \Bn
and of \FLG.  Let \O be a DVR with uniformizing element $\pi$ and residue
field of positive characteristic. Then the group law $F(T_1,T_2) := T_1 + T_2 +
\pi T_1T_2$ determines a formal Lie
group over $\Spec \O$.  Let $f(T) := \pi T$. Then, over the generic point
$\eta$, we have
\[
   f\bigl[ F\bigl(f^{-1}(T_1), f^{-1}(T_2)\bigr) \bigr)]
      = T_1 + T_2 + T_1T_2.
\]
Hence $f$ specifies an isomorphism $\wh\AA_\eta^F \isoarrow \wh\GG _m$.  But
$\wh\GG _m$ is
certainly not isomorphic to $\wh\AA^F$ over $\Spec\O$, since $\wh\AA^F_\O$
reduces to $\wh\GG_a$ at the closed point.  Hence $\wh\GG_m$
admits nonisomorphic extensions from the generic point to $\Spec\O$.
\end{egs}

The failure of \Bn and of $\Bn^{\geq h}$ to be separated prevents one from 
concluding formally that the valuative criterion used in the proof of
\eqref{thm:Bn.univ.closed} holds for $\FLG$ and for $\FLG^{\geq h}$, 
respectively.  Nevertheless, these stacks do satisfy a kind of ``formal 
universal closedness'', in the following sense.

\begin{thms}\label{st:fl_u_clsd}
Let \O be a valuation ring with field of fractions $K$.
\begin{enumerate}
\renewcommand{\theenumi}{\roman{enumi}}
\item \label{it:char_0_ess_surj}
   If $K$ has characteristic $0$, then $\FLG(\O) \to \FLG(K)$ is essentially 
	surjective.
\item \label{it:char_p_ess_surj}
   If $K$ has characteristic $p$ and is \emph{separably closed}, then 
	$\FLG^{\geq h}(\O) \to \FLG^{\geq h}(K)$ is essentially surjective.
\end{enumerate}
\end{thms}

\begin{proof}
\eqref{it:char_0_ess_surj}\: As is well-known, over a \QQ-algebra, every formal
group law is isomorphic to the additive law.

\eqref{it:char_p_ess_surj}\: By Lazard's theorem \cite{laz55}*{Th\'eor\`eme IV},
formal group laws over separably closed fields of characteristic $p$ are
classified up to isomorphism by their height.  Now use that group laws of every
height are defined over $\FF_p$, hence over \O.
\end{proof}

Our remarks in \eqref{rk:sep} suggest that the failure of $\Bn^{\geq h}$ to be
separated is tied to the additive $n$-bud, which has ``height $\infty$''.  So it
is natural to ask if the \emph{stratum} $\Bn^h$ is separated.  But the answer
here is also negative: by \eqref{st:Bn^h=B(AutH^(n))},
\eqref{st:AutH^(n)=clsd_sbgp}, and \eqref{st:AutH^(n)=sm_dim_h}, $\Bn^h$ is the
classifying stack of a group $\sAut(H^{(n)})$ which is positive dimensional and
affine,
so that $\sAut(H^{(n)})$
is not proper, so that $B\bigl(\sAut(H^{(n)})\bigr)$ is not separated
\cite{lamb00}*{7.8.1(2)}. There is, however, a positive result when we take the
limit over $n$.

\begin{thms}\label{st:fl_sep}
Let \O be a valuation ring and $K$ its field of fractions.  Then $\FLG^h(\O) \to 
\FLG^h(K)$ is fully faithful for all $h \geq 1$.
\end{thms}

\begin{proof}
Of course, the assertion only has content when $\charac K = p$, since
otherwise $\FLG^h(\O) = \FLG^h(K) = \emptyset$.  So assume $\charac K = p$.
By \eqref{st:FLG^h=limU_n}, $\FLG^h \approx \ilim_n B_\fet(\U_n^H)$, where
$\U_n^H$ is the finite \'etale group scheme over $\FF_p$ of
\eqref{def:E_n_and_U_n}. In particular, $\U_n^H$ is proper. Hence
$B_\fet(\U_n^H)$ is a separated algebraic stack over $\FF_p$
\cite{lamb00}*{7.8.1(2)}. Hence $B(\U_n^H)(\O) \to B(\U_n^H)(K)$ is fully
faithful. Now use that a limit of fully faithful maps is fully faithful
\eqref{rk:lim_of_ff's}.
\end{proof}


\begin{rks}
As noted in the introduction, when \O is a \emph{discrete} valuation ring,
\eqref{st:fl_sep} is a special case of de Jong's theorem
that, when $\charac K = p$, the base change functor
\[\label{disp:restn_fctr}\tag{$*$}
   \biggl\{\parbox{25ex}{\centering $p$-divisible groups and\\
                 homomorphisms over \O}\biggr\}
   \to
   \biggl\{\parbox{25ex}{\centering $p$-divisible groups and\\
                 homomorphisms over $K$}\biggr\}
\]
is fully faithful \cite{dj98}*{1.2}.  (Tate proved that \eqref{disp:restn_fctr} is fully faithful
when $\charac K = 0$ \cite{tate67}*{Theorem 4}.)  Note that \eqref{st:fl_sep}
only asserts bijections between $\Isom$ sets of objects, not $\Hom$
sets, as in de Jong's theorem.  But it appears that the methods used to prove
\eqref{st:fl_sep} extend to give bijections between $\Hom$ sets, provided one
considers stacks of \emph{categories}, not just stacks of groupoids.
\end{rks}

\renewcommand{\thesection}{A}
\section{Appendix}



This appendix is devoted to some of the basic aspects of limits in
$2$-categorical contexts; or more precisely, to limits in \emph{bicategories}. 
From the point of view of general bicategory theory, the applications we shall
have in the main body are all of a rather simple sort, and it would doubtless
require less labor on the whole to treat them in a more ad hoc fashion as they
arise. But we believe that bicategories afford a convenient and natural
setting in which to understand many of the various notions at play, and we have
chosen to take a little time to work through some of the foundations.

Almost everything we shall discuss is probably well-known, but for some of the
material --- notably that in \s\s\ref{ss:lim_stacks}--\ref{ss:fin_lims} --- we have failed to find suitable references.  We extend our apologies to those whose work we have overlooked.

We have chosen to systematically ignore certain foundational set-theoretic issues typically resolved through consideration of universes or of regular cardinals.

\subsection{Bicategories}

A bicategory is a categorification of an (ordinary) category.  We shall decline
to recall the precise definition; see, for example, \cite{cwm}*{XII \s6},
\cite{borc94}*{7.7.1}, or --- the original reference --- \cite{ben67}*{\s1}
for complete details.  Let us instead content ourselves with a few informal
remarks.  A bicategory consists of objects; morphisms between objects, called
$1$-cells; and morphisms between $1$-cells with common source and target, called
$2$-cells.  The generalization from categories to bicategories can be understood
to a large extent in terms of the familiar generalization from sets to
categories.  In an ordinary category, morphisms $A \to B$ form a set
$\Hom(A,B)$; whereas in a bicategory, $1$-cells $A\to B$ and the $2$-cells
between them form a \emph{category} $\Hom(A,B)$.  In an ordinary category,
composition is specified by a function $\Hom(A,B) \times \Hom(B,C) \to
\Hom(A,C)$; whereas in a bicategory, composition is a \emph{functor} $\Hom(A,B)
\times \Hom(B,C) \to \Hom(A,C)$.  In an ordinary category, associativity is
expressed as an equality of functions
\[\label{disp:assoc_cond}\tag{\dag}
   \Hom(A,B) \times \Hom(B,C) \times \Hom(C,D) \rra \Hom(A,D).
\]
In a bicategory, the arrows in \eqref{disp:assoc_cond} become
functors between categories. But in category theory, it is typically unnatural
to demand that two functors be equal on the nose.  A better notion of sameness
is that the functors be naturally isomorphic.  So in a bicategory, the functors
in \eqref{disp:assoc_cond} agree up to specified natural isomorphism; that is,
the isomorphism is specified as part of the data of the bicategory.  Similarly,
each object $C$ in a bicategory is equipped with a distinguished $1$-cell $\id_C
\in\ob\Hom(C,C)$, but left and and right composition with $\id_C$ is the
identity functor only up to respective specified ``left identity'' and ``right
identity'' natural isomorphisms.  Moreover, the associativity and identity
isomorphisms are not taken to be arbitrary, but are themselves subject to
natural coherence constraints.  Philosophically, one may view the associativity
and identity isomorphisms as ``canonical'': the coherence constraints ensure
that all possible ways to use the associativity and identity isomorphisms to
obtain a $2$-cell between given $1$-cells yield a common result.

\begin{rk}\label{rk:2-cat}
A bicategory with strict associativity and strict identity $1$-cells is a
\emph{$2$-category}.  Many examples of honest $2$-categories occur in practice,
e.g.\ \Cat.
\end{rk}

\begin{rk}\label{rk:2-cat/bicat}
From some points of view, it would be more natural to use the term
``$2$-category'' in place of ``bicategory''; under such usage, one would then
refer to the objects described in \eqref{rk:2-cat} as ``strict'' $2$-categories.
But we shall defer to the now well-established conventions of category theory
and use ``$2$-category'' for the strict notion and ``bicategory'' for the more
general, ``up-to-isomorphism'' notion.
\end{rk}

\begin{rk}
We shall often refer to ordinary categories as $1$-categories.
 Just as a set may be regarded as a category in which all morphisms are
identities, a $1$-category may be regarded as a bicategory, and indeed
$2$-category, in which all $2$-cells are identities.
\end{rk}

A $1$-cell $f\colon A \to B$ in a bicategory is \emph{fully faithful} if for every object $C$, the natural functor 
\[
   \Hom(C,A) \xra{f_*} \Hom(C,B)
\]
is fully faithful.  In the case of the bicategory \Cat, we recover the usual notion of a fully faithful functor between  $1$-categories.

A $1$-cell $f\colon A \to B$ in a bicategory is an \emph{equivalence} if there
exists a $1$-cell $g\colon B \to A$ such that $gf \iso \id_A$ and $fg \iso
\id_B$. Equivalently, $f$ is an equivalence $\iff$ for every object $C$, the
natural functor
\[
   \Hom(C,A) \xra{f_*} \Hom(C,B)
\]
is an equivalence of categories $\iff$ for every object $C$, the natural functor
\[
   \Hom(B,C) \xra{f^*} \Hom(A,C)
\]
is an equivalence of categories.

An object $A$ in a bicategory is \emph{groupoidal} if for every object $C$, $\Hom(C,A)$ is a groupoid.  In \Cat, the groupoidal objects are precisely the groupoids.

For any bicategory \C, one obtains a bicategory $\C^\opp$ in a natural way
by reversing $1$-cells.  There are two other notions of ``opposite'':  one may
reverse just the $2$-cells, or the $1$- and $2$-cells simultaneously.  We shall
have no use for these latter two.

\subsection{Morphisms of bicategories}\label{ss:bicat_morphs}

Part of the philosophy of category theory is that morphisms between
mathematical objects are just as important as, if not more than, the objects
themselves.  So we shall now spend a few words on morphisms between bicategories. 
As in the previous section, we shall conduct the discussion at an informal level
and leave precise details to the references.  Let \I and \C and be bicategories.

In rough form, a \emph{morphism}, or \emph{diagram}, $F\colon \I \to \C$
consists of
\begin{itemize}
\item
   an assignment on objects $\ob\I \to \ob\C$; and
\item
   for every pair $i$, $j\in\ob\I$, a functor $\Hom_\I(i,j) \to \Hom_\C(Fi,Fj)$,
\end{itemize}
compatible with identity $1$-cells and composition of $1$-cells up to specified
coherent invertible $2$-cells.  We shall decline to make ``coherent'' more
precise; see \cite{ben67}*{4.1} for the full definition.

\begin{rk}
Our notion of morphism is often called ``homomorphism'' in the literature;
see e.g.\ \cite{ben67}*{4.2}.  More generally, one may consider a notion 
in which assignment of $1$-cells is compatible with identities and
composition only up to not-necessarily-invertible coherent $2$-cells: it is
this that many authors call ``morphism'', whereas we would call it \emph{lax
morphism}.  The lax morphisms occupy an important place in the theory, but we
will not have occasion to consider them further.
\end{rk}

\begin{eg}[Corepresentable and representable morphisms]\label{eg:repble_morph}
For each $C\in\ob\C$, $h^C := \Hom_\C(C,-)$ defines a morphism
\[
   \xymatrix@R=0ex{
   \smash{\C}\vphantom{()} \ar[r] & \Cat\\
   \smash{D}\vphantom{()}  \ar@{|->}[r] & \smash{\Hom_\C(C,D)}\vphantom{()}
   }
\]
in a natural way.  We call $h^C$ the \emph{morphism corepresented by
$C$}.  Similarly, $h_C := \Hom(-,C)$ defines a morphism $\C^\opp \to \Cat$;
$h_C$ is the \emph{morphism represented by $C$}.
\end{eg}

\begin{rk}
When \I is a $1$-category, a bicategory morphism $\I \to \C$ is essentially a
pseudofunctor as defined in \cite{sga1}*{VI \s8} (strictly speaking,
\cite{sga1} treats the case $\C = \Cat$, but the modifications needed for 
general \C are slight).  In the main body, we shall only encounter morphisms
for which \I is a $1$-category, so let us give a more explicit description in
this case.  Such a morphism $F\colon \I \to \C$ consists of the data
\begin{itemize}
\item
	for every object $i$ in \I, an object $Fi$ in \C;
\item
	for every morphism $\mu\colon i\to j$ in \I, a $1$-cell $F \mu \colon F i\to F j$ in \C;
\item
	for every $i\in\ob \I$, an invertible $2$-cell $\chi_i \colon F \id_i \isoarrow \id_{F i}$ in \C; and
\item
	for every composition $\smash{i \xrightarrow\mu j \xrightarrow\nu k}$ in
	\I, an invertible $2$-cell
	\[
		\chi_{\nu,\mu}\colon \F(\nu\mu) \isoarrow (\F\nu)(\F\mu)
	\]
	in \C.
\end{itemize}
These data are subject to natural coherence conditions expressing left and right identity and associativity constraints on the $\chi$'s.  We shall decline to write out the constraints precisely.

In practice, the $\chi$'s typically arise as canonical isomorphisms.  So, to
simplify notation, we often just write same symbol ``$\can$'' in place
of the various $\chi$'s.

When \C is a $2$-category and one wishes to construct a morphism into \C, it can often be arranged that $F\id_i = \id_{Fi}$ and $\chi_i$ is the identity $2$-cell for each $i\in\ob\I$.
\end{rk}

\begin{rk}
Morphisms of bicategories carry equivalences of objects to equivalences. A
morphism of bicategories is itself called an \emph{equivalence} if it is
surjective on equivalence classes of objects and induces equivalences on all
$\Hom$ categories.
\end{rk}

\begin{rk}
As in \eqref{rk:2-cat/bicat}, there is a reasonable case for calling a morphism of 
bicategories a ``$2$-functor''.  But we shall again defer to standard usage 
in category theory and reserve ``$2$-functor'' for a morphism between 
$2$-categories which is strictly compatible with composition and with identities.
\end{rk}

Just as functors between $1$-categories admit natural transformations between
them, so do morphisms between bicategories admit arrows, which we call just
\emph{transformations}, between them.  In rough form, if $F$, $G\colon \I \to
\C$ are bicategory morphisms, then a transformation $\alpha\colon F \to G$
consists of a $1$-cell $\alpha_i \colon Fi \to Gi$ in \C for each $i\in\ob\I$,
compatible with the definitions of $F$ and $G$ on $1$-cells up to specified
coherent invertible $2$-cells.

\begin{rk}
We would obtain the notion of \emph{lax transformation} by requiring
compatibility only up to not-necessarily-invertible coherent $2$-cells.  Some
authors use ``transformation'' for our notion of lax transformation, and
``strong transformation'' for our notion of transformation.  As for lax
morphisms between bicategories, we shall not have occasion to consider lax
transformations further.
\end{rk}

There is one further notion of arrow to consider, new to the setting of
bicategories: if $\alpha$, $\beta\colon F \to G$ are transformations, then a
\emph{modification} $\xi\colon \alpha \to \beta$ consists of a
(not-necessarily-invertible) $2$-cell
\[
   \xy
      (-10,0)*+{Fi}="F";
      (10,0)*+{Gi}="G";
      {\ar@/^3ex/ "F";"G" ^-{\alpha_i}};
      {\ar@/_3ex/ "F";"G" _-{\beta_i}};
      {\ar@{=>} (0,3);(0,-3) ^-{\xi_i}}
   \endxy
\]
in \C for each $i\in\ob\I$, subject to a natural coherence condition.

Just as functors from one category to another and the natural transformations
between these functors are naturally organized into a category, the bicategory
morphisms, transformations, and modifications are naturally organized as the
respective objects, $1$-cells, and $2$-cells of a bicategory $\Hom(\I,\C)$.
When \C is a $2$-category, so is $\Hom(\I,\C)$.

\begin{rk}\label{rk:equivs_in_morph_bicat}
Let $F$, $G\colon \I \to \C$ be morphisms.  When $\I$ and $\C$ are $1$-categories, it is a familiar fact that a natural transformation $F \to G$ is an isomorphism exactly when it is an objectwise isomorphism, i.e.\ when $Fi \to Gi$ is an isomorphism for all $i\in\ob\I$.  Analogously, when $\I$ and $\C$ are arbitrary bicategories, a transformation $F \to G$ is an equivalence exactly when it is an objectwise equivalence.
\end{rk}

\begin{rk}[Yoneda's lemma]
Yoneda's lemma admits the following formulation in the setting of
bicategories:  for every object $C$ in \C and bicategory morphism
$F\colon\C^\opp \to \Cat$, the natural functor $\Hom_{\Hom(\C,\Cat)}(h_C,F) \to
FC$ is an equivalence of categories.  The analogous statement holds for $h^C$.
\end{rk}

\begin{rk}[Diagonal morphism]\label{rk:diag}
There is a canonical morphism of bicategories $\C \to \C^\I := \Hom(\I,\C)$,
called the \emph{diagonal} and denoted $\Delta_{\I,\C}$ or often just
$\Delta$, which we now describe.

As a warm-up, let us review the diagonal in the setting of $1$-categories \I and
\C; see e.g.\ \cite{cwm}*{III \s3}.  The diagonal is defined on each object
$C$ in $\C$ by taking $\Delta C$ to be the functor $\I \to \C$ with constant
value $C$ on objects and constant value $\id_C$ on morphisms.  The diagonal is
defined on morphisms in the evident way.

When \I and \C are bicategories, the diagonal is defined quite analogously.  For each $C \in \ob\C$, there is a natural morphism $\I \to \C$ sending all objects to $C$, all $1$-cells to $\id_C$, and all $2$-cells to $\id_{\id_C}$; this defines $\Delta$ on objects.  There is then an evident way to define $\Delta$ on $1$-cells and $2$-cells in \C, and $\Delta$ admits a natural structure of bicategory morphism $\C \to \Hom(\I,\C)$.
\end{rk}

\subsection{Limits in bicategories}\label{ss:lims_in_bicats}

Let \C be a $1$-category.  Informally, the limit of a diagram in \C is an object
$L$ in \C with the property that, for each object $C$ in \C, to give a morphism
from $C$ to $L$ is to give a family of morphisms from $C$, one to each object in
the diagram, compatible with the transition morphisms in the diagram.

More precisely, let us say that the diagram in \C is given by an ``index''
category \I and a functor $F\colon \I \to \C$.  Then a limit of $F$ is a pair
$(L,\alpha)$, where $L\in\ob\C$ and $\alpha$ is a natural transformation $\Delta
L \to F$ \eqref{rk:diag}, with the property that, for each object $C$ in \C, the
composition
\[
   \Hom_\C(C,L) 
      \xra\Delta \Hom_{\Hom(\I,\C)}(\Delta C, \Delta L) 
      \xra{\alpha_*} \Hom_{\Hom(\I,\C)}(\Delta C, F)
\]
is a bijection of sets.  As is standard, we write $\ilim F$
for $L$ and typically omit $\alpha$ from the notation.  By Yoneda's lemma, the
pair $(\ilim F,\alpha)$ is well-defined up to canonical isomorphism.

The generalization to bicategories is straightforward.  Let $F\colon \I \to \C$
be a morphism of bicategories.  Again, recall the diagonal $\Delta$ of
\eqref{rk:diag}.

\begin{defn}\label{def:bicat_lim}
A \emph{limit} of $F$ is a pair $(L,\alpha)$, where $L\in\ob\C$ and $\alpha$ is
a transformation $\Delta L \to F$, with the property that, for each object $C$
in \C, the composition of functors
\[
  \Hom_\C(C,L) 
     \xra\Delta \Hom_{\Hom(\I,\C)}(\Delta C, \Delta L) 
     \xra{\alpha_*} \Hom_{\Hom(\I,\C)}(\Delta C, F)
\]
is an equivalence of $1$-categories.
\end{defn}

As for limits in $1$-categories, we write $\ilim F$ for $L$ and typically omit
$\alpha$ from the notation.  It is sometimes convenient to denote the limit more informally by $\ilim_{i\in\I} Fi$ or by $\ilim_i Fi$.  By Yoneda's lemma, the pair $(\ilim F,\alpha)$ is
well-defined up to equivalence.

\begin{rk}
Of course, if $(\ilim F,\alpha)$ is a limit of $F$, then we obtain a ``projection'' $1$-cell $\pr_i := \alpha_i\colon \ilim F \to Fi$ for each $i\in\ob\I$.  As $i$ varies, the $\pr_i$'s are compatible with the $F\mu$'s, $\mu\in\mor\I$, in the precisely the sense that $\alpha$ is a transformation $\Delta \ilim F \to F$.
\end{rk}

\begin{rk}\label{rk:onelim}
What we have called ``limit'' is usually called ``bilimit'' in the categorical
literature. Consider the case that \I and \C are $2$-categories and $F$ is a
$2$-functor. Then, forgetting structure, $F$ provides a functor between the underlying $1$-categories.
But the limit of $F$ in the $1$-categorical sense and the limit of $F$ in the
sense of \eqref{def:bicat_lim} need not agree, even when \I is a $1$-category;
see e.g.\ \eqref{eg:fib_prod}. Traditionally, ``bilimit'' has been used for
limits in the sense of \eqref{def:bicat_lim}, whereas ``limit'' has been
reserved for the $1$-categorical notion. But in the setting of bicategories,
\eqref{def:bicat_lim} is the more fundamental notion. So we feel it deserves the
plainer terminology, and it is that we shall call ``limit''. When we need to
distinguish the $1$-categorical notion, we shall refer to it as ``$1$-limit''
and denote it $\nlim F$.

Note that when \I is a $2$-category and \C is a $1$-category, the notions of
limit of $F$ and of $1$-limit of $F$ are the same.  In particular,
\eqref{def:bicat_lim} recovers the usual notion of limit in the case of a
functor between $1$-categories.
\end{rk}

\begin{rk}\label{rk:lim_as_rt_adjoint}
Of course, the limit of a given $F\colon \I \to \C$ may or may not exist.  At
the other extreme, suppose that \C admits all \I-indexed limits, that is, that
every morphism $\I \to \C$ admits a limit.  Then, choosing a limit $(\ilim F,
\alpha)$ for each $F\colon \I \to \C$, and using the universal property of a
limit, it is straightforward (though tedious) to verify that passing to the
limit admits a structure of bicategory morphism $\Hom(\I,\C) \to \C$, right
adjoint to $\Delta$ (in the sense of bicategories).
\end{rk}

\begin{rk}\label{rk:lim_of_ff's}
Let $G\colon \I \to \C$ be another morphism and $\alpha\colon F \to G$ a transformation. It follows immediately from the definitions that if $\alpha$ is fully faithful as a $1$-cell in $\Hom(\I,\C)$ and $\ilim F$ and $\ilim G$ exist, then the natural induced $1$-cell $\ilim F \to \ilim G$ (well-defined up to isomorphism) in \C is fully faithful.  Moreover, there is the following easy criterion for $\alpha$ to be fully faithful, often satisfied in practice:
\[
   \text{$\alpha_i$ is fully faithful in \C for all $i\in\ob\I$} \implies \text{$\alpha$ is fully faithful.}
\]
\end{rk}

\begin{rk}\label{rk:gpdl_valued}
It follows immediately from the definitions that if $F$ is a groupoidal object in $\Hom(\I,\C)$ and $\ilim F$ exists, then $\ilim F$ is groupoidal in \C.  Moreover, there is the following easy criterion for $F$ to be groupoidal, often satisfied in practice:
\[
   \text{$Fi$ is groupoidal in \C for all $i\in\ob\I$} \implies \text{$F$ is groupoidal.}
\] 
\end{rk}

\begin{rk}\label{rk:modifications_via_limits}
For later use, consider the situation of an object $C$ in \C and transformations $\alpha$, $\beta\colon \Delta C \to F$.  For simplicity, assume that \I is a $1$-category; this will be the case of interest to us in applications.  Then it is easy to verify that the assignment $i \mapsto \Hom_\C(\alpha_i,\beta_i)$ defines a functor $\I \to \Sets$ in a natural way.  Moreover, it is then immediate from the definition of modification that the limit $\ilim_{i\in\I}\Hom_\C(\alpha_i,\beta_i)$ of this functor is canonically identified with the set of modifications $\alpha \to \beta$.  When \I is an arbitrary bicategory, one still obtains a morphism $\I \to \Sets$ whose limit is the set of modifications $\alpha \to \beta$, but the needed verifications require more work.
\end{rk}

\begin{rk}
More generally, Simpson \cite{simp97} has defined limits in $n$-categories and 
begun a study of them.
\end{rk}

\begin{rk}
More generally, there is a notion of \emph{weighted limit} in a bicategory.  
See \eqref{rk:wtd_lims} for brief remarks and a reference.
\end{rk}

\subsection{Limits of categories}\label{ss:cat_lims}

In this section we consider the important special case of limits in \Cat.

As motivation, let us begin with a brief review of limits in \Sets.  Let \I be a small $1$-category and $F\colon \I \to \Sets$ a functor.  Of course,
it is well-known that the limit of $F$ exists.  But even if we didn't know this,
we could discover the limit in the following simple way.  Let $*$ denote a singleton set, so that
\[ 
   S \ciso \Hom_\Sets(*,S)
\]
for any set $S$.  If $\ilim F$ exists, then the displayed formula and the universal property of the limit would force
\[ 
   \ilim F \ciso \Hom_\Sets(*,\ilim F) \ciso \Hom_{\Hom(\I,\Sets)}(\Delta *, F).
\]
On the other hand, we may take the right-hand side of this last display as a \emph{definition} of the limit:  one verifies directly that there is a canonical functorial bijection of sets
\[
   \Hom_{\Hom(\I,\Sets)}(\Delta S, F)
      \ciso \Hom_\Sets\bigl(S, \Hom_{\Hom(\I,\Sets)}(\Delta *, F)\bigr)
\]
for any set $S$; this is just the version for covariant functors of the familiar adjunction between forming the constant presheaf and taking global sections.  Hence $\Hom_{\Hom(\I,\Sets)}(\Delta *, F)$ has the universal
property of the limit.

Everything generalizes to limits in \Cat in a straightforward way.  Let $F\colon \I \to \Cat$ be a morphism of bicategories.  We now
regard $*$ as a category with a single object and single (identity) morphism. 
Then for any category $C$, there is a canonical isomorphism (not just
equivalence) of categories
\[
   C \ciso \Hom_\Cat(*,C).
\]
Hence, just as we saw for limits in \Sets, provided $\ilim F$ exists, we are forced to compute it as
\[\label{disp:ilim_fmla}\tag{$\dag$}
   \ilim F \ciso \Hom_\Cat(*,\ilim F) \approx \Hom_{\Hom(\I,\Cat)}(\Delta *, F).
\]
On the other hand, one again finds that the right-hand side of \eqref{disp:ilim_fmla} serves to \emph{define} the limit: there is a canonical functorial isomorphism (not just equivalence) of categories
\[
\Hom_{\Hom(\I,\Cat)}(\Delta C, F)
   \ciso \Hom_\Cat \bigl(C, \Hom_{\Hom(\I,\Cat)}(\Delta *, F)\bigr)
\]
for any category $C$.  Hence $\Hom_{\Hom(\I,\Cat)}(\Delta *, F)$ has the universal
property of the limit.

\begin{rk}\label{rk:pr_i}
For each $i\in\ob\I$, the projection functor
\[
    \ilim F = \Hom_{\Hom(\I,\Cat)}(\Delta *, F) \xra{\pr_i} Fi
\]
has a simple description.  Indeed, let $X\colon \Delta * \to F$ be a
transformation.  Then $\pr_i X$ is just the object of $Fi$ identified with 
$X_i\colon * \to Fi$ via the canonical isomorphism $Fi\ciso \Hom(*,Fi)$.
\end{rk}

\begin{eg}[Fibered product of categories]\label{eg:fib_prod}
Let
\[
   \xymatrix{
                & D \ar[d]^-g \\
      C \ar[r]^-f  & E
   }
\]
be a diagram of categories and functors, say given by an (honest) functor
\[
   \vcenter{
   \xymatrix{
             & \bullet \ar[d] \\
      \bullet \ar[r] & \bullet
   }
   }
   \qquad
   \xymatrix{
      \ar[r]^F &
   }
   \qquad\Cat,
\]
where we regard the source of $F$ as a $1$-category \I.  Then
$\Hom_{\Hom(\I,\Cat)}(\Delta *, F)$ is isomorphic to the category whose
\begin{itemize}
\item objects are tuples $(X,Y,Z,\varphi,\psi)$, where
   \[
      X\in\ob C, \quad Y \in \ob D, \quad Z\in \ob E,
   \]
   and
   \[
      \varphi\colon fX \isoarrow Z \quad\text{and}\quad
      \psi\colon gY \isoarrow Z
   \]
   are isomorphisms in $E$; and
\item morphisms $(X,Y,Z,\varphi,\psi) \to (X',Y',Z',\varphi',\psi')$ are
triples $(\alpha,\beta,\gamma)$, where
\[
   \alpha\colon X\to X', \quad \beta\colon Y \to Y', \quad\text{and}
      \quad \gamma\colon Z\to Z'
\]
are morphisms in $C$, $D$, and $E$, respectively, making the evident diagrams
commute.
\end{itemize}
We take the fibered product $C \fib E D$ to be the category so defined.

It is common to find $C \fib E D$ defined instead as the category of triples
$(X,Y,\varphi)$, where $X\in\ob C$, $Y\in\ob D$, and $\varphi\colon fX \isoarrow
gY$ is an isomorphism in $E$.  One readily checks that this category is
equivalent, though not in general isomorphic, to the fibered product as we've
defined it.  In practice, we will only ever be interested in categories up to
equivalence, and we will freely use either definition.

On the other hand, since $F$ is an honest functor, we may speak of its
$1$-limit \eqref{rk:onelim}, that is, the $1$-categorical fibered product: this
is the category of pairs $X\in\ob C$ and $Y\in\ob D$ for which $fX = gY$.  It is
easy to write down simple choices of $C$, $D$, and $E$ for which $\ilim F$ and $\nlim F$ are not equivalent.
\end{eg}

\begin{eg}[Limit indexed on a $1$-category]\label{eg:lims_on_1-cats}
More generally, let \I be a $1$-category and $F\colon \I \to \Cat$ a morphism of
bicategories.  Then, by \eqref{disp:ilim_fmla}, we may take as $\ilim F$ the
category of families $(X_i,\varphi_\mu)_{i\in\ob\I,\mu\in\mor\I}$, where
\begin{itemize}
\item
   for each object $i$ in \I, $X_i$ is an object in $Fi$, and
\item
   for each morphism $\mu\colon i \to j$ in \I, $\varphi_\mu$ is an isomorphism 
   $(F\mu)X_i \isoarrow X_j$ in $Fj$,
\end{itemize}
subject to the cocycle condition 
\begin{itemize}
\item
    for every composition $i \xra\mu j \xra \nu k$ in \I, the diagram
    \[
       \xymatrix@C=8ex{
          \bigl(\F(\nu\mu)\bigr) X_i \ar[d]^-\sim_-\can \ar[r]^-{\varphi_{\nu\mu}}_-\sim
             & X_k \\
          (\F\nu)(\F\mu) X_i \ar[r]^-{(\F\nu)\varphi_\mu}_-\sim
             & (\F\nu) X_j \ar[u]^-\sim_-{\varphi_\nu}
       }
    \]
    commutes in $Fk$.
\end{itemize}
A morphism $(X_i,\varphi_\mu) \to (X_i',\varphi_\mu')$ is a family
$(\alpha_i)_{i\in\ob\I}$, with $\alpha_i\colon X_i \to X_i'$ a morphism in $Fi$
for each $i$, compatible with the $F\mu$'s and $\varphi_\mu$'s.

The reader may have noticed that, for fibered products in \eqref{eg:fib_prod}, the
families we took as objects in $\ilim F$ did not include any morphisms $\varphi$
indexed by an identity morphism in \I.  It is an easy exercise to verify that,
under the present description of $\ilim F$ in the case of fibered products, any
$\varphi$ indexed by an identity morphism must itself be an identity morphism. 
So the present description of $\ilim F$ agrees with that in \eqref{eg:fib_prod}
up to isomorphism (not just equivalence). 
\end{eg}

\begin{rk}\label{rk:lims_via_cat_lims}
It is a familiar fact in $1$-category theory that limits in arbitrary
categories can be characterized in terms of limits in \Sets.  Indeed, let
$F\colon \I \to \C$ be a functor between $1$-categories.  For each $C \in \ob
\C$, let $h^C := \Hom_\C(C,-)$ be the functor $\C \to \Sets$ corepresented by
$C$. Then the universal property of the limit of $F$ can be expressed as an
isomorphism
\[\label{disp:ilim_sets}\tag{$\ddag$}
   \Hom_\C(C,\ilim F) \iso \ilim\, (h^C \circ F)
\]
functorial in $C\in\ob\C$; or, more informally,
\[
   \Hom_\C\biggl(C,\ilim_{i\in\I}Fi\biggr) \iso \ilim_{i\in\I}\Hom_\C(C,Fi).
\]
The equivalence of the universal property formulated in \eqref{disp:ilim_sets}
with the universal property formulated in \eqref{def:bicat_lim} is
immediate, since there is a tautological identification
\[
   \ilim (h^C \circ F) \ciso \Hom_{\Hom(\I,\C)}(\Delta C, F)
\]
amounting to nothing more than the definition of natural transformation of
functors.

Quite analogously, limits in bicategories can be characterized in terms of
limits in \Cat: if $F\colon \I \to \C$ is a morphism of bicategories, then the
universal property of the limit can be expressed as an equivalence of categories
\[
   \Hom_\C(C,\ilim F) \approx \ilim\,(h^C\circ F)
\]
functorial (in the sense of bicategories) in $C \in \ob \C$; here $h^C = \Hom_\C(C,-)$ is the bicategory morphism
$\C \to \Cat$ corepresented by $C$ \eqref{eg:repble_morph}. Indeed, this time
the definitions furnish a tautological isomorphism (not just equivalence) of
categories
\[
   \ilim\, (h^C \circ F)
      = \Hom_{\Hom(\I,\Cat)}(\Delta *, h^C\circ F)
      \ciso \Hom_{\Hom(\I,\C)}(\Delta C, F),
\]
so that we recover the universal property of \eqref{def:bicat_lim}.
\end{rk}

\begin{rk}[Weighted limits]\label{rk:wtd_lims}
One may take \eqref{rk:lims_via_cat_lims} as point of departure for the notion 
of \emph{weighted limit} in a bicategory.  As we saw, the limit of $F$ is a representing object in \C for the morphism 
\[
   \vcenter{
   \xymatrix@R=0ex{
      \smash{\C^\opp}\vphantom{()} \ar[r] & \Cat\\
      \smash{C}\vphantom{()} \ar@{|->}[r] & \smash{\Hom_{\Hom(\I,\Cat)}(\Delta *, h^C\circ F).}\vphantom{()}
   }
   }
\]
More generally, one may replace $\Delta *$ in the display by an arbitrary morphism $W\colon \I \to \Cat$; a \emph{$W$-weighted limit of $F$} is a representing object in \C for the morphism $\C^\opp \to \Cat$ so obtained.  See \cite{lack07}*{\s6} for an informal introduction, and references therein for more details.
\end{rk}

\begin{rk}[Morphisms in limits of categories]\label{rk:morphs_in_lim_cats}
Let $F\colon\I \to \Cat$ be a bicategory morphism, and let $X$ and $Y$ be
objects in $\ilim F$. For simplicity, assume
that \I is a $1$-category; this will be the case of interest to us in
applications. Then by \eqref{rk:modifications_via_limits} and \eqref{rk:pr_i}, we obtain a canonical identification, which we express in informal notation,
\[
   \Hom_{\underset{i \in \I}{\ilim}Fi}(X,Y) \ciso \ilim_{i\in\I}\Hom_{Fi}(\pr_i X, \pr_i Y).
\]

We have assumed that \I is a $1$-category only so that we may apply \eqref{rk:modifications_via_limits}; one obtains the displayed formula for an arbitrary bicategory \I as soon as one allows \eqref{rk:modifications_via_limits} for arbitrary \I.
\end{rk}

\begin{rk}\label{rk:sga_ref}
Limits of categories indexed by $1$-categories were introduced in
\cite{sga1}*{VI 5.5}, and studied further in \cite{gir64}, via fibered
categories; see also \cite{sga4-2}*{VI 6.10--.11}. Recall that if \I is a
$1$-category, then there is a natural $2$-functor 
\[
   \Hom(\I,\Cat) \to \Fib(\I^\opp)
\]
\cite{sga1}*{VI \s9} which is an equivalence of $2$-categories.  Let $F\colon \I
\to \Cat$ be a bicategory morphism, and denote the corresponding fibered
category over $\I^\opp$ by \F.  In particular, $\Delta*\colon \I \to \Cat$
corresponds (up to isomorphism) to $\I^\opp$ itself, and we have
\[
   \ilim F = \Hom_{\Hom(\I,\Cat)}(\Delta*,F) 
           \approx \Hom_{\Fib(\I^\opp)} (\I^\opp,\F).
\]
The right-hand side is the notion of limit defined in \cite{sga1}.

Limits of toposes indexed by $1$-categories were first studied in
\cite{sga4-2}*{VI \s8}.  In particular, (8.1.3.3) of loc.\ cit.\ discusses the
general universal mapping property of the limit in the particular context of
toposes.
\end{rk}

\subsection{Limits in morphism bicategories}\label{ss:lims_in_morph_bicats}
It is a familiar fact in $1$-category theory that if a category \D contains all
limits, then so does $\Hom(\C,\D)$ for any category \C.  More precisely, let \I
be a category and $F\colon \I \to \Hom(\C,\D)$ a functor.  Then for each
$C\in\ob \C$, we obtain a diagram
\[
   F_C\colon
   \xymatrix@R=0ex{
      \smash{\I}\vphantom{()} \ar[r] & \smash{\D}\vphantom{()}\\
      \smash{i}\vphantom{()} \ar@{|->}[r] & (Fi)(C).
   }
\]
If $\ilim F_C$ exists for all $C$, then the assignment $C\mapsto \ilim F_C$
defines a functor $L\colon \C \to \D$ in a natural way, and $L$ is a naturally
a limit of $F$.  In more informal notation,
\[
   \biggl(\ilim_{i\in\I} Fi\biggr)(C) = \ilim_{i\in\I} (Fi)(C);
\]
or in more informal words, limits are computed objectwise.  In particular,
suppose \D contains all \I-indexed limits; that is, suppose $\ilim$ defines a
functor $\Hom(\I,\D) \to \D$, right adjoint to $\Delta_{\I,\D}$ \eqref{rk:diag}.
Then $\Hom(\C,\D)$ has all \I-indexed limits, and they are computed via the
composition
\[\label{disp:J_lim_fmla}\tag{\dag}
   \Hom\bigl(\I,\Hom(\C,\D)\bigr)
      \ciso \Hom\bigl(\C,\Hom(\I,\D)\bigr)
      \xra{(\ilim)_*} \Hom(\C,\D).
\]

As in previous sections, everything now generalizes to the
situation of bicategories.  Let \C, \D, and \I be bicategories and $F\colon \I 
\to \Hom(\C,\D)$ a morphism.  Then for each $C\in\ob\C$, we again get
a morphism $F_C\colon \I \to \D$ sending $i\mapsto (Fi)(C)$.  Suppose each $F_C$
admits a limit, and fix a particular choice $\ilim F_C$ for each $C$.  Of
course, this time the assignment $C\mapsto \ilim F_C$ does not canonically
define a morphism $\C \to \D$, since we only have equivalences
\[
   \Hom_\D\bigl(\ilim F_C, \ilim F_{C'}\bigr)
      \approx \ilim_i \Hom_\D\bigl(\ilim F_C, (Fi)(C')\bigr)
\]
for varying $C$, $C'\in\ob\C$.  Nevertheless, the right-hand side of the display
allows us to choose an assignment of $1$-cells
\[
   \bigl(C\overset{f}{\rightarrow} C'\bigr)
      \mapsto \biggl(\ilim F_C \xra{\ilim f} \ilim F_{C'}\biggr)
\]
as $f$ runs through the $1$-cells in \C.  It is straightforward to verify that
our choices determine a morphism $L\colon\C \to \D$ in a natural way, and
moreover $L$ is naturally a limit of $F$.  As in the $1$-categorical case, if 
\D has all \I-indexed limits, so that $\ilim$ admits a structure of morphism
$\Hom(\I,\D) \to \D$, then $\Hom(\C,\D)$ has all $\I$-indexed limits; these last
may be computed exactly as in \eqref{disp:J_lim_fmla}, interpreted in the
setting of  bicategories.

\subsection{Limits of fibered categories and of stacks}\label{ss:lim_stacks}
In this section we shall apply some of the general considerations in previous
sections to the particular setting of fibered categories and stacks.  Let \C be
a $1$-category.  Strictly speaking, we shall apply our previous work directly to $\Hom(\C^\opp,\Cat)$, and then to $\Fib(\C)$ via the equivalence of $2$-categories
\[
   \Hom(\C^\opp,\Cat) \to \Fib(\C)
\]
we recalled earlier in \eqref{rk:sga_ref}.

Our first result is an immediate consequence of \s\s\ref{ss:cat_lims} and 
\ref{ss:lims_in_morph_bicats}.

\begin{thm}\label{st:Fib_has_lims}
$\Fib(\C)$ contains all limits.\hfill $\square$
\end{thm}

\begin{rk}\label{rk:Fib_obwise}
Quite explicitly, let $F\colon \I \to \Fib(\C)$ be a morphism of bicategories,
and let $(Fi)(C)$ denote the category fiber of $Fi$ over each $C\in\ob\C$.  Then
$\ilim F$ is the fibered category whose fiber over each $C$ is the category
$\ilim_{i\in\I} (Fi)(C)$.  In other words, limits of fibered categories are 
computed fiberwise.
\end{rk}

Of course, \eqref{st:Fib_has_lims} is a generalization of the familiar fact
that for any category \C, the category of presheaves of sets on \C contains all (small)
limits.

In the case of a diagram of CFG's in $\Fib(\C)$, \eqref{rk:gpdl_valued} asserts that the limit will again be a CFG.  So we deduce the following.

\begin{cor}\label{st:CFG_has_lims}
$\CFG(\C)$ contains all limits; they are computed exactly as limits in $\Fib(\C)$. \hfill $\square$
\end{cor}

Let us now turn to stacks.  We continue with our $1$-category \C, and suppose it 
equipped with a Grothendieck topology; we shall not consider here Grothendieck 
topologies on arbitrary bicategories.  It is a familiar fact in sheaf theory 
that a limit of sheaves (of sets, say), computed in the category of presheaves, 
is always again a sheaf.  We shall now arrive at the analogous result for 
stacks.  

Let $F\colon \I \to \Fib(\C)$ be a morphism of bicategories; then $\ilim F$ exists by \eqref{st:Fib_has_lims}.

\begin{thm}\label{st:lim_stacks}
Suppose $Fi$ is a stack for each $i\in\ob \I$.  Then $\ilim F$ is a stack.  In
particular, $\St(\C)$ contains all limits.
\end{thm}

\begin{proof}
Let $C\in\ob\C$ and $U \to \ul{C}$ a covering sieve on $C$.  We must show that 
the natural functor
\[\label{disp:st_cond}\tag{\dag}
   \Hom_{\Fib(\C)}(\ul C,\ilim F) \to \Hom_{\Fib(\C)}(U,\ilim F)
\]
is an equivalence of categories.  By \eqref{rk:lims_via_cat_lims}, we may
replace \eqref{disp:st_cond} with
\[\label{disp:rewrite}\tag{\ddag}
   \ilim_{i\in\I}\Hom_{\Fib(\C)}(\ul C, Fi) \to \ilim_{i\in\I} \Hom_{\Fib(\C)}(U, Fi).
\]
Since $Fi$ is a stack for each $i\in\ob\I$, each $\Hom(\ul C,Fi) \to \Hom 
(U,Fi)$ is an equivalence.  Hence $\Hom_{\Fib(\C)}(\ul C, F-)$ and 
$\Hom_{\Fib(C)}(U, F-)$ are equivalent objects in $\Hom(\I,\Cat)$ 
\eqref{rk:equivs_in_morph_bicat}.  Hence \eqref{disp:rewrite} is an equivalence.
\end{proof}

\subsection{Limits of objects in limits of categories; group objects}\label{ss:lims_in_lim_cats}

Let $F\colon \I \to \Cat$ be a morphism of bicategories.  By 
\s\ref{ss:cat_lims}, there exists a limit category $\ilim F$.  Let now \J be a 
$1$-category and $G\colon \J \to \ilim F$ a functor.  In this section, we wish 
to describe the limit of $G$ in $\ilim F$ in terms of limits in the categories 
$Fi$, $i\in\ob\I$, under certain natural assumptions on $F$ and $G$.

We shall make use below of \eqref{rk:morphs_in_lim_cats}.  So, as in \eqref{rk:morphs_in_lim_cats}, we shall make the simplifying assumption that \I is a $1$-category.  Everything we shall do would work for an arbitrary bicategory \I, but we only need the $1$-categorical case in applications. We'll leave the verifications needed for the general case to the reader.

Our assumptions on $F$ and $G$ are as follows.  For each $i\in\ob\I$, let $G_i$ denote the composite functor
\[
   \J \xra G \ilim F \xra{\pr_i} Fi.
\]
We shall assume
\begin{itemize}
\item
   for all objects $i$ in $\I$, $G_i$ admits a limit in $Fi$; and
\item
   for all morphisms $\mu\colon i \to i'$ in \I, the functor $F\mu\colon Fi \to Fi'$ sends limits for $Gi$ to limits for $Gi'$.
\end{itemize}

For each $i\in\ob\I$, choose a limit $\ilim G_i$ of $G_i$ in $Fi$.  By 
assumption, for each $\mu\colon i \to i'$ in \I, the objects $(F\mu)(\ilim G_i)$
and $\ilim G_{i'}$ are canonically isomorphic in $Fi'$.  Hence the $\ilim G_i$'s
determine an object $L$ in $\ilim F$ \eqref{eg:lims_on_1-cats}.

\begin{prop}\label{st:lims_in_lim_cats}
The object $L$ is naturally a limit of $G$ in $\ilim F$.
\end{prop}

\begin{proof}
Let $X\in\ob\ilim F$.  Then
\begin{alignat*}{2}
   \smash{\Hom_{\Hom(\J,\ilim F)}(\Delta X, G)}
      &\ciso \Hom_{\underset{i\in\I}{\ilim}\Hom(\J, Fi)}(\Delta X, G) 
               & &\quad \eqref{rk:lims_via_cat_lims}\\
      &\ciso \ilim_{i\in\I}\Hom_{\Hom(\J,Fi)}(\Delta\,\pr_i X,G_i)
               & &\quad\eqref{rk:morphs_in_lim_cats}\\
      &\ciso \ilim_{i\in\I}\Hom_{Fi}(\pr_i X,\ilim G_i)  \\
      &\ciso \Hom_{\ilim F}(X,L) 
               & &\quad\eqref{rk:morphs_in_lim_cats}.
\end{alignat*}
Hence $L$ has the universal property of the limit.
\end{proof}

As an application, let us now consider group objects in $\ilim F$. Quite
generally, for any category \C, write $\Gp(\C)$ for the category of group
objects and homomorphisms in \C. Let us note that there is some ambiguity as to
what we really mean by this category: namely, the products of objects in \C are
only defined up to canonical isomorphism, so some care is needed in defining
the objects of $\Gp(\C)$ precisely. Nevertheless, it is clear that all reasonable notions of $\Gp(\C)$ produce equivalent categories.  So we shall regard the issue as one of pedantry and ignore it from now on.  Similarly, we write $\Ab(\C)$ for the full subcategory of $\Gp(\C)$ of commutative group objects in \C.

We continue with a $1$-category \I and a morphism of bicategories $F\colon\I \to \Cat$.  Suppose that
\begin{itemize}
\item
   for all $i\in\ob\I$, $Fi$ admits all finite (including empty) products; and
\item
   for all $\mu\in\mor\I$, $F\mu$ preserves all finite (including empty) products.
\end{itemize}
Then each $F\mu\colon Fi \to Fi'$ induces $\Gp(Fi) \to \Gp(Fi')$, and we obtain a morphism of bicategories
\[
   \xymatrix@R=0ex{
      \smash{\I}\vphantom{()} \ar[r] & \Cat\\
      \smash{i}\vphantom{()} \ar@{|->}[r] & \smash{\Gp(Fi).}\vphantom{()}
   }
\]
Moreover, by \eqref{st:lims_in_lim_cats}, the projection $\ilim_{i'} Fi' \xra{\pr_i} Fi$ induces $\Gp(\ilim_{i'} Fi') \to \Gp(Fi)$ for all $i\in\ob\I$, and we obtain an arrow
\[\tag{$\sharp$}\label{disp:gp_obs_arrow}
   \Gp\Bigl(\ilim_{i\in\I} Fi\Bigr) \to \ilim_{i\in\I} \Gp(Fi).
\]
Similarly, we obtain the category $\ilim_i \Ab(Fi)$ and an arrow
\[\tag{$\flat$}\label{disp:ab_gp_obs_arrow}
   \Ab\Bigl(\ilim_{i\in\I}Fi\Bigr) \to \ilim_{i\in\I} \Ab(Fi).
\]

\begin{prop}\label{st:gp_obs_in_lim}
The arrows \eqref{disp:gp_obs_arrow} and \eqref{disp:ab_gp_obs_arrow} are equivalences of categories.
\end{prop}

\begin{proof}
Immediate from \eqref{rk:morphs_in_lim_cats} and \eqref{st:lims_in_lim_cats}.
\end{proof}

\subsection{Pro-objects and limits}\label{ss:pro-ob_lims}

Let 
\begin{itemize}
\item
   \C be a $1$-category;
\item
   $\pro\C$ denote the category of pro-objects in \C;
\item
   \D be a bicategory; and 
\item
   $F\colon \C \to \D$ be a morphism of bicategories.
\end{itemize}
Our goal for the section is to show that if \D contains all filtered limits, 
then $F$ induces a morphism $\ula F \colon \pro\C \to \D$ in a ``natural'' --- or more precisely, in a canonical-up-to-equivalence --- way.

For convenience, let us begin by recalling some of the basic definitions involved.  A $1$-category \I is \emph{filtered} if
\begin{enumerate}
\item \I is nonempty;
\item for every $i_1$, $i_2 \in\ob \I$, there exist $i \in\ob \I$ and
	morphisms
	\[
		\xymatrix@R=0ex@C+2ex{
			& i_1\\
			i \ar@{-->}[ru] \ar@{-->}[rd]\\
			& i_2;
			}
	\]
	and
\item for every diagram
	\[
		\xymatrix{
			i_1 \ar@<.62ex>[r] \ar@<-.3ex>[r] & i_2
			}
	\]
	in \I, there exist $i\in \ob \I$ and an equalizing morphism
	\[
		\xymatrix{
			i \ar@{-->}[r] & i_1 \ar@<.62ex>[r] \ar@<-.3ex>[r] & i_2,
			}
	\]
	that is, the two compositions in this last diagram are equal.
\end{enumerate}

For applications, our primary interest is in limits indexed on the filtered set
\NN (e.g.\ $\ilim_n \nInf$ in \s\ref{ss:inf_inf_sheaves}, $\ilim_n \Gn$ in \s\ref{ss:flv_stack}, and $\ilim_n\Bn$ in \s\ref{ss:flg_stack}), where we regard
\NN as the category
\[
	\cdots \to 3 \to 2 \to 1,
\]
with compositions and identities suppressed.

\begin{rk}
Many authors use ``filtered'' in the sense dual to ours; see
e.g.\ \cite{sga4-1}*{I 2.7} or \cite{cwm}*{IX \s1}.  In our use of the
terminology, one typically encounters filtered limits and cofiltered colimits.
\end{rk}

A \emph{pro-object in \C} \cite{sga4-1}*{I \s8.10} is a functor $X \colon \I \to
\C$ defined on some
small filtered category \I.  We typically write $X_i$
for $Xi$, $i\in\ob\I$, and
$\prolim_{\!\!\!\!i\in \I} X_i$ for $X$.  The pro-objects in \C form a category
$\pro\C$, with morphisms defined by
\[
   \Hom_{\pro\C}\biggl( \prolim_{i\in \I} X_i, \prolim_{j\in \J} Y_j \biggr)
      := \ilim_{j\in \J} \clim_{i\in \I} \Hom_\C (X_i,Y_j).
\]
The composition law in \C induces a natural one in $\pro\C$, well-defined, as one readily checks,
because the index categories are filtered.

For simplicity, we shall consider neither the notion of filtered nor that of
pro-object in the setting of arbitrary bicategories.

Let us now explain the construction of the morphism $\ula F \colon \pro\C \to 
\D$.  It is useful to first consider the case \D is a $1$-category.  On objects, $\ula F$ sends the pro-object $X\colon \I \to \C$ to $\ilim (F\circ X)$; or, more informally,
\[
   \prolim_{i\in \I} X_i \overset{\ula{F}}{\mapsto} \ilim_{i\in\I} FX_i.
\]
$\ula F$ is defined on morphisms in the evident way, using the universal property of the limit.  It is then easy to verify, using that \I is filtered, that $\ula F$ is an honest functor.  Plainly, $\ula F$ is well-defined up to canonical isomorphism.

In the case of a general bicategory \D, essentially the same approach works, but a bit more care must be taken in the definition of $\ula F$ on morphisms.  For each pro-object $X = \prolim_{\!\!\!\! i\in\I} X_i$ in \C, we define $\ula F X$ to be any choice of limit $\ilim_{i\in\I} FX_i$ in \D.  For each morphism of pro-objects
\[
   f = (f_j)_{j\in\ob\J} \in \Hom_{\pro\C}\biggl( \prolim_{i\in \I} X_i, \prolim_{j\in \J} Y_j \biggr)
      = \ilim_{j\in \J} \clim_{i\in \I} \Hom_\C (X_i,Y_j),
\]
for each $j\in\ob J$, choose a representative $\wt{f_j}\colon X_{i_j} \to Y_j$ of $f_j$.  Then we get a $1$-cell $\ula{f_j}\colon\ilim_i FX_i \to FY_j$ in \D by composing
\[
	\xymatrix{
		\ilim_i FX_i \ar[r]^-{\pr_{i_j}} \ar@/_4ex/[rr]_-{\ula{f_j}}
			& FX_{i_j} \ar[r]^-{F\wt{f_j}}
			& FY_j
		}.
\]
For each $\nu\colon j_1 \to j_2$ in \J, denote by $\nu_*$ the induced morphism
$Y_{j_1} \to Y_{j_2}$.  One then verifies, using that \I is filtered, that the diagram
\[
	\xymatrix@C=0cm@R=5ex{
		& \ilim_i FX_i \ar[dl]_-{\ula{f_{j_1}}} \ar[dr]^-{\ula{f_{j_2}}}\\
	   FY_{j_1} \ar[rr] _-{F\nu_*}
		   & & FY_{j_2}
		}
\]
commutes up to canonical isomorphism.  Hence, by the universal property of the 
limit (\ref{eg:lims_on_1-cats}, \ref{rk:lims_via_cat_lims}), we obtain a $1$-cell $\ilim_i FX_i \to \ilim_j FY_j$, which we define to be $\ula F f$.

\begin{prop}\label{st:proF}
The assignments $X \mapsto \ula F X$ on objects and $f \mapsto \ula F f$ on morphisms admit a natural structure of bicategory morphism $\ula F \colon \pro\C \to \D$.
\end{prop}

The proof is straightforward, and we shall leave the details to the reader.  Of course, the notion of filtered category plays an essential role.

It is clear that $\ula F$ is well-defined up to equivalence.

\begin{rk}
There is a canonical functor $\chi\colon \C \to \pro\C$ sending each object $C$ to the 
pro-object indexed on $*$ with value $C$.  Moreover, one verifies easily that 
$\pro\C$ contains all small filtered limits.  One may now distinguish $\pro\C$
and $\chi$ amongst the maps from \C into $1$-categories by the following 
universal mapping property:  up to isomorphism of functors, to give a functor 
$\C \to \D$ into a $1$-category \D containing small filtered limits is to give 
a functor $\pro\C \to \D$ preserving small filtered limits.  Analogously,
$\pro\C$ and $\chi$ enjoy a universal position amongst the bicategories:  up to 
equivalence of morphisms of bicategories, to give a  morphism $\C \to \D$ into 
a bicategory \D containing small filtered limits is to give a morphism $\pro\C 
\to \D$ preserving small filtered limits.
\end{rk}

\subsection{Initiality and limits}\label{ss:fin_lims}

It is a standard fact in $1$-category theory that limits are invariant under
right composition with an initial functor.  Precisely, recall \cite{cwm}*{IX
\s3} that a functor $L \colon \J \to \I$ is \emph{initial} if for all
$i\in\ob\I$, the comma category $\J_{/i}$ (denoted $L\downarrow i$ in the
notation of \cite{cwm}) is nonempty and connected.  If $L\colon \J \to \I$ is
initial, $F\colon \I \to \C$ is any  functor, and $\ilim FL$ exists, then one
proves easily that $\ilim FL$ is canonically a limit of $F$ \cite{cwm}*{IX 3.1}.  

In this section we shall obtain a suitable generalization to the following
situation:
\begin{itemize}
\item
   \C is an arbitrary bicategory;
\item
   for simplicity, \I and \J are $1$-categories, and $L\colon \J \to \I$ is an
   initial functor between them; and
\item
   $F$ is an arbitrary bicategory morphism $\I \to \C$.
\end{itemize}
Let $C\in\ob\C$.

\begin{prop}
The natural functor
\[
   \Hom_{\Hom(\I,\C)}(\Delta_{\I,\C}C,F)
      \to \Hom_{\Hom(\J,\C)}(\Delta_{\J,\C}C,FL)
\]
induced by $L$ is an equivalence of categories.
\end{prop}

\begin{proof}[Sketch proof]
The proof is tedious but straightforward, so we'll just give a sketch.  Let
$\Phi$ denote the functor in the statement of the proposition.  

One sees that
$\Phi$ is faithful directly from the definition of $2$-cells in morphism
bicategories and from the assumption that every comma category $\J_{/i}$
is nonempty.

To see that $\Phi$ is full, suppose there are $1$-cells
\[
   \xy
      (-10,0)*+{\Delta C}="C";
      (10,0)*+{F}="FL";
      {\ar@/^3ex/ "C";"FL" ^-{\alpha}};
      {\ar@/_3ex/ "C";"FL" _-{\beta}};
   \endxy
\]
in $\Hom(\I,\C)$ and a $2$-cell
\[
   \xy
      (-10,0)*+{\Delta C}="C";
      (10,0)*+{FL}="FL";
      {\ar@/^3ex/ "C";"FL" ^-{\Phi\alpha}};
      {\ar@/_3ex/ "C";"FL" _-{\Phi\beta}};
      {\ar@{=>} (0,3);(0,-3) ^-{\xi}}
   \endxy
\]
in $\Hom(\J,\C)$.  We must find a $2$-cell $\wt\xi\colon\alpha \to
\beta$ mapping to $\xi$.  For each $i\in\ob\I$, choose an object $(j, Lj
\xra\mu i)$ in $\J_{/i}$. One then defines $\wt\xi_i$ in an obvious way
using $\xi_j$ and $\mu$.  One next shows that $\wt\xi_i$ is independent of the
choice of object in $\J_{/i}$, first in the special case that two choices
admit a morphism in $\J_{/i}$ between them, and then in general using the
special case and
that $\J_{/i}$ is connected.  It then follows easily from independence
of choices that the $\wt\xi_i$'s define a $2$-cell $\wt\xi$ mapping to $\xi$.

To see that $\Phi$ is essentially surjective, let $\alpha\colon \Delta C \to
FL$ be a $1$-cell in $\Hom(\J,\C)$.  As usual, for each $i\in\ob\I$, choose
$(j,Lj\xra\mu i)$ in $\J_{/i}$, and define $\wt\alpha$ to be the
composite $1$-cell
\[
   \xymatrix{
      C \ar[r]^-{\alpha_j} \ar@/_3.5ex/[rr]_-{\wt\alpha_i}
         & FLj \ar[r]^-{F\mu}
         & Fi 
   }.
\]
This time $\wt\alpha_i$ is not independent of the choice of $j$ and $\mu$.  But
using that $\J_{/i}$ is connected, one verifies that different choices
produce canonically isomorphic $\wt\alpha_i$'s.  So the $\wt\alpha_i$'s define
$\wt\alpha\colon \Delta C \to F$ in a natural way, and moreover $\Phi\wt\alpha
\ciso
\alpha$.
\end{proof}

\begin{cor}\label{st:initial_lim}
$F$ admits a limit $\iff$ $FL$ does, and the limits are equivalent when they
exist. \hfill$\square$
\end{cor}


\begin{bibdiv}
\begin{biblist}

\bib{sga4-1}{book}{
  label={SGA4$_{\text {1}}$},
  author={Artin, M.},
  author={Grothendieck, Alexander},
  author={Verdier, J. L.},
  title={Th\'eorie des topos et cohomologie \'etale des sch\'emas. Tome 1: Th\'eorie des topos},
  contribution={avec la collaboration de N.~Bourbaki, P.~Deligne et B.~Saint-Donat},
  language={French},
  series={S\'eminaire de G\'eom\'etrie Alg\'ebrique du Bois Marie 1963--1964 (SGA 4). Lecture Notes in Mathematics},
  volume={269},
  publisher={Springer-Verlag},
  place={Berlin-New York},
  date={1972},
  pages={xix+525},
  review={\MR {0354652 (50 \#7130)}},
}

\bib{sga4-2}{book}{
  label={SGA4$_{\text {2}}$},
  author={Artin, M.},
  author={Grothendieck, Alexander},
  author={Verdier, J. L.},
  title={Th\'eorie des topos et cohomologie \'etale des sch\'emas. Tome 2},
  contribution={avec la collaboration de N.~Bourbaki, P.~Deligne et B.~Saint-Donat},
  language={French},
  series={S\'eminaire de G\'eom\'etrie Alg\'ebrique du Bois Marie 1963--1964 (SGA 4). Lecture Notes in Mathematics},
  volume={270},
  publisher={Springer-Verlag},
  place={Berlin-New York},
  date={1972},
  pages={iv+418},
  review={\MR {0354653 (50 \#7131)}},
}

\bib{ben67}{article}{
  label={B\'e},
  author={B{\'e}nabou, Jean},
  title={Introduction to bicategories},
  conference={ title={Reports of the Midwest Category Seminar}, },
  book={ publisher={Springer}, place={Berlin}, },
  date={1967},
  pages={1--77},
  review={\MR {0220789 (36 \#3841)}},
}

\bib{borc94}{book}{
  label={Bo},
  author={Borceux, Francis},
  title={Handbook of categorical algebra. 1},
  series={Encyclopedia of Mathematics and its Applications},
  volume={50},
  note={Basic category theory},
  publisher={Cambridge University Press},
  place={Cambridge},
  date={1994},
  pages={xvi+345},
  isbn={0-521-44178-1},
  review={\MR {1291599 (96g:18001a)}},
}

\bib{dj98}{article}{
  author={de Jong, A. J.},
  title={Homomorphisms of Barsotti-Tate groups and crystals in positive characteristic},
  journal={Invent. Math.},
  volume={134},
  date={1998},
  number={2},
  pages={301--333},
  issn={0020-9910},
  review={\MR {1650324 (2000f:14070a)}},
}

\bib{sga3-1}{book}{
  label={SGA3$_{\text {I}}$},
  author={Demazure, M.},
  author={Grothendieck, Alexander},
  title={Sch\'emas en groupes. I: Propri\'et\'es g\'en\'erales des sch\'emas en groupes},
  language={French},
  series={S\'eminaire de G\'eom\'etrie Alg\'ebrique du Bois Marie 1962/64 (SGA 3). Lecture Notes in Mathematics},
  volume={151},
  publisher={Springer-Verlag},
  place={Berlin},
  date={1970},
  pages={xv+564},
  review={\MR {0274458 (43 \#223a)}},
}

\bib{demga70}{book}{
  author={Demazure, Michel},
  author={Gabriel, Pierre},
  author={},
  title={Groupes alg\'ebriques. Tome I: G\'eom\'etrie alg\'ebrique, g\'en\'eralit\'es, groupes commutatifs},
  language={French},
  note={Avec un appendice {\it Corps de classes local}\ par Michiel Hazewinkel},
  publisher={Masson \& Cie, \'Editeur, Paris},
  date={1970},
  pages={xxvi+700},
  review={\MR {0302656 (46 \#1800)}},
}

\bib{dieu57}{article}{
  author={Dieudonn{\'e}, Jean},
  title={Groupes de Lie et hyperalg\`ebres de Lie sur un corps de caract\'eris\-tique $p>0$. VII},
  language={French},
  journal={Math. Ann.},
  volume={134},
  date={1957},
  pages={114\ndash 133},
  issn={0025-5831},
  review={\MR {0098146 (20 \#4608)}},
}

\bib{froh68}{book}{
  author={Fr{\"o}hlich, A.},
  title={Formal groups},
  series={Lecture Notes in Mathematics},
  volume={74},
  publisher={Springer-Verlag},
  place={Berlin},
  date={1968},
  pages={iii+140},
  review={\MR {0242837 (39 \#4164)}},
}

\bib{gir64}{article}{
  label={Gi},
  author={Giraud, Jean},
  title={M\'ethode de la descente},
  language={French},
  journal={Bull. Soc. Math. France M\'em.},
  volume={2},
  date={1964},
  pages={viii+150},
  review={\MR {0190142 (32 \#7556)}},
}

\bib{goerss04}{article}{
   label={Go1},
   author={Goerss, Paul G.},
   title={(Pre-)sheaves of ring spectra over the moduli stack of formal
   group laws},
   conference={
      title={Axiomatic, enriched and motivic homotopy theory},
   },
   book={
      series={NATO Sci. Ser. II Math. Phys. Chem.},
      volume={131},
      publisher={Kluwer Acad. Publ.},
      place={Dordrecht},
   },
   date={2004},
   pages={101--131},
   review={\MR{2061853 (2005d:55007)}},
}

\bib{goerss?}{book}{
   label={Go2},
   author={Goerss, Paul G.},
   title={Quasi-coherent sheaves on the moduli stack of formal groups},
   status={in preparation},
}

\bib{ghmr05}{article}{
  author={Goerss, P.},
  author={Henn, H.-W.},
  author={},
  author={Mahowald, M.},
  author={},
  author={Rezk, C.},
  title={A resolution of the $K(2)$-local sphere at the prime 3},
  journal={Ann. of Math. (2)},
  volume={162},
  date={2005},
  number={2},
  pages={777--822},
  issn={0003-486X},
  review={\MR {2183282 (2006j:55016)}},
}


\bib{egaIV4}{article}{
  label={EGAIV$_{\text {4}}$},
  author={Grothendieck, Alexander},
  title={\'El\'ements de g\'eom\'etrie alg\'ebrique. IV. \'Etude locale des sch\'emas et des morphismes de sch\'emas, Quatri\`eme partie},
  contribution={r\'edig\'es avec la collaboration de J.~Dieudonn\'e},
  language={French},
  journal={Inst. Hautes \'Etudes Sci. Publ. Math.},
  number={32},
  date={1967},
  pages={361 pp.},
  issn={0073-8301},
  review={\MR {0238860 (39 \#220)}},
}

\bib{gro74}{book}{
  author={Grothendieck, Alexander},
  title={Groupes de Barsotti-Tate et cristaux de Dieudonn\'e},
  language={French},
  publisher={Les Presses de l'Universit\'e de Montr\'eal, Montreal, Que.},
  date={1974},
  pages={155},
  review={\MR {0417192 (54 \#5250)}},
}

\bib{sga1}{book}{
  label={SGA1},
  author={Grothendieck, Alexander},
  title={Rev\^etements \'etales et groupe fondamental (SGA 1)},
  language={French},
  contribution={augment\'e de deux expos\'es de Mme M. Raynaud},
  series={S\'eminaire de G\'eom\'etrie Alg\'ebrique du Bois Marie 1960--61. Documents Math\'ematiques (Paris) [Mathematical Documents (Paris)], 3},
  publisher={Soci\'et\'e Math\'ematique de France},
  place={Paris},
  date={2003},
  edition={\'edition recompos\'ee et annot\'ee du volume 224 des Lecture Notes in Mathematics publi\'e en 1971 par Springer-Verlag},
  pages={xviii+327},
  isbn={2-85629-141-4},
  review={\MR {2017446 (2004g:14017)}},
  eprint={arXiv:math.AG/0206203},
}

\bib{haz78}{book}{
  label={Ha},
  author={Hazewinkel, Michiel},
  title={Formal groups and applications},
  series={Pure and Applied Mathematics},
  volume={78},
  publisher={Academic Press Inc. [Harcourt Brace Jovanovich Publishers]},
  place={New York},
  date={1978},
  pages={xxii+573pp},
  isbn={0-12-335150-2},
  review={\MR {506881 (82a:14020)}},
}

\bib{hol?}{article}{
  label={Hol},
  author={Hollander, Sharon},
  title={Characterizing algebraic stacks},
  date={2007-08-20},
  status={to appear in Proc. Amer. Math. Soc.},
  note={\href{http://www.arxiv.org/abs/0708.2705v1}{\texttt{arXiv:0708.2705v1 [math.AT]}}},
}

\bib{hop99}{misc}{
  label={Hop},
  author={Hopkins, Michael J.},
  title={Complex oriented cohomology theories and the language of stacks},
  date={1999-08-13},
  status={unpublished},
  note={Course notes for 18.917, Spring 1999, MIT, currently available at
         \url {http://www.math.rochester.edu/people/faculty/doug/papers.html}},
}

\bib{hovstr05}{article}{
  author={Hovey, Mark},
  author={Strickland, Neil},
  title={Comodules and Landweber exact homology theories},
  journal={Adv. Math.},
  volume={192},
  date={2005},
  number={2},
  pages={427--456},
  issn={0001-8708},
  review={\MR {2128706 (2006e:55007)}},
}


\bib{lack07}{article}{
  label={Lack},
  author={Lack, Stephen},
  title={A $2$-categories companion},
  date={2007-02-19},
  note={\href {http://www.arxiv.org/abs/math/0702535v1}{\texttt
            {arXiv:math/0702535v1 [math.CT]}}},
}

\bib{land73}{article}{
   label={Lan1},
   author={Landweber, Peter S.},
   title={Annihilator ideals and primitive elements in complex bordism},
   journal={Illinois J. Math.},
   volume={17},
   date={1973},
   pages={273--284},
   issn={0019-2082},
   review={\MR{0322874 (48 \#1235)}},
}

\bib{land76}{article}{
   label={Lan2},
   author={Landweber, Peter S.},
   title={Homological properties of comodules over $M{\rm U}\sb\ast (M{\rm
   U})$\ and BP$\sb\ast $(BP)},
   journal={Amer. J. Math.},
   volume={98},
   date={1976},
   number={3},
   pages={591--610},
   issn={0002-9327},
   review={\MR{0423332 (54 \#11311)}},
}

\bib{lamb00}{book}{
  author={Laumon, G{\'e}rard},
  author={Moret-Bailly, Laurent},
  title={Champs alg\'ebriques},
  language={French},
  series={Ergebnisse der Mathematik und ihrer Grenzgebiete. 3. Folge. A Series of Modern Surveys in Mathematics [Results in Mathematics and Related Areas. 3rd Series. A Series of Modern Surveys in Mathematics]},
  volume={39},
  publisher={Springer-Verlag},
  place={Berlin},
  date={2000},
  pages={xii+208},
  isbn={3-540-65761-4},
  review={\MR {1771927 (2001f:14006)}},
}

\bib{laz55}{article}{
  author={Lazard, Michel},
  title={Sur les groupes de Lie formels \`a un param\`etre},
  language={French},
  journal={Bull. Soc. Math. France},
  volume={83},
  date={1955},
  pages={251\ndash 274},
  issn={0037-9484},
  review={\MR {0073925 (17,508e)}},
}

\bib{lub64}{article}{
  label={Lub},
  author={Lubin, Jonathan},
  title={One-parameter formal Lie groups over ${\germ p}$-adic integer rings},
  review={\MR {0168567 (29 \#5827)}},
  partial={
    journal={Ann. of Math. (2)},
    volume={80},
    date={1964},
    pages={464\ndash 484},
    issn={0003-486X},
  },
  partial={
    part={correction},
    journal={Ann. of Math. (2)},
    volume={84},
    date={1966},
    pages={372},
  },
}

\bib{lur06}{article}{
  label={Lur},
  author={Lurie, Jacob},
  title={A survey of elliptic cohomology},
  date={2007-04-19},
  note={Currently available at \url {http://www-math.mit.edu/~lurie/}},
}

\bib{cwm}{book}{
  label={CWM},
  author={Mac Lane, Saunders},
  title={Categories for the working mathematician},
  edition={2},
  series={Graduate Texts in Mathematics},
  volume={5},
  publisher={Springer-Verlag},
  place={New York},
  date={1998},
  pages={xii+314},
  isbn={0-387-98403-8},
  review={\MR {1712872 (2001j:18001)}},
}

\bib{mes72}{book}{
  label={Me},
  author={Messing, William},
  title={The crystals associated to Barsotti-Tate groups: with applications to abelian schemes},
  series={Lecture Notes in Mathematics},
  volume={264},
  publisher={Springer-Verlag},
  place={Berlin},
  date={1972},
  pages={iii+190},
  review={\MR {0347836 (50 \#337)}},
}

\bib{mil03}{article}{
  label={Mi},
  author={Miller, Haynes R.},
  title={Sheaves, gradings, and the exact functor theorem},
  date={2003-09-14},
  note={Currently available at 
     \url {http://www-math.mit.edu/~hrm/papers/papers.html}},
}

\bib{nau07}{article}{
  author={Naumann, Niko},
  title={The stack of formal groups in stable homotopy theory},
  date={2007-02-14},
  note={\href{http://www.arxiv.org/abs/math/0503308v2}{\texttt{arXiv:math/0503308v2 [math.AT]}}},
}

\bib{prib04}{thesis}{
  author={Pribble, Ethan},
  title={Algebraic stacks for stable homotopy theory and the algebraic chromatic convergence theorem},
  date={2004},
  organization={Northwestern University},
  type={Ph.D. thesis},
  note={Currently available at \url {http://pribblee.googlepages.com}},
}

\bib{qui69}{article}{
  author={Quillen, Daniel},
  title={On the formal group laws of unoriented and complex cobordism theory},
  journal={Bull. Amer. Math. Soc.},
  volume={75},
  date={1969},
  pages={1293\ndash 1298},
  review={\MR {0253350 (40 \#6565)}},
}

\bib{rav86}{book}{
  author={Ravenel, Douglas C.},
  title={Complex cobordism and stable homotopy groups of spheres},
  series={Pure and Applied Mathematics},
  volume={121},
  publisher={Academic Press Inc.},
  place={Orlando, FL},
  date={1986},
  pages={xx+413},
  isbn={0-12-583430-6},
  isbn={0-12-583431-4},
  review={\MR {860042 (87j:55003)}},
}

\bib{rav92}{book}{
  author={Ravenel, Douglas C.},
  title={Nilpotence and periodicity in stable homotopy theory},
  series={Annals of Mathematics Studies},
  volume={128},
  note={Appendix C by Jeff Smith},
  publisher={Princeton University Press},
  place={Princeton, NJ},
  date={1992},
  pages={xiv+209},
  isbn={0-691-02572-X},
  review={\MR {1192553 (94b:55015)}},
}

\bib{simp97}{article}{
  label={Si},
  author={Simpson, Carlos},
  title={Limits in $n$-categories},
  date={1997-08-07},
  note={\href {http://www.arxiv.org/abs/alg-geom/9708010v1}{\texttt
           {arXiv:alg-geom/9708010v1}}},
}

\bib{sm07}{thesis}{
  label={Sm},
  author={Smithling, Brian},
  title={On the moduli stack of commutative, $1$-parameter formal Lie groups},
  date={2007},
  organization={University of Chicago},
  type={Ph.D. thesis},
  note={\href {http://www.arxiv.org/abs/0708.3326v2}{\texttt {arXiv:0708.3326v2 [math.AG]}}},
}

\bib{tate67}{article}{
  author={Tate, J. T.},
  title={$p$-divisible groups.},
  conference={ title={Proc. Conf. Local Fields}, address={Driebergen}, date={1966}, },
  book={ publisher={Springer}, place={Berlin}, },
  date={1967},
  pages={158--183},
  review={\MR {0231827 (38 \#155)}},
}


\end{biblist}
\end{bibdiv}
\end{document}